\documentclass{amsart}
\usepackage{amscd,graphicx}
\usepackage[latin1]{inputenc}
\numberwithin{equation}{section}
\makeatletter
\@namedef{subjclassname@2010}{%
  \textup{2010} Mathematics Subject Classification}

\newcommand{\teq}{\arabic{section}.\arabic{equation}}

\newcommand{\teql}{\Alph{section}.\arabic{equation}}

\newcommand{\sqr}[2]{{\vcenter{\vbox{\hrule height.#2pt\hbox{\vrule width.#2pt
height#1pt \kern#1pt\vrule width.#2pt}\hrule height.#2pt}}}}

\newcommand{\ssquare}{{\qquad\hfill$\square$}}

\newcounter{eqcount}
\newcounter{ttopic}

\renewcommand{\labelenumi}{{{\rm (\teq \alph{enumi})}}} 
\newenvironment{edesc}{\refstepcounter{equation}\begin{enumerate}}%
{\end{enumerate}}

\newcommand{\ring}[1]{{\mathbb #1}}
\newcommand\bZ{{\ring{Z}}}

\newcommand\bC{{\ring{C}}} \newcommand\bR{{\ring{R}}}
\newcommand\bF{{\ring{F}}} \newcommand\bQ{{\ring{Q}}}
\newcommand\bH{{\ring{H}}}
\newcommand{\csp}[1]{{\mathbb #1}}
\newcommand{\tsp}[1]{{\mathcal #1}}

\newcommand{\esp}[1]{{\mathcal #1}}

\newcommand{\prP}{\csp{P}}

\newcommand{\sA}{{\esp{A}}} 
 
\newcommand{\sO}{{\tsp{O}}} 
\newcommand{\sQ}{\tsp{Q}}
\newcommand{\sP}{{\tsp {P}}} 
\newcommand{\sL}{{\tsp {L}}} 
\newcommand{\sT}{{\tsp {T}}} \newcommand{\sH}{{\tsp {H}}}
 
\newcommand{\sM}{{\tsp {M}}} 
\newcommand{\sG}{{\tsp {G}}}

\newcommand{\bT}{{\csp {T}}} 

\newcommand{\eql}[2]{{\rm (\ref{#1}\ref{#2})}} 

\newcommand{\vect}[1]{{\pmb #1}} 
\newcommand{\ba}{{\vect{a}}} \newcommand{\bg}{\vect{g}}
\newcommand{\bd}{{\vect{d}}} 
\newcommand{\bp}{{\vect{p}}} \newcommand{\bx}{{\vect{x}}}
\newcommand{\bv}{{\vect{v}}} \newcommand{\bw}{{\vect{w}}}
 \newcommand{\bz}{{\vect{z}}}
\newcommand{\bh}{{\vect{h}}}  
\newcommand{\row}[2]{{#1_1,\ldots,#1_{#2}}}

\newcommand{\smatrix}[4]{{\big(\begin{array}{cc}
\!\lower2pt\hbox{$\scriptstyle#1$} &\lower2pt\hbox{$\scriptstyle#2$}\!
\\\! \raise2pt\hbox{$\scriptstyle#3$} &\raise2pt\hbox{$\scriptstyle#4$}
\!\end{array}\big)}}
\newcommand{\col}[2]{{\big(\begin{array}{c}
\!\lower2pt\hbox{$\scriptstyle#1$}  \!
\\\! \raise2pt\hbox{$\scriptstyle#2$}
\!\end{array}\big)}}

\newcommand{\texto}[1]{{\textr{#1}}}
\newcommand{\GL}{\texto{GL}} \newcommand{\SL}{\texto{SL}}
 
\newcommand{\PSL}{\texto{PSL}} \newcommand{\PGL}{\texto{PGL}}
 
 \renewcommand{\ni}{\texto{Ni}}
 \newcommand{\Pic}{\texto{Pic}}
\newcommand{\textr}[1]{{\text{\rm #1}}}
\newcommand{\tr}{\textr{tr}} \newcommand{\ord}{\textr{ord}}
\newcommand{\abs}{\textr{abs}}  
 
 \newcommand{\inn}{\textr{in}}
 \newcommand{\Aut}{\textr{Aut}}
\newcommand{\pr}{\textr{pr}}

\newcommand{\norm}{{\triangleleft\,}}

\newcommand{\BCL}{{\text{\rm BCL}}}

\newcommand{\rd}{\texto{rd}}

\newcommand{\tG}[1]{{}_{#1}\tilde G}

\newcommand{\sph}{{\vphantom 1}}

\newcommand{\textb}[1]{{\text{\bf #1}}}
\newcommand{\bfC}{{\textb{C}}}
\newcommand{\longmapright}[2]{\smash{\mathop{\hbox to
#2pt{\rightarrowfill}}\limits^{#1}}}
\newcommand{\Longmapright}[2]{\smash{\mathop{\hbox to
#2pt{\Rightarrowfill}}\limits^{#1}}}
\newcommand{\longmapleft}[2]{\smash{\mathop{\hbox to
#2pt{\leftarrowfill}}\limits^{#1}}}
\newcommand{\mapdown}[1]{\Big\downarrow\rlap{$\vcenter{\hbox{$\scriptstyle#1$}}
$}}

\newcommand{\np}{{+}}   \newcommand{\nm}{{-}}

\newcommand{\lrang}[1]{{\langle #1\rangle}}

\newcommand{\eqdef}{\stackrel{\text{\rm def}}{=}}

\newfont{\sevenrm}{cmr7}
\newfont{\bsevenrm}{cmbx7}
\newfont{\mathseven}{cmsy7}
\newfont{\bigmath}{cmsy10 scaled 1200}
\newfont{\fiverm}{cmr5}
\newfont{\bfiverm}{cmbx5}
\newfont{\hel}{cmbx10 scaled 1400}
\newfont{\eu}{eufb10}
\newfont{\sseu}{eufm5}
\newfont{\seu}{eufm7}
\newfont{\Cal}{cmmib10}
\newfont{\sCal}{cmmib7}
\newfont{\zch}{eusb10}

\theoremstyle{plain}
\newtheorem{thm}{Theorem}[section] 
\newtheorem{lem}[thm]{Lemma}
\newtheorem{princ}[thm]{Principle}

\newtheorem{prop}[thm]{Proposition}
\newtheorem{cor}[thm]{Corollary}

\theoremstyle{definition}
\newtheorem{defn}[thm]{Definition}
\newtheorem{exmp}[thm]{Example}
\newtheorem{guess}[thm]{Conjecture}

\newtheorem{prob}[thm]{Problem}

\theoremstyle{remark}
\newtheorem{rem}[thm]{Remark}
\newtheorem*{sol}{Solution}

\newcommand{\xs}{\times^s\!}
\newcommand{\wsp}{{$\,$---$\,$}} 

\def\pic #1 by #2 (#3){\vbox to #2{\hrule width 
#1 height 0pt depth 0pt\vfill\special{picture #3}}}
\def\scaledpicture#1
by #2 (#3 scaled #4){{\dimen0=#1 \dimen1=#2\divide\dimen0 by
 1000 \multiply\dimen0 by #4\divide\dimen1 by 1000 \multiply\dimen1 
by #4\pic \dimen0 by \dimen1 (#3 scaled #4)}}

\newcommand{\sJ}{{\text{\bf \sl J}}}
 
\newcommand{\comm}[1]{{}}
\newcommand{\Symm}{{\rm Symm}}

\newcommand{\sB}{{\tsp {B}}}
\newcommand{\bP}{{\tsp {P}}}

\newcommand{\Id}{{\text{\rm Id}}}

\renewcommand{\phi}{\varphi}
\newcommand{\Fr}{\text{Fr}}
\newcommand{\MT}{\text{\bf MT}}

\newcommand{\HIT}{\text{\bf HIT}}
\newcommand{\Hi}{\text{\bf Hi}}

\newcommand{\rest}{\text{\rm rest}}

\newcommand{\HM}{\text{\bf HM}}
\newcommand{\Cu}{\text{\rm Cu}}

\newcommand{\OIT}{\text{\bf OIT}}
\newcommand{\ST}{\text{\bf ST}}
\renewcommand{\BCL}{\text{\bf BCL}}

\newcommand{\Cen}{\text{\rm Cen}}

\newcommand{\pu}{\text{pu}}
\newcommand{\lcm}{\text{\rm lcm}}

\newcommand{\C}{{\text{\rm C}}}
\newcommand{\CM}{\text{\bf CM}}
\newcommand{\cmb}[1]{\bar #1}

\newcommand{\SM}{\text{\rm SM}}
\newcommand{\RH}{\text{\rm RH}}
\newcommand{\Inn}{\text{\rm Inn}}

\newcommand{\one}{{\pmb 1}}

\newcommand{\bm}{{\pmb m}}
\newcommand{\ab}{{{}_{\text{\rm ab}}}}
\newcommand\cO{{{}_\text{c}\sO}}
\newcommand{\geng}{{{\text{\bf g}}}}
\newcommand{\sh}{{{\text{\bf sh}}}}
\newcommand{\br}{{{\text{\rm br}}}}
\newcommand{\red}{{{\text{\rm rd}}}}

\newcommand{\by}{{\pmb y}}

\newenvironment{exmpl}{\begin{exmp}}{\hfill $\triangle$ \end{exmp}}

\newcommand{\Ct}[1]{{\bfC_{{3^{#1}}}}}

\newcommand{\mpr}{{\text{\bf mp}}} 
\newcommand{\wid}{{\text{\bf wd}}}
\newcommand{\g}{\text{\bf g}}

\newcommand{\fplus}[4]{{\hbox{\lower.3ex\hbox{$\scriptscriptstyle #1$}} \atop \hbox{\raise.3ex\hbox{$\scriptscriptstyle
#2$}}}\!\!\!\!+\!\!\!\!{\hbox{\lower.3ex\hbox{$\scriptscriptstyle #3$}} \atop \hbox{\raise.3ex\hbox{$\scriptscriptstyle #4$}}}}

\newcommand{\hplus}[2]{{{\scriptstyle #1}\hbox{\raise.2 ex\hbox{$\scriptscriptstyle|$}}{\scriptstyle #2}}}
  
\newcommand{\vplus}[2]{{\hbox{\lower.3ex\hbox{$\scriptscriptstyle #1$}} \atop \hbox{\raise.3ex\hbox{$\bar {\scriptscriptstyle #2}$}}}}

\newcommand{\DI}{\text{\bf DI}}

\begin{document}
\baselineskip=17pt
\hoffset.75in

\title[Hurwitz space components]{Hurwitz space components; and \\ the Coleman-Oort Conjecture}

\author[M.~D.~Fried]{Michael
D.~Fried}
\address{Emeritus, UC Irvine \\ 1106 W 171st Ave, Broomfield CO, 80023}
\email{michaeldavidfried@gmail.com, mfried@math.uci.edu} 

\begin{abstract}  Hurwitz spaces are moduli of isotopy classes of covers. A specific space, $\sH(G,\bfC)^\dagger$, with ${}^\dagger$ an equivalence relation, arises from a finite group $G$ and $\bfC$, $r$ of its conjugacy classes. Components, $\sH'$, of $\sH(G,\bfC)^\dagger$ interpret as a {\sl braid orbits\/} on {\sl Nielsen classes}, $\ni(G,\bfC)^\dagger$.   

A main theorem relates absolute ($\dagger={}^\abs$, using a permutation representation, $T$, of $G$) and inner ($\dagger={}^\inn$) equivalence classes. A normalizer subgroup from $T$ and the Schur multiplier of $G$ enter into component properties. 

Serre's Open Image Theorem  distinguishes two types of decomposition groups --  $\GL_2$ and $\CM$ -- on towers of modular curves, for $G$ related to dihedral groups. With mild constraints for any pair $(G,\bfC)$, we generalize modular curve towers to {\sl Modular Towers\/}. Applying properties of Jacobian varieties generalizes Serre's Theorem and connects Hilbert's Irreducibility theorem to the Coleman-Oort conjecture.

\end{abstract}  

\maketitle 
\setcounter{tocdepth}{3} 
\tableofcontents
\listoffigures

\section{Invariants separating moduli space components}  
\cite{FrV91}  and \cite{FrV92}  applied the comparison of absolute Hurwitz spaces $\sH(G,\bfC)^\abs$ and inner Hurwitz spaces $\sH(G,\bfC)^\inn$  to a general Inverse Galois Problem application. Here we consider components of these types under one umbrella using a definition in \cite{GoH92} (with more clarity in \cite{GhT24}) and so generalize these papers.  Absolute covers and their equivalences come from using a  permutation representation, $T$, of $G$. Two situations produce multiple components:
\begin{itemize} \item[1.] the action of a normalizer subgroup from $T$ on components; and 
\item[2.] distinct components from the Schur multiplier of $G$ (the Fried-Serre {\sl lift invariant\/}).
\end{itemize}

Our categories are (moduli) families of compact Riemann surfaces covering the Riemann sphere, $\prP^1_z$. We compare two papers, \cite{GoH92} and \cite{FrV91}, using \cite{GhT24}, that start with Galois covers but draw conclusions on more general families. The precise topic is connected components of two related families computed from this initial data:  $(G,\bfC, T)$, 
\begin{edesc} \label{indata} \item \label{indataa}  $G$ is the Galois closure group of covers of degree $n$ of a (faithful, transitive) permutation representation $T$, with 
\item \label{indatab} the covers having branch cycles in $r=r_\bfC$ conjugacy classes, $\bfC$, of $G$.\footnote{In most of our examples, $r\ge 4$ where there is a serious moduli space.} \end{edesc}  \S\ref{OS} gives notation to introduce the objects we study, components of Hurwitz spaces, and briefly goes through the examples we use to show the types of components that arise and how we detect them.  \S\ref{statements} describes the two layers of our main Theorem, based on homeomorphisms of covers of the projective line, $\prP^1_z$,  and how that puts structure in the different types of components that arise. Then, \S\ref{braiding} uses the braid group to construct the spaces and {\sl braid orbits\/} on Nielsen classes to distinguish the components.  \S\ref{totspaces}  reminds of the key tools for describing these components effectively -- the lift invariant and moduli definition fields -- allowing these spaces to support generalizing the Open Image Theorem. 

Serre's case, referred to as \OIT\ (Open Image Theorem) has $G=(\bZ/\ell)^2\xs \bZ/2$, $\ell\ne 2$, (Lem.~\ref{JacCond} uses it as dihedral related) with $\bfC$ four repetitions of the involution conjugacy class. In our notation, Serre's $\GL_2(\bZ_\ell)$ case is called \HIT\ (Hilbert's Irreducibility Theorem) because, assuming a certain property of the tower of spaces -- it is eventually $\ell$-Frattini, Def.~\ref{evenfratt}-- with a conclusion that is a precise version of what is expected from applying Hilbert's Theorem, with a conspicuous exception (called \CM, for complex multiplication), you get an open subgroup of the whole arithmetic monodromy group of the tower fibers. We concentrate on the production of the analog towers, called Modular Towers (\MT s) and the role of the {\sl lift invariant\/} for their existence and properties (beyond the use of that tool in \cite{FrV91}) using example groups $G$ for which our computations can be explained with basic linear algebra. 

Our first example is an addition to Serre's, showing the lift invariant appearing as a substitute for conclusions from the Weil pairing. Our other two examples have $G$ run (respectively) over alternating groups and $(\bZ/\ell)^2\xs \bZ/3$, $\ell\ne 3$. Both have serious literature precedents. Ex.~\ref{COvsAO} concludes with a statement to show how we use the Jacobians of curves occurring in Hurwitz spaces to form spaces, based on using braid action on Nielsen classes and the lift invariant, akin to those Serre used to see if his conclusion holds in far greater generality, reflecting on a range of problems far outside what would come from considering Siegel space and variants as Shimura did. 

\subsection{Objects of Study} \label{OS}  \S\ref{PN} gives the notation to display the spaces and components.   \S\ref{ES} summarizes the main properties of these objects, as in \S\ref{HSTowers}, which places Hurwitz spaces in towers comparable to modular curve towers. The examples section \S\ref{examples} shows the relation of these towers to properties of Jacobians (as in the Andr\'e-Oort conjecture), Weil's $\ell$-adic pairing, and  Serre's Open Image Theorem. Jacobian varieties are the semi-linear objects attached to curves. Here, we utilize them to interpret unsolved problems regarding families of covers of the Riemann sphere and their interrelationships.\footnote{\'Evariste Galois's death (1832) in approaching 200 years ago, shows how unlikely that someone will magically (and usefully) pluck solutions to the regular inverse Galois problem with some perspicacious trick. Better to limit its scope, keeping connection to significant problems -- Serre's \OIT, versions of Andr\'e-Oort, Complex Multiplication -- that reveal why the full problem has eluded serendipity.} 

\subsubsection{Preliminary Notation} \label{PN} 
Denote automorphisms of $G$ by $\Aut(G)$; those -- keeping multiplicity of appearance the same --- permuting classes of $\bfC$ by $\Aut(G,\bfC)$.  Automorphisms associated with \eqref{indata} are the subgroup of $\Aut(G,\bfC)$ of the normalizer, $N_{S_n}(G,\bfC)=K$,  in $S_n$ of $G$.  
	
\begin{lem} \label{transfaith} A transitive representation, $T$,  acts on the (we take right) cosets of the stabilizer, $G(T,1)$, of an integer in the representation: $g\mapsto$ effect of right multiplication on $G(T,1)g_i$ with each $g_i$ chosen to map 1 to $i$,  $i=1,\dots, n$. Then, $T$ is faithful if and only if $\cap G(T,i)=\{1_G \}$. \end{lem}

Denote the configuration space of $r$ distinct points in $\prP^r(\bC)$ by $U_r$. Our spaces are all moduli or $r$-branched covers of the projective line, $\prP^1_z$, uniformized by the standard complex variable $z$, and they will naturally map to either $U_r$, or for {\sl reduced\/} Hurwitz spaces to $U_r/\PSL_2(\bC)$ with $\PSL_2(\bC)$ M\"obius transformations. The distinction doesn't change the description of components since $\PSL_2(\bC)$ is connected. Denote the normalizer of $G$ in $S_n$ by $N_{S_n}(G)$ and $N_{S_n}(G)\cap \Aut(G,\bfC)$ by $N_{S_n}(G,\bfC)$. For $T$ the {\sl regular representation\/}, then  $N_{S_n}(G,\bfC)=\Aut(G,\bfC)$, but that is rarely our best choice of $T$ (there may  be several). 

Here is how the pairs arise.  The first space is $\sH(G,\bfC)^K\eqdef \sH(G,\bfC)^\abs$, with $K=N_{S_n}(G,\bfC)$, the space of $\deg(T)$ covers, up to the usual equivalence (called absolute).  The second space is $\sH(G,\bfC)^\inn$: Galois closures of covers in $\sH(G,\bfC)^K$, modulo inner equivalence.\footnote{Algebraic number theory assumes that all field extensions occur inside a fixed algebraic closure of the base field $F$. Therefore, the Galois closure of an extension of $F$ in that field is well-defined. For several reasons, that is not a valid assumption. So,  \S\ref{fibGalClos} considers the fiber product construction of the Galois closure.} This uses the Hurwitz space version of the fiber product construction of Galois closures of covers. 
Thm.~\ref{cosetbr} sets up the dichotomy from using $T$ based on this {\sl Galois Closure Principle}: 
\begin{edesc} \label{gcp} \item  \label{gcpa} Components of $\sH(G,\bfC)^K$ are  homeomorphism-separated;  and
\item   \label{gcpb} components of $\sH^\inn$ above a given $\sH(G,\bfC)^K$ component are automorphism-separated. \end{edesc} 

Example: \eql{gcp}{gcpb} says, if $\sH_j^\inn\to U_r$, $j=1,2$, are components from braid orbits on $\ni(G,\bfC)^\inn$, lying above the same component, $\sH'$, of $\sH^K$, then their braid orbits (in $\ni(G,\bfC)^\inn$) differ by a non-braidable $\alpha\in K=N_{S_n}(G,\bfC)$. From  Cor.~\ref{autocomps},   $\sH_j\to \sH'$, $j=1,2$, are equivalent as covers, though they support different families of Galois covers of $\prP^1_z$.   

We usually assume $T$ is understood. {\sl Nielsen classes\/} (Def.~\ref{HNC}) associated to each of these two types, respectively $\ni(G,\bfC)^\abs$ and $\ni(G,\bfC)^\inn$,  allow making computations of their properties. The covers in each family have a genus -- with resp.~notation like $\geng_\abs$ or $\geng_\inn$ -- computed from Riemann-Hurwitz. Don't confuse this, when $r=4$, with the genus \eqref{RHOrbitEq} of the reduced Hurwitz space (a nonsingular projective curve) attached to each space. 

We use the following notation for these families: 
\begin{edesc} \label{hurnot} \item \label{hurnota} $\sH(G,\bfC,T)\eqdef \sH(G,\bfC)^\abs\eqdef \sH(G,\bfC,T)^{N_T}$, meaning, equivalence these $\deg(T)$ covers when their branch cycles differ by the action of $N_{S_n}(G,\bfC)$;  and 
\item \label{hurnotb} $\sH(G,\bfC)^\inn$, the family of covers given by taking the Galois closure of the covers in \eql{hurnot}{hurnota}, modulo conjugation by $G$ (inner equivalence).\end{edesc}  

\S\ref{statements} describes the classification of components and states brief versions of the paper's results about them. \S\ref{isotopy-Hur} provides the tools for studying the properties of Hurwitz spaces.  With a given permutation representation $T$ of $G$, Thm.~\ref{cosetbr} divides consideration of components into two steps: first listing components of the absolute space (homeomorphism-separated), and then organizing components of the inner Hurwitz space above a given absolute component (automorphism-separated). Thus, this part improves on \cite{GoH92}, \cite{FrV91} and \cite{GhT24}. 

\subsubsection{Serre's Case and our examples} \label{ES} 
Our examples follow a pattern of generalizing Serre's case. We refer to Serre's case as the Open)I(mage)T(heorem) (or \OIT).  That started by looking at modular curves as Hurwitz spaces \cite[Introduction]{Fr95}.  Roughly speaking, the generalization, based on the notation $(G,\bfC)$, goes from $G$ related to dihedral groups and $\bfC$ four repetitions of the involution conjugacy class -- producing sequences of modular curves -- to where $G$ is a general finite group and $\bfC$ is chosen to ensure the production of non-trivial spaces.   

Serre's program for  modular curve towers $\{X(\ell^{k\np1})\}_{k\ge 0}$ compared these groups:
\begin{edesc} \label{projlim} \item \label{projlima} the projective limit of decomposition groups of a projective sequence of points above a particular $j_0\in \prP^1_j$ (the $j$-line) with;  
\item \label{projlimb} the projective sequence of monodromy groups, arithmetic and geometric (esp.~$\GL_2(\bZ_\ell)$ and $\SL_2(\bZ_\ell)$) of the components over the $j$-line. \end{edesc} 

\S\ref{homeoa} has the essential definitions we use repeatedly for group covers. One is especially important, allowing constructing the towers of spaces generalizing those used by Serre in his \OIT:  \begin{defn} \label{deffrat} A profinite cover $\psi: H\to G$  is {\sl Frattini\/} if, for any $H^*\le H$ with $\psi(H^*)=G$, then $H^*=H$. It is central (resp.~$\ell$-Frattini) if $\ker(\psi)$ is in the center of $H$ (resp.~an $\ell$ group), etc.
\end{defn}  

\S\ref{bkliftinv} applies the universal abelianized $\ell$-Frattini cover of $G$ to form the spaces that generalize the framework for Serre's \OIT. The existence of a nonempty sequence of irreducible components of the spaces at level $k\ge 0$ has one potential obstruction. The check for its vanishing is our most sophisticated use of the lift invariant. By applying T. Weigel's generalization, Thm.~\ref{WeigelThm}, of Serre's use of an $\ell$-Poincar\'e duality group, we give an if and only if criterion for this. This includes that there is no obstruction whenever the $\ell$ part of the Schur multiplier of $G$ is trivial. 

 Prop.~\ref{HIT2} connects to \HIT\ by giving the criterion that, general decomposition groups on a \MT  are open subgroups  of the \MT\ imonodromy when it is {\sl eventually $\ell$-Frattini} (Def.\ref{evenfratt}).  

\S\ref{HSJac} returns to Serre's case to interpret, with Jacobians of the level 0 curves, how to compare extension of constants and the moduli definition field of a \MT. We remind of Shimura's generalization of complex multiplication points to consider -- comparing with \HIT -- how to distinguish level 0 points of a \MT\ with radical differences between their corresponding decomposition groups. 

\S\ref{cmabelvar} summarizes the Shimura-Taniyama notation of \CM\ (or \ST) points on Siegel space, emphasizing this is about the corresponding abelian variety. Our main comparison is with the Coleman-Oort Conjecture, since our questions concern the Jacobians associated with the curve covers on a Hurwitz space.  Many Hurwitz spaces include as covers almost every curve of genus $\geng$, for example \cite[Thm.~6.15]{Fr10}  with Nielsen classes of odd order branching and the corresponding questions about nontrivial $\theta$-nulls and their connection to Hilbert's original paper on \HIT. 

\S\ref{weilpairing} warms up using the Fried-Serre lift invariant (\S\ref{l'lift}), applying the Hurwitz space interpretation to relate to the Weil pairing, and the moduli definition field. The two \OIT\ cases: 
\begin{edesc} \label{oit} \item \label{oita} CM: $j_0$ is a complex multiplication point; and the decomposition group, an open subroup of $\hat \bZ_\ell$, identifies as the group of the maximal abelian $\ell$-adic extension of $\bQ(j_0)$; and \item \label{oitb} $\GL_2$: In the Hurwitz space interpretation, an open subgroup of $\GL_2(\bZ_\ell)$. \end{edesc}  \cite[\S 2]{Fr78} took the case $G=\bZ/\ell^{k\np1}\xs \bZ/2$ ($\ell\ne 2$), a dihedral group and $\bfC=\bfC_{2^4}$ four repetitions of the involution class.  This recasts Serre's \CM\  case as generalizing a famous conjecture of Schur from its statement about polynomials to rational functions.\footnote{Describing prime-squared degree exceptional rational functions is equivalent to Serre's $\GL_2$-case as in \cite[\S6.1Ð6.3]{Fr05b}, which also documents the result of \cite{GSM03}: All other degrees of indecomposable exceptional rational functions are sporadic (fall in finitely many Nielsen classes).}

Then, \S\ref{absinnAn} with $G=A_n$ and $\bfC$ consisting of odd order conjugacy classes engages (with elements of collaboration with Serre) has results that tie together a sizable literature. \S\ref{Anabs} gives collections where the Lift Hypothesis holds \eqref{schursep}, and when, if it doesn't, to producing situations -- called pure-cycle -- to generalize the result on irreducible components first produced by \cite{LO08} for which I use an interpretation of \cite{Se90} (or \cite[\S 2.2]{Fr10}).  \S\ref{absinnAn} has this special case: 
\begin{thm}  \label{LO08Se90Fr10case} With $G=A_n$, $n\ge 4$, $T$ the standard degree $n$ representation and $\sH(G,\bfC)^\abs$ is any genus 0 Nielsen class with $\bfC$ any $2'$ classes, $\sH(G,\bfC)^\abs$ has precisely one component.  \end{thm}  

\S\ref{heiscase} is our major example with $G=(\bZ/\ell)^2\xs \bZ/3$. Notationally, it resembles \S\ref{weilpairing} (which has $G=(\bZ/\ell)^2\xs \bZ/2$) but it goes beyond the \OIT.  It illustrates all aspects of Thm.~\ref{cosetbr}, including computing the lift invariant explicitly. 

\S\ref{totspaces}  starts the arithmetic of the Galois closure process applied to covers and their moduli. While \cite{FrV91} used the lift invariant to delineate components of Hurwitz spaces given by the parameters $(G,\bfC)$, we assumed the multiplicity of the classes appearing in $\bfC$ is large. That would do nothing for generalizing the \OIT. \cite{BFr02} developed Modular Towers (\MT, the projective sequence of spaces generalizing modular curve towers) beyond \cite{Fr95} and showed how it applied to  $G=A_n$ for $n=5$ and four repetitions of 3-cycle classes. 

\subsubsection{Three uses of the lift invariant} \label{highmultsec}  The first use of the lift invariant is the division of Thm.~\ref{cosetbr} into two levels of component types: absolute spaces and, above these,  inner spaces based on taking Galois closures. The second use is to form the towers (\MT s, Def.~\ref{def-MT}) of inner moduli spaces of curves, that generalize Serre's use of modular curves. Third: Sometimes the lift-invariant helps us determine the moduli-defining field of inner-space components.

The example of \S\ref{heiscase}  displays all three of these lift-invariant uses. This allows comparing expectations with formulations of others (Rem.~\ref {AOB}) based on the Siegel upper half-space and complex multiplication.\footnote{\cite{Fr10} shows I have nothing against Siegel space, but curves and their arithmetic are the tougher nonlinear case for which Jacobian varieties are an aid.} 

\begin{defn} \label{highmult} By increasing the multiplicity of {\sl each\/}  conjugacy class in $\bfC$ -- refer to this as {\sl high multiplicity\/} -- \eql{liftinvuse}{liftinvuseb}  shows the configuration of components in Thm.~\ref{cosetbr} simplifies.\footnote{The Ex.~\ref{HM-separated} result is explicit on {\sl high multiplicity}. To keep the result of applying the \BCL\ Thm.~\ref{bcl} the same, increase the multiplicity of classes in $\bfC$ so the cyclotomic action on the new $\bfC$ doesn't change.}  \end{defn} Our examples have $r=r_\bfC=4$, so high multiplicity doesn't hold. Even in the most intricate cases, the structure of Thm.~\ref{bcl} clearly displays the components, separating out the most serious arithmetic and identifying the moduli definition fields of \HM\ components. 

Def.~\ref{liftinv} gives the formula for the lift invariant, $\hat \bg\in \ni(G,\bfC)^\inn\mapsto s_{\hat\bg}$.  Our examples satisfy $(\ell, N_{\bfC})=1$. Then,  $s_{\hat \bg}$  is always an element in the $\ell$ part, $\SM_{G,\ell}$,  of the Schur multiplier of $G$. It's a braid invariant, constant on any braid orbit. We give an explicit formula for it in our examples. 

There is a natural action of $N_T/G$ (Def.~\ref{NormactonLiftInv}) on the lift invariants attached to the components of $\sH(G,\bfC)^\inn$ lying over a component $\sH'\le \sH(G,\bfC)^\abs$. Property \eql{liftinvuse}{liftinvusea}, follows from  Main Thm.~\ref{cosetbr}.    
\begin{defn} \label{NormactonLiftInv} With $\sH'$ corresponding to the braid orbit of $\bg\in \ni(G,\bfC)^\abs$ and $\hat \bg\in \ni(G,\bfC)^\inn$ lying over $\bg$,  $\alpha\in N_T/G: s_{\hat \bg} \to s_{\hat \bg^\alpha}$. \end{defn}

\begin{edesc} \label{liftinvuse}  \item \label{liftinvusea}  The components of $\sH(G,\bfC)^\inn$ lying over a component $\sH'\le \sH(G,\bfC)^\abs$, correspond to elements of an orbit of $N_T/G$ on $s_{\hat \bg}$. 

 \item \label{liftinvuseb} With high multiplicity, each $s'\in \SM_{G,\ell}$ will have the form $s_{\hat \bg}$ for some $\hat \bg\in \ni(G,\bfC)^\inn$ and components of $\sH(G,\bfC)^\abs$ correspond one-one to orbits of $N_T/G$ on $\SM_{G,\ell}$. \end{edesc}
 
 Comments on \eqref{liftinvuse}: 
 \begin{edesc} \label{liftinvusecom} 
\item The description of components in \eql{liftinvuse}{liftinvusea} is independent of the representative $\bg$. 
\item Rem.~\ref{notlperf} give the formula for $s_{\hat \bg}$ without the $(\ell, N_{\bfC})=1$ assumption, adding only a slight complication to Def.~\ref{NormactonLiftInv}. 
\item For $N_T(G)/G$ an $\ell'$ group, its action can be expressed in a less mysterious form (Cor.~\ref{compmult2}) using the universal abelianized $\ell$-Frattini cover of $G$. \item \label{liftinvused}  \eql{liftinvuse}{liftinvusea} was started in \cite{FrV91}, assuming High Multiplicity in $\bfC$; it didn't show how that affected the two-sequence result of  Thm.~\ref{cosetbr}. \end{edesc}  

The production of the Schur multiplier at all levels of the \MT\ and the explicit computation of the lift invariant (as was done in the Alternating group case above at level $k=0$) allows comparing with the \OIT\ case. Example \S\ref{heiscase} has Hurwitz spaces $\sH((\bZ/\ell^{k\np1})^2\xs \bZ/3, \bfC_{\pm 3})$ and as with Serre's case, we eventually go to reduced Hurwitz spaces by modding out by $\PSL_2(\bC)$. \S\ref{nonbraidcomps} shows the superficial resemblance of this to Serre's case, but here, finding projective sequences of components must deal with potentially obstructed components, coming from the lift invariant, to ensure the possibility of taking projective sequences of points. 

Thm.~\ref{WeigelThm}, Weigel's generalization of Serre's oriented $ p$-Poincar\'e duality group, handles this, except here we have an extension, $\sL\to \hat G_\ell\to G$, of $G$ by an $\ell$-adic lattice, $\sL$ defined in \S\ref{l'lift}.  This gives a sequence of Frattini covers with abelian $\ell$-group kernels $G_{k\np 1} \to G_k$, $k\ge 0$, $G_0=G$, and given our conjugacy classes a tower of Hurwitz spaces $\{\sH(G_k,\bfC)\}_{k=0}$. 
The topic of obstructed components and the construction of \MT s  first arose in \cite[Obst.~Comp.~Lem. 3.2]{FrK97} to give an if and only if criterion that all tower levels are nonempty.  Princ.~\ref{liftinvprinc} gives the main theorem -- a lift invariant criterion -- for the existence of an abelianized \MT\ through a specific component at a specific level, which requires only a check at level 0.  

\begin{princ} \label{liftinvprinc} There exists $k_0$ with $\psi_{k_0}: G_{k_0}\to G$, the $\ell$-Frattini cover above, factors through an $\ell$-reptresentation cover  $H\to G$. Then, the spaces above form a non-empty \MT\  over a component corresponding to a braid orbit in $O\le  \ni(G,\bfC)$ if and only if there is $\bg_{k_0} \in \ni(G_{k_0},\bfC)$ over $\bg\in O$. 

This obstruction interprets as saying, in generalization to the lift inv. notation above, that $s_{H,\bg_{k_0}}=0$. In particular, this holds if the $\ell$ part of the Schur multiplier of $G$ is trivial. \end{princ} 

 As with \eqref{oit}, generalizing what arose in Serre's case (especially the idea of an {\sl eventually Frattini\/} projective sequence of finite groups), allows generalizing Hilbert's Irreducibility Theorem. The first result, Thm.~\ref{HIT2}, describes when, for general points on a \MT,  the analog of \eql{projlim}{projlima} is an open subgroup of the analog of \eql{projlim}{projlimb}.

\subsection{Results and homeomorphisms of covers} \label{statements}  
 \S\ref{comptypes} emphasizes Thm.~\ref{cosetbr},  in terms of what we know about component types. It displays how the list \eqref{liftinvuse} (with corresponding comments \eqref{liftinvusecom}) works based on the natural map from components of  $\sH(G,\bfC)^\inn$ to those of $\sH(G,\bfC)^\abs$ in \eqref{hurnot}.  \S\ref{cenprob} defines the {\sl moduli definition field\/} and the key problems on components: the $G_\bQ$ action on them, and finding the correct field over which a point on a component represents a cover. 
 
 \begin{rem}[Warning!] As Ex.~\ref{davprs} -- from solving Davenport's problem\footnote{That should set straight any misunderstanding that definition fields for all reasonable moduli spaces are $\bQ$.} -- shows 
 the moduli definition field, in general,  is a proper extension of the definition field of the moduli space component with its map to the configuration space. \end{rem} 
 
\subsubsection{Types of components} \label{comptypes} Components correspond to braid orbits on a {\sl Nielsen class\/} (Def.~\ref{HNC}).  Improving the main result of \cite{GoH92} and \cite{GhT24}, they distinguish Nielsen class components. Suppose $\sH_i$, $i=1,2$, are inner space components, both over the same absolute component, $\sH'$. 
\begin{edesc} \label{nonbrauto} \item  Then, each cover in $\sH_1$ is homeomorphic (Def.~\ref{homeogal}) to a cover in $\sH_2$, the homeomorphism commuting  between the covering maps to $\prP^1_z$ inducing $\alpha\in \Aut(G,\bfC)$, but  
\item $\alpha$ is {\sl non-braidable\/} (Def.~\ref{braidable}). \end{edesc}  

\S\ref{isotopy-Hur} reminds us of isotopy classes of covers and how to compute components and their properties using an explicit quotient of the braid group. Suppose $\sH'$ and $\sH^{''}$ are distinct components of $\sH(G,\bfC)^\abs$. We call them homeomorphism-separated. We don't yet know exactly {\sl what\/} distinguishes homeomorphism-separated components,  yet most \S\ref{examples} examples of homeomorphism-separated components, $\sH'$, of $\sH(G,\bfC)^\abs$ have this {\sl Schur-separation\/} property using  the collection, $S_{\sH'}$, of lift invariants of inner components above $\sH'$ (Def.~\ref{NormactonLiftInv} or Def.~\ref{schur-sep}): 
\begin{equation} \label{schursep} \text{ $S_{\sH'}$ determines $\sH'$ uniquely.}\end{equation} 
Exceptions often have multiple Harbater-Mumford (Def.~\ref{HM}, lift invariant 0) components.

\subsubsection{Moduli definition problem} \label{cenprob}  Denote the least common multiple of the order of elements in $\bfC$ by $N_\bfC$. Given $\sigma\in G_\bQ$, its restriction to the cyclotomic numbers gives $n_\sigma\in (\bZ/N_{\bfC})^*$ (Def.~\ref{ratclass}). Given $\psi: X\to \prP^1_z$ representing  $\bp\in \sH(G,\bfC)^\dagger(\bar \bQ)$, $\dagger=\inn$ or $\abs$, denote its conjugate by applying $\sigma$ by $\psi^\sigma$. Here is the first corollary of the {\sl Branch Cycle Lemma\/} \S\ref{bclmodel} (\BCL\ of \cite{Fr77}). 
  
\begin{cor}  \label{GQaction} Then, $\psi^\sigma$ is  a representative of $\bp^\sigma\in \sH(G,\bfC^{n_\sigma})^\dagger(\bar \bQ)$.  \end{cor} The \BCL\ gives much more: For example, under the assumption that \begin{equation} \label{BCLassump} \text{$\sH(G,\bfC)^\dagger$ is irreducible and has fine moduli,}\end{equation}  it gives the precise (minimal) {\sl cyclotomic field}, $\bQ_{\sH^\dagger}$, for which  $\bp\in\sH(G,\bfC)^\dagger$ has a representing cover over $ \bQ_{\sH^\dagger}(\bp)$. Assuming \eqref{BCLassump}, this makes $\bQ_{\sH^\dagger}$ the  {\sl moduli definition field\/} (Def.~\ref{mdf}) of the Hurwitz space.  When the Hurwitz space has more than one component, we consider the moduli definition field, $\bQ_{\sH^\dagger}$, for a component $\sH^\dagger$.  

That  definition uses the {\sl total space\/} over $\sH^\dagger$:  $\sT^\dagger\to \sH^\dagger\times \prP^1_z$, with fibers, $\sT^\dagger_\bp\to \bp\times \prP^1_z$ for $\bp\in \sH^\dagger$,  representing covers. (We also use the reduced space version.) 
\S\ref{mdfsec}  does better -- a reason for choosing $T$ carefully when possible -- effectively generalizing the \BCL\  assuming:    
\begin{equation} \label{gcpadd}  \text{we know $\bQ_{\sH'}$ ($\sH'\le \sH(G,\bfC)^K$); and  \eqref{BCLassump} holds for $\sH^*\le \sH(G,\bfC)^\inn$ above $\sH'$.}\end{equation}  

A general result for Schur-separated absolute components \eqref{schursep} with {\sl cyclic\/} (or trivial) Schur multiplier gives the moduli definition field that suffices for the \S\ref{examples} examples. That excludes the case of multiple \HM\ components in \S\ref{heiscase}.  Going beyond condition \eqref{gcpadd} is under the heading of {\sl extension of constants\/},  starting in \eqref{extconstants} and taking off in \S\ref{fibGalClos}. This abstracts the central mystery in using {\sl Hilbert's Irreducibility Theorem\/}, generalizing how \cite{Fr78} viewed \cite{Se68}.  

\begin{prob} \label{unirationality}  Unirationality question:  In the cases \cite{GoH92} and \cite{GhT24} give, where the spaces equivalence {\sl  all\/} covers if they are conjugate by $\Aut(G)$, are the moduli spaces unirational?\end{prob} By computing some genuses of reduced spaces when $r=4$, we show the answer is \lq\lq No!\rq\rq\ These examples illustrate our main Thm.~\ref{cosetbr} on components and give the significance of finding $G_\bQ$ orbits and -- more strongly --  moduli definition fields. 

 \S\ref{absinnAn} with $G$ an alternating group, generalizes results of Fried, Liu-Osserman, and Serre. Computing moduli definition fields for components reverts to finding an easily stated property of discriminants of genus 0 covers over $\bQ$.   \S\ref{heiscase} is our main case for the full force of Thm.~\ref{cosetbr}  to handle the configuration of components and their moduli definition fields. It has $G=(\bZ/\ell^{k\np1})^2\xs \bZ/3$, $\ell$ a prime, $k\ge 0$ as an example extending Serre's Open Image Theorem (\OIT). This is a case of \MT s developed to handle the simplest unanswered example for any $\ell$-perfect group $G$: 
\begin{equation} \label{IGPstatement} \begin{array}{c} \text{Assuming the regular inverse Galois is correct (say, over $\bQ$),}\\  \text{where are the regular realizations of $\ell$-Frattini covers of $G$?}\end{array}  \end{equation} If the main conjecture for \MT s is correct -- Rem.~\ref{A5MT} reminds of evidence for it as a generalizaiton generalization of Faltings Theorem -- then, the appearance of those regular realizations requires rational points on a sequence of Hurwitz spaces of  unbounded dimension.\footnote{That applies to the case $G$ is a dihedral group, putting generalizations of Mazur's modular curve result as a particular case \cite{DFr94}.}   

Our examples use spaces of four branch point covers whose reduced versions are (therefore) upper half-plane quotients \cite[\S 2.10]{BFr02}.  Though these are rarely modular curves, we can still explicitly compute their genuses (answering a question in \cite{GhT24} negatively). The main computation is computing the orders of cusps (points over $j=\infty$). This gives one tool for verifying that they aren't modular curves. Proposition~\ref {A4L0} makes that computation in a particular case.

\cite{GhT24}  wanted total spaces.\footnote{Even with $G$ abelian (so fine moduli doesn't hold).} Their approach differed from \cite{FrV91}, and their total spaces may have several repeats of the same cover.  \S\ref{withoutfinemod} contrasts this with the {\sl Grothendieck nonabelian cohomology\/} approach of \cite{Fr77}. 

\begin{rem}[Main \MT\ Conj.] \label{A5MT} \cite{BFr02} proved the Main \MT\ conjecture -- high tower levels have no rational points --  for the \MT\ with $n=5$, $r=4$, $\ell=2$, of Ex.~\ref{Anexmp}. That explicitly showed, by level 2, the genus of the reduced components  -- using a version of \eqref{RHOrbitEq} -- exceeds 1.  Applying Faltings' inner Hurwitz space tower levels $k\ge 2$  have only finitely many rational points over a fixed number field, $F$.  Rem.~\ref{twotychonof}  now gives this case of the Main \MT\ conjecture. \end{rem}

\begin{rem}[Other uses of the lift invariant\footnote{Uses of the Tychonoff Theorem came together in different papers at different times.}]  \label{twotychonof} The conclusion (for $r=4$ over a number field $F$) of Rem.~\ref{A5MT} used Weil's Theorem on the Frobenius action and a reduction theorem of Grothendieck, Falting's Theorem and the Tychonoff Theorem to show a \MT, with reduced Hurwitz space components of genus $> 1$, could have $F$ points off the cusps at only finitely many levels. 

Otherwise, they would produce an $\ell$-adic representation on the Jacobian of a particular cover in the Nielsen class over a finite field, with trivial $G_F$ action.  The Falting's part is not explicit, but the level of the high genus result is.  The hardest case of the Main \MT\ Conj.~(for any $r$) is when there is a uniform bound on the moduli definition field of the tower levels. Ex.~\ref{HM-separated} has examples of explicit $(G,\bfC,r=r_{\bfC})$ with $r_\bfC>4$ for which this holds.  \end{rem}

\subsection{Isotopies and braidable automorphisms}  \label{braiding} \S\ref{3tools} explains three main tools: 
 \begin{edesc} \label{tools} \item \label{toolsa} using pairs of related cover types described by corresponding {\sl Nielsen classes}; 
 \item \label{toolsb}  recognizing homeomorphic covers that differ by nonbraidable automorphisms; and 
 \item \label{toolsc}  classifying covers that aren't homeomorphic, though they are in the same Nielsen class. 
\end{edesc} 
Our model for \eql{tools}{toolsa} comes from classical pairs of modular curves. Using it, Thm.~\ref{cosetbr} effectively separates components of type \eql{tools}{toolsb} -- covers in different components might be homeomorphic, but differ by a non-braidable automorphism -- from those of type  \eql{tools}{toolsc}.\label{isotopy-Hur} 

\S\ref{dragbybps} defines isotopy of covers using \lq\lq dragging a cover by its branch points,\rq\rq\ and so the Hurwitz monodromy group, $H_r$. With this, we can compute the components of a natural space of such covers using Nielsen classes. In $(\prP^1_z)^r$,  the {\sl fat diagonal\/}, $\Delta_r$, consists of points with two or more coordinates equal. Denote the quotient result, $\Delta_r/S_n$, on coordinates by $D_r$. This sits inside  $(\prP^1_z)^r/S_r=\prP^r$, projective $r$-space. 

The collection of possible, unordered, and distinct branch points for an $r$-branched cover of $\prP^1_z$ is given by $U_r\eqdef \prP^r\setminus D_r$. Consider $U_\bz\eqdef \prP^1_z\setminus \{\bz\}$. For $z'$ distinct from any of the coordinates of $\bz$, form  $ \pi_1(U_{\bz})^\inn$ by modding out by inner automorphisms on $\pi_1(U_{\bz},z')$. \S\ref{draggivesMainT} gives Main Thm.~\ref{cosetbr}, dividing Hurwitz space components into a heirarchy of types. 

\subsubsection{Tools}  \label{3tools} The symbol $\prP^1$ denotes the Riemann sphere.  (Nonsingular, ramified) covers of it here are compact Riemann surfaces $X$ with a nonconstant morphism $\phi: X\to \prP^1$. Until we get to examples and comparison with classical constructions, we use the notation $\prP^1_z$ (and its like) to mean $z$ is an explicit (inhomogeneous) uniformizing variable (as in 1st-year complex variables). 

\S\ref{classgens}  describes {\sl classical generators\/} $\sP$ of the (fundamental group of the) $r$-punctured sphere,  
$\pi_1(U_{\bz},z_0)$ with the punctures at $\bz=\row z r$, and $U_\bz\eqdef \prP^1_z\setminus \{\bz\}$ and $z_0$ distinct from any entries of $\bz$. Given $(\sP,z_0)$, a cover $\phi$ -- with a fixed naming of the points, $\phi^{-1}(z_0)$, above $z_0$ -- with branch points $\bz$ is analyticially  determined by the branch cycles $\bg$ computed from $(\sP,z_0)$. 

\eqref{hurnot} references covers as given by branch cycles and absolute and inner equivalences of covers using branch cycles.   
Given any such $\phi$ by its branch cycles $\bg$, elements in $S_n$, with $n=\deg(T)$, we can always reference the Galois closure, $\hat \phi: \hat X\to \prP^1_z$ which has group $G=\lrang{\bg}$. Several possible branch cycles, $\hat \bg$, associated to $\hat \phi$, differ by actions of $N_T$ fixed on $\bg$.  \S\ref{fibGalClos} reminds of our construction, including for families of covers. \"Mobius transformations of $\prP^1_z$ act on such covers:\footnote{Our examples will illustrate the equivalences on branch cycles from applying this action.}  \begin{equation} \label{PSLact} \beta\in \PSL_2(\bC): \phi\to \beta\circ \phi. \text{ This action on spaces of covers forms their {\sl reduced versions}.} \end{equation}   

\cite{GoH92} starts with a pair, $(\hat  X_1,G)$, 
\begin{edesc} \label{startpt} \item  \label{startpta} $\hat \phi_1: \hat X_1\to \prP^1_z$, a Galois cover with group $G$, and then considers 
\item  \label{startptb} all homeomorphic Galois covers, $\hat X\to \prP^1_z$, by $\hat \theta: \hat X_1\to \hat X_2$ (Def.~\ref{homeogal}) with group $G$. 
\end{edesc}  

\begin{defn} \label{homeogal} For covers, $\phi_i: X_i\to \prP^1_z$, $i=1,2$,  a homeomorphism $\theta$ between them is a homeomorphism $\theta:  X_1\to X_2$ that {\sl preserves fibers\/}: maps a fiber $\phi_1^{-1}(z_1)$ of $\phi_1$ to a fiber of $\phi_2$. So, it is also a homeomorphism on $\prP^1_z$. We say the {\sl covers\/} are {\sl homeomorphic\/}. \end{defn} 

By contrast, \cite{FrV91} starts with a group $G$ and $\bfC=\{\row \C r\}$, a collection of conjugacy classes in $G$. Then, it has two related approaches. 
\begin{edesc} \label{FrVstartpt} \item \label{FrVstartpta}  Consider all Galois covers, $\hat \phi: \hat X\to \prP^1_z$, with group $G$, having branch cycles, \\ $\bg=(\row g r)$, for the cover in the classes $\bfC$ (with the same multiplicity).  
\item \label{FrVstartptb} For a given (usually faithful and transitive) permutation representation $T:G\to S_n$,  consider all covers $\phi: X \to \prP^1_z$ with Galois closures given by \eql{FrVstartpt}{FrVstartpta}. 
\end{edesc} 

In both cases of \eqref{FrVstartpt}, we say $\bg\in \bfC$.  Def.~\ref{HNC} has mandatory {\sl product-one\/} and {\sl generation\/} conditions for elements of $\bg$. This defines {\sl Nielsen classes}, $\ni(G,\bfC)$, with which we can be explicit about these objects and refer to $\bg \in \ni(G,\bfC)$: 
\begin{edesc} \item equivalences of covers;  
\item connected components of families of those covers up to one of those equivalences;
\item a braid group action for computing those components; and properties of covers by which we can recognize those components.\end{edesc}  

Applications rarely require naming points in $\phi^{-1}(z_0)$. Equivalences change this naming, starting with equivalencing $\bg$ and $h\bg h^{-1}\in \ni(G,\bfC)$, for $h\in G$: they differ by inner automorphisms, $\Inn(G)$, of $G$. Denote $\pi_1(U_{\bz_0},z_0)$ 
 mod inner automorphisms by $\pi_1(U_{\bz_0})$. 
\begin{edesc} \label{modbyinner} \item \label{modbyinnera}   Inner equivalence  for covers of $\prP^1_z$ relative to a given set of classical generators, $\sP$, around $\bz_0$ implies a representation $\pi_1(U_{\bz_0},z_0)$ for inner equivalence factors through $\pi_1(U_{\bz_0})$.
\item  \label{modbyinnerb} We can always braid inner automorphisms \cite[Lem.~3.8]{BiFr82}. Using such an equivalence class doesn't change finding the components we are after.\footnote{Indeed, not using inner equivalence would make many applications untenable.}  
 \end{edesc} 

Def.~\ref{HNC} gives the first topological invariant preserved by a homeomorphism of covers associated with the same permutation representation $T: G\to S_n$.

\begin{defn} \label{HNC} Consider a subgroup, $\Inn(G)\le K\le \Aut(G,\bfC)$.  This gives $K$-{\sl Nielsen classes}:  
$$\ni(G,\bfC)^K\eqdef \{\bg\in \bfC \mid \prod_{i=1}^r \row g r=1 \text{ (product-one) and } \lrang{\bg}=G \text{ (generation)}\}/K .$$ 
Denote the special case $K=\Inn(G)$ by $\ni(G,\bf C)^\inn$. From Prop.~\ref{innerbraiding}, $K$-Nielsen classes make sense. With $K=N_{S_n}(G,\bfC)$, and $T$ understood, call these {\sl absolute classes}, $\ni(G,\bfC)^\abs$.\end{defn} 

Then,  denote the Nielsen classes with $K=N_{S_n}(G,\bfC)$ by $\ni(G,\bfC)^\abs$ when $T$ is understood; these  Nielsen class elements characterize the usual equivalence of covers of $\prP^1_z$ of degree $\deg(T)$. 

Equivalence of covers $f: \prP^1_w\to \prP^1_z$, with $f: w \to f(w)=z$ a rational function, is usually absolute equivalence. From Prop.~\ref{innerbraiding}, $\ni(G,\bfC)/K$ (with $\Inn(G)\le K\le \Aut(G,\bfC)$), {\sl $K$-Nielsen classes}, makes sense.\footnote{There are other -- beyond quotienting by $K$ as here -- useful equivalences on Nielsen classes (as used in, say, \cite{BiFr82}). This paper only uses these.} 

Def.~\ref{HM} gives Nielsen classes representatives that arise often. Ex.~\ref{HM-separated}  uses them to produce abundant components of absolute spaces with trivial lift invariant and these properties: 
\begin{edesc} \item they are homeomorphism-separated from all other  components; 
\item  they have only one inner component above them. \end{edesc}  

\begin{defn}[\HM\ reps]  \label{HM} $\bg=(g_1,g_1^{-1},\dots g_s,g_s^{-1})\in \ni(G,\bfC)$ (so $2s=r$) is called a Harbater-Mumford rep. Its braid orbit (or its component) is  \HM. \end{defn}  

Many classical generators are based at $(\bz,z_0)$. Variations of them --  in the process of \lq\lq dragging a cover, up to inner equivalence,  by its branch points\rq\rq\  \eqref{drag} -- produces the braid action for computations in this paper (as in \eqref{iso}).  For now, fix classical generators $\sP_{\bz_0,z_0}$ based at $z_0$, with all covers in Lem.~\ref{homcovers} branched at $\bz_0$ and branch cycles computed from them.  

\begin{lem}  \label{homcovers} Take $\theta: \phi_1 \to \phi_2$,  a  homeomorphism of  covers, with branch cycles $\bg_i\in \ni(G,\bfC)^\abs$, $i=1,2$. Then, $\phi_i$ has a Galois closure cover $\hat\phi_i$ with branch cycles $\hat\bg_i\in \ni(G,\bfC)^\inn$, $i=1,2$, and an extending homeomorphism $\hat \theta: \hat \phi_1 \to \hat \phi_2$. Further, there is $\alpha\in N_{S_n}(G,\bfC)$ with $\hat \bg_1^\alpha=\hat \bg_2$.  \end{lem} 

\begin{proof} \cite[\S 3.1.3]{BFr02} gives the fiber product description of the Galois closure of $\phi_1$. We have added details for our application in  Prop.~\ref{inn-absmdf} for the fiber product construction for a family of covers. Use here \eqref{fiberGalclos} for constructing an individual cover in the family. 

For $\phi_1$, the Galois closure is {\sl a\/} component, $\hat \phi_1: \hat X_1 \to \prP^1_z$, of the fiber product of $\phi_1$ taken $n$ times with the fat diagonal removed. The subgroup of the natural $S_n$ action fixing $\hat X_1$ identifies with the group of the Galois closure.\footnote{It is the Galois group because it has as many elements as the degree of the cover.} 

The chosen group is $G$, but if the group $G\not = S_n$, then the complete set of components (off the fat diagonal) comes by applying coset reps of $G$ in $S_n$ to the given component. As families of covers of $\prP^1_z$, covers in these components have Galois groups identified as a conjugate in $S_n$ of $G$.  

\begin{center} We want components  with covers having groups identified precisely with $G$. \\ \bf Inner components with that property differ by conjugating\\ by  representatives of cosets of $G$ in $N_{S_n}(G,\bfC)$.\end{center} 

Do the same for $\phi_2$, and apply $\theta$ to the fiber product construction. It will map $\hat X_1$ to a component, $\hat \phi_2: \hat X_2\to \prP^1_z$, of the fiber product for $\phi_2$, which also has $G$ as the group of its projection to $\prP^1_z$. Refer to the extension of $\theta$ to those components as $\hat \theta$. This induces a morphism between the respective groups (both of which are $G$) that we denote by $\alpha=\alpha_{\hat \theta} \in N_{S_n}(G)$. The induced map on branch cycles for $\hat \phi_2$ is given by conjugating by $\alpha$ on the branch cycles $\bg_1$.\end{proof} 

\subsubsection{Dragging a cover by its branch points} \label{dragbybps} 
\cite{Fr77} (also \cite[\S2.2.1]{Fr20}) calls the following process {\sl dragging a cover, $\phi_0$, by its branch points\/} along a path $\bar P$, $z_{\bar P}: [0,1] \to U_r$ in $U_r$ starting at $\bz_0$.   

Choose  $z_t\in U_{\bz_t}$ continuously with $z_t$ distinct from entries in $z_{\bar P}(t)$. Take classical generators $\sP_{\bz_0,z_0}$ (above Lem.~\ref{homcovers}). For a cover, $\phi_0$,   branched at $\bz_0$: 
\begin{edesc}  \label{drag} 
\item \label{dragb}  $\sP_{\bz_0,z_0}$ canonically defines $\bg_0\in \ni(G,\bfC)^\inn$ and  {\sl dragging\/} $\sP_{\bz_0,z_0}$ along $\bar P$ gives classical generators $\sP_{\bz_t,z_t}$ on $U_{\bz_t}$ based at $z_{t}$. 
\item \label{dragc} This produces a path of homeomorphic covers, $\phi_t: X_t\to \prP^1_z$,  with (the same) branch cycles $\bg$ relative to $(\sP_{\bz_t},z_t)$, for all $t\in [0,1]$.\end{edesc}  
\cite[Lem.~1.1]{Fr77}  shows the independence of the basepoint in this process and the representative $z_{\bar P}$ of its homotopy class. 

\begin{defn} \label{covisotopy}  The cover $\phi_1: X_1\to\prP^1_z$ is the {\sl isotopy\/} of $\phi_0$ along $\bar P$.  For $\bar P\in \pi_1(U_r,\bz_0)$ a closed path, and $\phi_1: X_1\to \prP^1_z$ the cover at the end of the path, define $\bg_1$ to be branch cycles for $\phi_1$ {\sl relative to\/} $\sP_0$. Then the {\sl braid action\/}, $q_{\bar P}$, of $\bar P$ is given as $\bg\mapsto (\bg)q_{\bar P}=\bg_1$. 
This works equally well as a braid action on any $K$-Nielsen class elements. \end{defn} 

\begin{center} \bf Unless otherwise said, assume the transitive permutation \\ representation $T$ is given, and $N_{S_n}(G,\bfC)\eqdef N_T$.\end{center}

This leads to the following ingredients for describing isotopies of covers parametrized by paths in $U_r$, up to homotopy classes of $\pi_1(U_r,\bz_0)$. The following statements are documented in \cite{BiFr82} and \cite{Fr77} (with expositions in \cite{V96} and  \cite[\S2.2]{Fr20}).

\begin{edesc}  \label{iso} \item   \label{isoa}  Identification of  $\pi_1(U_r,\bz_0)$ with the {\sl Hurwitz monodromy group\/},  $H_r$. 
\item  \label{isob}  With $\Inn(G)\le N_T\le \Aut(G)$, the $H_r$ action on $\ni(G,\bfC)^{N_T}$ has two generators:\footnote{Conjugating $q_2$ by the $i$th power of $\sh$ gives the $(i\np 2)$-twist $q_{i\np 2}$, $-1\le i\le r\nm1$.} 
$$ \begin{array}{c} \text{The 2-twist }q_2:  \bg  \mapsto (g_1,g_2g_{3}g_2^{-1},g_2,g_{4},g_5,\dots);\\ 
\text{The {\sl shift }} \sh:  \bg\mapsto (g_2,g_3,\dots,g_r,g_1). 
\end{array}$$   
\end{edesc} 

Def.~\ref{braidable} is the key for Thm.~\ref{cosetbr}, for which we consider a braid orbit $\sO^\inn$ in $\ni(G,\bfC)^\inn$. 
\begin{defn} \label{braidable} An $\alpha\in N_T$ is {\sl braidable on $\sO^\inn$\/} if for $\bg\in \sO^\inn$, $(\bg)^\alpha \in \sO^\inn$. Denote the subgroup of $N_T$, of braidable elements on $\sO^\inn$, by $N_T^\br$ (or with related appropriate decoration).  
\end{defn} 

\begin{lem} \label{innerbraiding}   \lq\lq Dragging\rq\rq\ corresponds each element of  $\ni(G,\bfC)^{N_T}$ to a representative cover -- up to isotopy -- branched over any choice of $\bz_0\in U_r$ with classical generators, $\sP$, based at $z_0\not \in \{\bz_0\}$. 

From  \eql{bract}{bracta}, up to $G$ inner action, a Def.~\ref{covisotopy} isotopy is independent of the choice of $z_t$. 
\begin{edesc} \label{bract} \item \label{bracta}  For $h\in G$ and $\bg\in \ni(G,\bfC)$, there is $q\in H_r$ with $(\bg)q=h\bg h^{-1}$. 
\item \label{bractb} Conjugating $\bg\in\ni(G,\bfC)^\inn$  by $\alpha\in N_T$  commutes with the action of $H_r$. 
\item \label{bractc} Elements of $N_T/\Inn(G)$ permute the braid orbits of $H_r$ on $\ni(G,\bfC)^\inn$. \end{edesc}
\end{lem} 

\begin{proof} The first sentence follows from the description of the \lq\lq dragging\rq\rq\ process \eqref{drag}. 

\cite[Lem.~3.8]{BiFr82} shows \eql{bract}{bracta}. An explicit check on generators of $H_r$ in \eql{iso}{isob} gives \eql{bract}{bractb}.  Then, \eql{bract}{bractc} follows from the previous statements. \end{proof}

 \begin{rem} \label{HMrepuse} From \eql{bract}{bractb}, you can test if $\alpha\in N_T$ is braidable on just one element of $\sO^\inn$. That has often been used effectively  (e.g.~in \cite{FrV91}) on Harbater-Mumford braid orbits (Def.~\ref{HM}).  \end{rem}

\subsubsection{Dragging gives Thm.~\ref{cosetbr}}  \label{draggivesMainT} 
From covering space theory, the permutation action of $H_r$ on $\ni(G,\bfC)^{N_T}$ defines a cover $\Psi\eqdef \Psi_{\phi_0,\sP}: \sH(G,\bfC)^{N_T}\to U_r$. It can have more than one component. One is $\sH_{\phi_0,\sP}$, defined by the orbit, $O_{\phi_0}$, of $H_r$ on $\bg_0\in \ni(G,\bfC)^K$ corresponding to $\phi_0$. 

List the braid orbits on $\ni(G,\bfC)^{N_T}$ as orbit collections denoted $\row {\sO^{N_T}} u$, $1\le i\le u$.  Consider $\bg\in \sO^{N_T}_{i}$. Thm.~\ref{cosetbr} compares the braid orbits of $\sH(G,\bfC)^{N_T}$  with the braid orbits of $\ni(G,\bfC)^\inn$.  Each of the latter lies above a unique braid orbit of the former.

\begin{edesc} \label{homeoassumps} \item \label{homeoassumpsa} Assume $u=1$ and denote this unique braid orbit by $\sO^{N_T}$. 
\item \label{homeoassumpsb}  If $u>1$, all the  $\row {\sO^{N_T}} u$ are homeomorphism-separated. 
 \end{edesc}  

\begin{lem}[Check at $\bz_0$]  \label{checkbz} To check the division of braid orbits on $\sH(G,\bfC)^\inn$, for the situations listed in \eqref{homeoassumps}, it suffices to choose any $\sP$ classical generators based at any choice of $(\bz_0,z_0)$. Then, compute covers representing isotopy classes by their corresponding branch cycles. \end{lem} 

\begin{proof} If two covers $\phi_i\to \prP^1_z$ are homeomorphic, and are branched at $\bz_1$, then they have Galois closure with branch cycles $\hat \bg_i$ computed relative to $\sP'$ related by $(\hat \bg_1)\alpha=\hat \bg_2$, for some $\alpha\in K$. 

Apply the \lq\lq dragging\rq\rq\ process to drag them back to $\bz_0$ and compute their branch cycles relative to $\sP$, etc. Since the braid action commutes with the action of $\alpha$, \eql{bract}{bractb}, this proves the lemma. 
\end{proof} 

Thm.~\ref{cosetbr} runs through  \eqref{homeoassumps} by applying Lem.~\ref{checkbz} on branch cycles in $\ni(G,\bfC)^\inn$. Lem.~\ref{homcovers} says if two (not necessarily Galois) covers are homeomorphic, so are their Galois closures. 

\begin{thm} \label{cosetbr}  Assume \eql{homeoassumps}{homeoassumpsa} with  braid orbits in $\ni(G,\bfC)^\inn$ above $\sO^{N_T}$ listed as $\row {\sO^\inn} v$. With $\bg\in \ni(G,\bfC)^{N_T}$,  $\hat \bg\in \sO^\inn_1$ above it, denote braidable elements of $N_T$ on $\sO^\inn_1$ by $N^\br_1$. 

Then, $v=(N_T:N^\br_1)$. With $\{\alpha_j\mid j=1,\dots,v\}$ coset representives, 
\begin{edesc}  \label{homeoorbitsu=1} \item \label{homeoorbitsu=1a} $\{(\bg)\alpha_j\}$ are branch cycle reps. of covers in each braid orbit on $\ni(G,\bfC)^\inn, j=1,\dots,v$.
\item \label{homeoorbitsu=1b} The degree of the Hurwitz space component $\sH_{\sO^\inn_1}$ over $\sH_{\sO^{N_T}}$ is $(N^\br_1:\Inn(G))$. \end{edesc} 

Now consider the case  $\ni(G,\bfC)^{N_T}$ has $u> 1$ braid orbits as in \eql{homeoassumps}{homeoassumpsb}. Then, no two are automorphism-separated. List the braid orbits $\sO^\inn_{1_i},\dots,\sO^\inn_{v_i}$ in $\ni(G,\bfC)^\inn$ above $\sO_i^{N_T}$. Denote the braidable $\alpha$s on $\sO^\inn_{1_i}$ by $N^\br_{1_i}$, $1\le i\le u$. Then, juxtupose the braid orbits on $\ni(G,\bfC)^\inn$ by, running over $i$, replacing $v=(N_T:N^\br_1)$ by $v_i=(N_T:N_{1_i}^\br)$, etc.  
 \end{thm}

\begin{proof}  Choose a representative $\bg\in \ni(G,\bfC)^{N_T}$ with $\hat \bg\in \ni(G,\bfC)^\inn$ lying over it. The branch cycles of covers over the cover represented by $\bg$ are of the form $(\hat \bg)\alpha$ with $\alpha$ in the cosets of $G$ in $N_{S_n}(G,\bfC)$. Two belong in the same braid orbit if $\alpha$ is braidable. The expressions of \eqref{homeoorbitsu=1} make explicit the degrees of inner and absolute covers using braidable vs non-braidable automorphisms. That handles case \eql{homeoassumps}{homeoassumpsa}.  

Suppose $u>1$. Here is why \eql{homeoassumps}{homeoassumpsb} holds. If $\bg_1,\bg_2\in \ni(G,\bfC)^{N_T}$ are in separate braid orbits but not homeomorphism-separated, then above them are, respectively, $\hat \bg_1,\hat \bg_2\in \ni(G,\bfC)^\inn$ that are automorphism-separated by an element in $N_T$. Since $\bg_1,\bg_2$ are obtained from $\hat \bg_1 ,\hat \bg_2$ by modding out by $N_T$, modulo a braid, $\bg_1 $ and $\bg_2$ are in the same braid orbit, contrary to our assumption. 

Consider \eql{homeoassumps}{homeoassumpsb} and how to count braid orbits by dividing the branch cycles in  $\ni(G,\bfC)^\inn$ according to the braid orbits of $\ni(G,\bfC)^{N_T}$ they lie over. Then, taking representatives of these, apply the naming we have given by using which automorphisms are braidable as in \eql{homeoassumps}{homeoassumpsa}. 
\end{proof}

\subsection{More on Thm.~\ref{cosetbr}} \label{totspaces}  Def.~\ref{liftinv} defines the  lift invariant for use in two ways that never made an appearance in \cite{FrV91}, though it did in subsequent papers, especially \cite{Fr95} and \cite{BFr02}. Refer to the statements of \eqref{liftinvuse}: Thm.~\ref{cosetbr} immediately gives \eql{liftinvuse}{liftinvusea}. 

The appendix of \cite{FrV91} (for general Nielsen classes in \cite{Fr10}) says, assuming high multiplicity,  lift invariants determine inner Hurwitz space components. Also, for absolute spaces, $N_T$ orbits on lift invariants collect the inner spaces above a given absolute component. That gives \eql{liftinvuse}{liftinvuseb}.

Having {\sl one\/} absolute component \eql{homeoassumps}{homeoassumpsa} arises for Hurwitz space variants of classical spaces, say, as interpreting problems related to hyperelliptic jacobians, for example \S\ref{lpreliminary}. Indeed, all our examples play on this. 
 Cor.~\ref{autocomps} is almost immediate from Thm.~\ref{cosetbr}, using  $N_T$ orbits. 

    \begin{cor} \label{autocomps} With ${}_j\sH^\inn$ spaces corresponding to ${}_j\sO^\inn$, $j=1,2$, etc.,  $\Phi_j: {}_j\sH^\inn\to\sH_{\sO^{N_T}}$, $j=1,2$ are equivalent covers of $\sH_{\sO^{N_T}}$. The degree of the Hurwitz space component $\sH_{\sO^\inn_1}$ over $\sH_{\sO^{N_T}}$ is $(N^\br_1:\Inn(G))$ with $N^\br_1$ the braidable elements on $\sO^{N_T}$.\footnote{That $\Phi_{j}$, $j=1,2$ are equivalent covers does not mean that the families of covers corresponding to the spaces are the same: for that you must include the total families \eqref{totfam} and their moduli definition fields.} \end{cor} 

\begin{proof} Consider $\bg'\in \sO^{N_T}$ lying under ${}_1\bg\in {}_1\sO^\inn$.  The cover ${}_1\sH^\inn\to \sH_{\sO^{N_T}}$ 
is determined by the action of the subgroup of the braid group stabilizing ${}_1\bg$ acting on the elements, $S_1$, of ${}_j\sO^\inn$ lying over $\bg'$. The action of the braid group commutes with the action of $\alpha$. Therefore, applying $\alpha$ to the elements of $S_1$ gives $S_2$ with compatible braid actions on the Nielsen classes in ${}_2\sO^\inn$. This makes the corresponding covers equivalent.  \end{proof} 

\S\ref{homeoa} defines what it means that the components of $\sH(G,\bfC)^{N_T}$ are Schur-separated.  It starts the connection between representation covers and the Schur multiplier in the context of Frattini covers of a finite group. These connections show why Modular Towers and the lift invariant fit together. The section ends by giving the $\ell$-adic interpretation of \MT\ fibers and how they generalize the questions that arose in Serre's case \eqref{MTladic}. \S\ref{homeoc} shows how Thm.~\ref{cosetbr} strengthens \cite{GoH92} and \cite{GhT24} for situations of \eqref{homeoassumps} that arise in practice. \S\ref{homeod} distinguishes the geometric and arithmetic monodromy of covers.  Thm.~\ref{cosetbr} is a statement on the geometric monodromy groups of components of $\sH(G,\bfC)^{N_T}\to  U_r$. Interpreting the {\sl moduli definition fields\/} (Def.~\ref{mdf}, in particular, definition fields)  of these and the components of $\sH(G,\bfC)^\inn$ is the significant addition. \S\ref{homeod} does the first step in using {\sl Hilbert's Irreducibility Theorem\/} as a tool for the $G_\bQ$ action on Hurwitz space components.\footnote{\HIT\ has always been underappreciated, but \cite{Se68}, and the related \cite{Fr78} show their fascination with enhancing it. \cite[\S 5.1]{Se97} \cite[Chaps.~13 and 14]{FrJ86}$_4$ give more extensive references in support of that.} We abbreviate reference to it by the acronym \HIT. 
 
\subsubsection{Schur-separated definitions} \label{homeoa}  A representation cover, $\hat \psi: \hat G \to G$, is a Frattini central extension of $G$ whose kernel -- the Schur multiplier of $G$ -- is $\SM_{G}$. As with all Frattini covers, we can write this as the fiber product over $G$ of $\ell$-Frattini covers $\hat \psi_\ell: \hat G_\ell \to G$ (an $\ell$-representation cover of $G$) for which the kernel is the $\ell$ part, $\SM_{G,\ell}$ of $\SM_G$.  Our examples have these conditions:
\begin{edesc} \label{easySchur} \item \label{easySchura} The $\ell'$ condition, $(N_{\bfC},\ell)=1$, on $\bfC$ holds and there is only one prime $\ell$ dividing $\SM_G$. 
\item \label{easySchurb} From \eql{easySchur}{easySchura}, and Schur-Zassenhaus we interpret the classes of $\bfC$ uniquely as classes in the representation cover. 
\item \label{easySchurc} The $\ell$-representation cover, $\tilde \psi_\ell$, is $\ell$-perfect (has no $\bZ/\ell$ quotient).\footnote{Therefore, the $\ell$-representation cover is uniquely defined and is a characteristic quotient of the universal $\ell$-Frattini cover of $G$.}
\end{edesc}

\begin{lem} \label{perfnongen} A profinite group of order divisible by $\ell$ is $\ell$-perfect if and only if it has generators among its $\ell'$ elements. \end{lem}

\begin{proof} The subgroup, $H$, of $G$ generated by all its $\ell'$ elements is a normal subgroup of $G$. It is easy to see that $H=G$ if and only if $G$ is $\ell$-perfect.\end{proof}

Def.~\ref{liftinv} is the formula for the lift invariant when  the $\ell'$ condition holds. 
\begin{defn}[Lift invariant]  \label{liftinv}   For $O$ a braid orbit on $\ni(G,\bfC)$, and $\bg\in O$,  as in \eql{easySchur}{easySchurc} $$\begin{array}{c}\text{the lift invariant is } s_{\bg}(O)\eqdef s_{\tilde \psi_\ell}=\prod_{i=1}^r \tilde g_i . \text{ More generally, }\\
\text{for $\psi_{H/G}: H\to G$ an $\ell$-central Frattini cover, define $s_{\psi_{H/G}}$  using $\tilde\bg\in \bfC\cap H$ over $\bg$.}  
\end{array} $$ It is an elementary exercise that \HM\ elements (Def.~\ref{HM}) always have trivial lift invariant; as do Nielsen classes with the generalizing form of \eql{steppingup}{steppingupd}. 
\end{defn} 

Recall \eqref{NormactonLiftInv} $\alpha\in N_T$ acting on a lift invariant: $s_{\hat \bg}\mapsto  s_{\hat \bg^\alpha}$. 
\begin{defn} \label{schur-sep} To a braid orbit $\sO'\le \ni(G,\bfC)^\abs$, attach the collection $S_{\sO'}$ of lift invariants running over braid orbits of $\sH(G,\bfC)^\inn$ above $\sO'$. 
Components corresponding to braid orbits $\sO'$ and $\sO{''}$ in $\ni(G,\bfC)^\abs$ are {\sl Schur-separated\/} if the $S_{\sO'}$ and $S_{\sO^{''}}$ are distinct. \end{defn}  
Lem \ref{liftinvlem} shows Schur-separated components have different moduli properties. Thus, their topological separation. 
\newcommand{\Li}{{\bf Li}} 
In our examples, $\SM_{G,\ell}$ -- always abelian -- will be cyclic. That allows determining the moduli definition field of $\sH(G,\bfC)^{N_T}$ components. 
From the $\ell'$ condition on $\bfC$, with  $H\to G$ an $\ell$-Frattini cover,  The notation $\tilde \bg \in \bfC\cap H$ as lying over $\bg\in \ni(G,\bfC)$ now makes sense. \S\ref{l'lift} puts this in the context of the Universal $\ell$-Frattini cover of $G$ when $\ell||G|$. 

If we can decide what values of the lift invariant are achieved,  this reduces finding moduli definition fields to finding them for automorphism-separated $\sH(G,\bfC)^\inn$ components.  Assume an $\ell$-representation cover $\hat \psi: \hat G_\ell\to G$ satisfies $\ell$-perfect condition \eql{easySchur}{easySchurc}.  

This holds in its purest form in the \OIT-related \S\ref{weilpairing} example:  We explicitly compute the (distinct) lift invariants of the $\sH(G,\bfC)^{N_T}$  components with $T$ the coset representation of the class of involutions in $(\bZ/\ell^{k\np1})^2\xs \bZ/2$, and $\bfC=\bfC_{2^4}$, four repetitions of the involution class.  Prop.~\ref{GL2}: There are two $\sH(G,\bfC)^{N_T}$ components, and above each the inner components (of $\sH(G,\bfC)^\inn$) have lift invariants that differ by the square of elements in $(\bZ/\ell^{k\np1})^*$. One corollary: The lift invariant gives the Weil pairing -- giving the moduli definition field --  on modular curves classically denoted $X_n$),  extending our interpretation of the modular curves $X_0(\ell^{k\np1})$ as Hurwitz spaces. 

\S\ref{absinnAn} (resp.~\S\ref{heiscase}) applies Thm.~\S\ref{cosetbr} when $G=A_n$ (resp.~$G=(\bZ/\ell^{k\np1})^2\xs \bZ/3$). \S\ref{ES} discussed the former in detail. For the latter, which we denote as $G_{\ell,0,3}$ ($k=0$ indicating the group at level 0), we encounter some of the problems that we haven't resolved in this paper. We divided the Hurwitz space structure into two types of components, whose union forms $\sH_{\HM-DI}$. Applying the lift invariant implies only the \HM\ components give \MT s. Further, at each level, there are several \HM\ components, so these -- with lift invariant 0 -- are not Schur-separated. 

Qualitative description of the geometric and arithmetic monodromy groups of the corresponding \MT s generalizes Serre's \OIT. \S\ref{HIT-ST} uses conjectures, Andr\'e-Oort and Coleman-Oort, in particular, to compare the nature of the arising of $\ell$-adic representations (on Tate modules) and decomposition groups as the image of $G_K$, $K$ a number field, from: 
\begin{edesc} \label{compladic} \item \label{compladica} Serre's representations of $G_K^\ab$, the Galois group of the abelian closure of $K$;
\item \label{compladicb} Shimura-Taniyama (\ST) abelian varieties in Siegel Space; and 
\item \label{compladicc} \MT\ Jacobian fibers from $(G,\bfC,\ell)$, distinguishing between the Shimura-Taniyama case and when the fiber decomposition is open in the arithmetic monodromy of the \MT. \end{edesc} 

Each of our cases designations a prime $\ell$, with the extra conditions that $G$ is $\ell$-perfect \eql{easySchur}{easySchurc} and $(\ell,N_{\bfC})=1$ (as above Def.~\ref{NormactonLiftInv}).\footnote{Many examples have a natural series of groups with their corresponding $\ell$, but this isn't necessary.} we phrase a \MT\ as a statement on $\ell$-adic representations with $\sH(G,\bfC,\ell)_0$ an irreducible component of the level 0 space $\sH(G,\bfC,\ell)^\inn_0$. 

\begin{edesc} \label{MTladic} \item \label{MTladica}  $\sH(G,\bfC,\ell)_0$ supports a family of $\ell$-adic representations $\sL \to \sH(G,\bfC,\ell)_0$ with the fiber over $\bp\in \sH(G,\bfC,\ell)_0$ a quotient of Tate module of the Jacobian, $J_\bp$.
\item \label{MTladicb} The monodromy action on the fibers comes from the braid action. 
\item \label{MTladicc} The condition $(\ell,N_{\bfC})=1$ is what produces \eql{MTladic}{MTladica} from the Tate module. \end{edesc} While Serre's characterization \eql{compladic}{compladica} is explicit, the gadgets he uses (e.g. Weil's clever restriction of scalars \eqref{restscalars}) are rarely used, and they don't produce $\ell$-adic representations that we know how to relate to $\ell$-adic cohomology, much less to abelian varieties.

Yet, having geometric objects representing the \MT s has a graphic representation, especially from  \sh-incidence cusp pairing diagrams from which we can apply our main tool, the braid action (to test the target property, that the \MT\ are eventually $\ell$-Frattini, Def.~\ref{evenfratt}). This paper focuses on how the lift invariant yields \MT s that support either of the types, \HIT\ and \ST\, which generalize the two fiber types, $\GL_2$ and \CM\, in Serre's example.   \cite{Fr26b} continues \S\ref{absinnAn} and \S\ref{heiscase} by computing generators of the $\ell$-adic modules from the branch cycle description $\bg$ of the fiber. This gives the monodromy action explicitly as in Serre's case in \S\ref{weilpairing}. \cite{Fr26b} concludes by summarizing why \MT s, and their basis in the universal $\ell$-Frattini covers of a finite group, are natural for the hoped-for properties generalizing the \OIT. 

\begin{rem} Dispensing with distinguishing between arithmetic and geometric monodromy isn't the complete story, but Ex.~\ref{HM-separated} gives \MT s, starting with any group $G$, where each level has a Schur-separated component, giving levels defined over $\bQ$  using the argument of \cite[Thm.~3.21]{Fr95}.\footnote{The value of $r=r_\bfC$ is explicit, but $>> 4$. Thus, it is beyond my hand-calculational ability. I could use a computer programmer here to apply GAP, say, to compute \sh-incidence matrices.} \end{rem}

\subsubsection{Thm.~\ref{cosetbr} strengthens \cite{GoH92} and \cite{GhT24}} \label{homeoc}  
 \cite{FrV91} primarily concentrated on Schur-Separated components, mostly by changing $\bfC$ so there was just one componant. Our examples show that it doesn't suffice in practical applications. 
As corollaries, \cite{FrV91} used Automorphism-separated components in \cite{FrV92}, though without recognizing the key definition of homeomorphic covers (Def.~\ref{homeogal}) for which a version dominates \cite{GoH92} and \cite{GhT24}. For those two papers, $T$ is the regular representation.  \footnote{I was unaware of \cite{GoH92} until 2022, while refereeing \cite{GhT24}. I thought those authors were unaware of \cite{FrV91}, but they list  a 1991 V\"olklein paper in their references (without citing it in the paper). \cite{FrV91} and \cite{FrV92} were written and sent to journals while the authors worked together at the University of Florida, 1986-1989.} 

The main result of \cite{GoH92} is the connectedness of the space of covers in $\sH(G,\bfC)^\Aut(G,\bfC)$ ($\Aut=\Aut(G,\bfC$) homeomorphic to a particular cover $\phi_0$ they select at the beginning. They use the connectedness of a Teichm\"uller ball. Thus, avoiding Teichm\"ulller theory, \cite{GhT24} rightfully claims an easier proof. Ours is easier still, using  \lq\lq dragging\rq\rq\ covers from \cite{Fr77}. Here, we add distinguishing between Automorphism-separated and Homeomorphism-separated components for comparison with the Schur-separated components, as in \S\ref{homeoa}. 

For Riemann surface covers, selecting one cover for comparison with all others lacks a moduli interpretation of the isotopy classes of cover collections. Usually, there are several possible permutation representations for $G$, and therefore different possibilities for $K$. In comparison with, say, \cite{GhT24}, we might want $K=\Aut(G)$ for application in the genus formula \eqref{RHOrbitEq}. It is optional to take the regular representation for this. \eqref{tchoice} gives reasons to choose $T$ thoughtfully.
\begin{edesc} \label{tchoice} \item \label{tchoicea} Doing so can produce $\sH(G,\bfC)^K\,$s (and reduced versions) as classical moduli, as in \S \ref{absinnAn}). 
\item \label{tchoicec} Having fine moduli is valuable, rare for  $\sH(G,\bfC)^K$ when $T$ is the regular rep. (Rem.~\ref{notfineabs}). \end{edesc} 

\begin{rem}[Automorphisms not preserving $\bfC$] \label{autononpres} An $\alpha\in \Aut(G)$ (or $N_{S_n}(G)$) not preserving $\bfC$, would not be braidable. Yet,  the equivalences used in \cite{GoH92} and \cite{GhT24} would still have included it.  Applying $\alpha$ to $\ni(G,\bfC)^\dagger$ would map it into another Nielsen class, $\ni(G,(\bfC)\alpha)^{(\dagger)\alpha}$ where $\alpha$ might even change the permutation representation. We excluded this consideration. 

Still: Components corresponding by $\alpha$ on $\ni(G,\bfC)^\dagger$ and $\ni(G,(\bfC)\alpha)^{(\dagger)\alpha}$ would give (as in Cor.~\ref{autocomps}) equivalent covers of $U_r$ (or of $J_r$). Although the Nielsen classes differ for these components, we can ask if some $\sigma\in G_\mathbb{Q}$ conjugates the total spaces over these components. Ex.~\ref{davprs} has a moduli definition field larger than the definition field of the configuration space cover. \end{rem}

\subsubsection{Geometric vs Arithmetic Monodromy of covers} \label{homeod}  
Throughout, we apply \cite{Gr-Re57} -- an analytic cover, $\phi: Y\to X$, (of normal varieties) of an open subset of a quasiprojective variety is algebraic -- as did  \cite{Fr77}, \cite{FrV91}, \cite{GhT24}, etc. This allows:
\begin{edesc} \label{consfinemod}  \item \label{consfinemoda} taking function fields of our main spaces over a defining field; and 
\item \label{consfinemodb} having a well-defined field generated by coordinates of a point on a {\sl fine\/} moduli space.\end{edesc}
The braid calculations of Thm.~\ref{cosetbr} give us (geometric) components of the spaces $\sH(G,\bfC)^\inn$ and $\sH(G,\bfC)^{N_T}$. 
We use moduli interpretations of the definition field of a cover $\hat \phi_{\hat \bp}: \hat X_{\hat \bp} \to \prP^1_z$,  the Galois closure of $\phi_\bp: X_\bp\to \prP^1_z$;  both are fibers in total spaces over $\sH(G,\bfC)^\inn$ and $\sH(G,\bfC)^{N_T}$:  
\begin{edesc} \label{moddefcomps} \item  \label{moddefcompsa} the coordinates of $\hat \bp\in \sH(G,\bfC)^\inn$ lying over $\bp\in \sH(G,\bfC)^{N_T}$ {\sl and\/};  \item \label{moddefcompsb}  definition fields of total spaces over components containing those points: as in Cor.~\ref{corinn-absmdf}.  \end{edesc}
Although \eql{moddefcomps}{moddefcompsb} does not appear explicitly in some classical moduli results, it is necessary.\footnote{Ex.~\ref{davprs}, implicit in the solution of Davenport's Problem, was put here explicitly to clarify that.}  

\cite[\S 0.C]{Fr77} has the details for considering finite/flat morphisms of normal varieties, giving the Grothendieck definition of (ramified) covers by quoting \cite{Mu66}. Except Mumford has everything over an algebraically closed field, inappropriate in our applications. 

The following, expressed in function fields, for the cover of normal, absolutely irreducible varieties $\phi: X\to Y$ with definition field $F$, appears in \cite[(2.2)]{Fr77}. For simplicity, assume $F\le \bC$, and $X$ is absolutely irreducible. 
We use it for defining the moduli definition field (Def.~\ref{mdf}), especially when Hurwitz spaces have more than one component. 

\begin{center} \bf Extension of Constants Diagram \end{center} 

With $\widehat{F(X)}$ a Galois closure of $F(X)/F(Y)$, denote the constants of $\widehat{F(X)}$, the {\sl extension of constants field},\footnote{See \eqref{hurnot} on the branch cycle view of choices.}  by $\hat F$. Rest. denotes restriction of automorphisms of $\widehat{F(X)}$ to $\hat F$, surjective in \eqref{extconstants} because $\overline{F} \cap F(\varphi) = F$.
The following sequence of groups is exact.   
\begin{equation} \label{extconstants} 
1 \rightarrow G\eqdef G(\widehat{F(X)}/\widehat{F}(Y)) \rightarrow G(\widehat{F(X)}/F(Y)) \xrightarrow{\rest.} G(\widehat{F}/F) \rightarrow 1
\end{equation}
The middle (resp.~first) term of \eqref{extconstants} is the
\emph{arithmetic} (resp.~\emph{geometric})  monodromy of the 
extension $F(X)/F(Y)$. The diagram produces $\hat F$ by applying {\sl Hilbert's Irreducibility Theorem\/} (often aiming for $\hat F=F$; so realizing $G$ as a Galois group over $F$). 

Prop.~\ref{inn-absmdf} applies it to $\phi_{\bp}: X_\bp\to \prP^1_z$, with $\bp \in \sH'$ a component of $\sH(G,\bfC)^{N_T}$, to compare with $\hat \bp\in \sH(G,\bfC)^\inn$  over $\bp$, to find the correct field over which a Galois cover of $\prP^1_z$ represents $\hat \bp$.  

We connect \HIT\ and the {\sl Coleman-Oort conjecture\/} (Rem.~\ref{AOB}) as about decomposition groups on towers of moduli spaces. \S\ref{totalspaces} reminds of total families of Hurwitz spaces. 

\S\ref{HSTowers} forms the generalization of modular curve towers on which we can formulate a result comparing the decomposition groups in the tower with the monodromy groups of the towers. This is the deepest place for the lift invariant: ensuring the existence of the tower using a generalization of a classical notion called $\ell$-Poincar\'e duality. Using Serre's \OIT\ as a guide, we introduce the two types of decomposition groups -- \HIT\ and \ST, respectively generalizing $\GL_2$ and $\CM$. 

The Coleman-Oort conjecture\footnote{Often it is the Andr\'e-Oort conjecture that is mentioned, but that is purely about points on Siegel space, and has none of the refinement of differentiating what happens for special curve locii.} concentrates on the locus of Jacobians of curves in Siegel space (and their variants). If true, it says that Serre's \OIT\ generalizes in a surprising way. Our goal, using examples, illustrates its relevance to modern problems: First, showing the relationship between the lift invariant applied to Serre's \OIT\ for the cyclotomic definition fields usually arising from the Weil pairing; and then two general cases where the group theory is modest, but gives dramatic Hurwitz space component results. 

\section{Total  Spaces} \label{totalspaces} A {\sl total space\/} -- the topic of \S\ref{fmc} -- over a component of $\sH(G,\bfC)^\dagger$ is given by  
\begin{equation} \label{totalspace}\begin{array}{c}  \text{$\Phi: \sT^\dagger \to \sH(G,\bfC)^\dagger\times \prP^1_z$ for which the fiber, $\sT_\bp \to \bp\times\prP^1_z$, } \\ \text{ over $\bp\in \sH(G,\bfC)^\dagger$ represents the cover corresponding to $\bp$.}\end{array}\end{equation}  

For spaces reduced by the action of $\PSL_2(\bC)$, the target is more complicated (Rem.~\ref{redtarget}). \S\ref{finered} gives everything required for reduced spaces to generalize the goals of \cite{GoH92} and \cite{GhT24}. 

\begin{prob}[$G_\bQ$ Goal] \label{GQgoal} Give the action of $G_\bQ$ on total spaces over components of $\sH(G,\bfC)^\dagger$.  \end{prob}   

Suppose $\sH'$ is a component of $\sH(G,\bfC)^\dagger$.  Then, the total space over $\sH'$ defines the {\sl moduli definition field\/} (Def.~\ref{mdf}), $\bQ_{\sH'}$,  of the component. Given $\bp\in \sH'(\bar \bQ)$, $\bQ_{\sH'}(\bp)$ is the minimal definition field of a representing cover corresponding to $\bp$. Even if $\bQ(\bp)=\bQ$, this will be a larger field if $[\bQ_{\sH'}:\bQ]>1$. \S\ref{mdfsec} thus gives structure to answer Prob.~\ref{GQgoal}. 

The {\sl branch cycle lemma\/},  Prop.~\ref{bcl}, is our model for computing the moduli definition field. It gives $\sQ_{\sH^\dagger}$ explicitly (with $\dagger=\inn$ or $\abs$)   when the Hurwitz space is absolutely irreducible.  

 Rem.~\ref{compmult2} answers Prob.~\ref{GQgoal} when we only have components defined by topological separation from Schur-separation, and $G$ has a cyclic Schur multiplier. 

\subsection{Fine moduli conditions} \label{fmc}  Again, $T$ is transitive and faithful. In Lem.~\ref{finemodabsinn}, denote $G(T,1)$ -- the stabilizer of 1 in the representation $T$ -- as $G(1)$. \S\ref{fineinnabs} gives the conditions for fine inner and fine absolute moduli corresponding to the parameters $(G,\bfC,T)$. \S\ref{withoutfinemod} compares the different approaches of \cite{Fr77} and \cite{GhT24} to forming total spaces without having fine moduli. 

\subsubsection{Fine inner and absolute moduli} \label{fineinnabs} 
Lem.~\ref{finemodabsinn} improves \cite[Prop.~2.2]{Fr77} by simplifying the relation between absolute and inner fine moduli, interpreting both on Nielsen classes. This enhancement relates fine absolute and fine inner moduli of Hurwitz spaces. \cite{FrV91}, and its corollary paper \cite{FrV92}, often assumed fine absolute, so, automatically,  fine inner moduli. 

Lem.~\ref{PDTR} tightens \cite[Lem.~2.2]{Fr77}.  Use the notation of the Extension of Constants diagram \eqref{extconstants} for a cover $\phi: X\to Y$ with $G=G(\widehat{F(X)}/\hat F(Y))$, $\hat F$ the constants of $\widehat{F(X)}$.    

\begin{lem}\label{PDTR} The normalizer of $G(1)=G(\widehat{F(X)}/\hat F(X))$ in $G$, $N_G(G(1))/G(1)$, identifies with $\Aut(X/Y, F)$. 
\footnote{As noted in \cite[Lem.~2.2]{Fr77},  in particular, if $T$ is  
primitive (meaning no groups properly between $G$ and $G(1)$; e.g., doubly transitive), and $G$ is not a cyclic group 
of prime degree, then $\Aut(X/Y, F)=\{\Id.\}$.}
\end{lem}

\begin{proof}
Let $x=x^{(1)}$ be a primitive generator of $F(X)/F(Y); x^{(1)},\ldots,x^{(n)}$ 
the conjugates of $x$ over $F(Y)$.  A $\beta \in \Aut(X/Y, F)$, induces a field
automorphism of $F(Y)(x^{(1)})$  determined by a polynomial 
$m(x) \in F(Y)[x]$.  Take $m(x) \in F(Y)[x]$ where $m(x)$ is the unique 
polynomial of degree at most $n-1$ with coefficients in $F(Y)$ with $m(x^{(1)})=\beta(x^{(1)})$.  Since $X$ is absolutely irreducible, automorphisms of $F(X)/F(Y)=F(Y,x^{(1)})/F(Y)$ correspond to automorphisms of $\hat F(X)/\hat F(Y)$. 

A fundamental lemma of Galois theory says any such automorphism extends to an automorphism $\beta^*$ of $\widehat{F(X)}/\hat F(Y)$, that maps $\hat F(Y,x^{(1)})$ into itself. Therefore, for $g\in G(1)$ -- fixed on $F(Y,x^{(1)})$ -- so is $(\beta^*)^{-1} g \beta^*$: $\beta^*$ normalizes $G(1)$. Thus, 
automorphisms of  $F(Y)(x^{(1)})/F(Y)$ identify with $N_G(G(1))/G(1)$.  \end{proof}

Denote the centralizer of $G$ in $S_n$ by $\Cen_{S_n}(G)$. It is a normal in $N_{S_n}(G)$. Prop.~\ref{finemodabsinn} puts fine inner and absolute moduli on par with a Nielsen class interpretation, making the former a natural result of the latter. 
\begin{prop} \label{finemodabsinn} Elements of $G$ that permute the right cosets of $G(1)$ by action on the left are in $N_{G}(G(1))$; distinct actions are given by  $N_{G}(G(1))/G(1)$. Then, $\Cen_{S_n}(G)$ 
is isomorphic to $N_{G}(G(1))/G(1)$ and so to the automorphisms, $\Aut(X/Y, F)$, of $X$ over $Y$, defined over $F$.
\begin{edesc} \label{fmod} \item Elements of $\Cen_{S_n}(G)$ also permute the (right) cosets of $G(1)$ by action on the left. 
\item  \label{fmodb} $N_{G}(G(1))/G(1) \cong \Cen_{S_n}(G)$, and if the former is trivial, then $G$ has no center. \end{edesc} 

Fine moduli for  $\sH(G,\bfC)^\inn$ (resp.~$\sH(G,\bfC)^{N_T}$) is that the center of $G$ (resp.~$\Cen_{S_n}(G$)) is trivial.\footnote{Referred to as the {\sl self-normalizing\/} condition.} So, the latter implies the former. \end{prop}

\begin{proof}  A $g\in G$ normalizes $G(1)$ if and only $g G(1)g_i= gG(1)g^{-1}gg_i=G(1)g_{\alpha_j}$ for some $\alpha_j$. Therefore, those  $g\in G$ that permute these cosets under multiplication on the \emph{left}  are exactly the elements of $N_{G}(G(1))$. Elements that stabilize all these left 
cosets are the elements of $G(1)$.  This identifies $N_{S_n}(G(1))/G(1)$ as acting faithfully by multiplication on the left of these cosets.  

List the elements of $h\in N_{G}(G(1))/G(1)$ as $\{\row h k \}$. Write $h_j G(1)=G(1)g_{\alpha_j}$, $1\le j\le k$ with the equation for $h'\in S_n$ centralizing $G$: 
\begin{equation} \label{commute} h' \circ T(g) = T(g) \circ h' \text{ for }  \text{ for each } g \in G. 
\end{equation}

We will form an $h'_j$ in $\Cen_{S_n}(G)$ that starts with $h'_j: 1 \mapsto \alpha_j$.  
\begin{center} Apply both sides of \eqref{commute}  to 1 with $g=g_i$: $(\alpha_j)T(g_i)=(i)h'_j$.\end{center}  This only depends on the coset $G(1)g_i$ and not on $g_i$ since the right side has the same image on 1.  

Running over coset representatives, $g_i$ determines $h'_j$ as a permutation that commutes with $G$.  Therefore, the orders of $\Cen_{S_n}(G)$ and $N_{G}(G(1))/G(1)$ both equal $k$. 

Now we interpret fine moduli for the spaces $\sH(G,\bfC)^\inn$ and $\sH(G,\bfC)^{N_T}$. Start with the latter. List an element of  $\ni(G,\bfC)^\abs=\ni(G,\bfC)^{N_T}$ as the set $\bar \bg\eqdef \{\alpha \bg\alpha^{-1}\}_{\alpha\in N_{S_n}(G)}$. Suppose $\sH'$ is a component of $\sH(G,\bfC)^{N_T}$ and at $(\bz_0, z_0)$ we have, relative to classical generators, $\sP_{\bz_0,z_0}$, as used in \S\ref{isotopy-Hur}, chosen branch cycles $\bg$ for a cover with given labeling of the points over $z_0$. Suppose $\bg'\in \bar\bg$ are branch cycles of the path dragged to the end point of $\bar P$ relative to $\sP_{\bz_0,z_0}$. 

Fine moduli over $\sH'$ is equivalent, for every event described above, to uniquely picking out an isomorphism between the original cover given by $\bg$ and the new cover given by $\bg'$. This isomorphism comes from choosing an element in $N_T$ to rename the points over $z_0$. For fine moduli, we need constraints to make this choice unique. This happens if and only if conjugation by only one element of $N_T$ gives the same branch cycles. This is equivalent to $\Cen_{S_n}(G)$ is trivial. 

The same argument applies to a component of $\sH(G,\bfC)^\inn$, except here, instead of $N_{S_n}(G)$, we must verify the conclusion for $\Inn(G)$. That is, if and only if the center of $G$ is trivial. 
\end{proof}

The conclusion of fine moduli with $\dagger=\inn$ or $\abs$ (\cite[\S5]{Fr77} or \cite[Main Theorem]{FrV91}) is the existence of a   unique total family over $\sH(G,\bfC)^\dagger$ (as in \eqref{totalspace}): 
\begin{equation} \label{totfam} \Phi_{G,\bfC,\dagger}: \sT^\dagger\to \sH(G,\bfC)^\dagger\times\prP^1\longmapright {\pr\times \Id} {40} U_r\times \prP^1_z.\end{equation}  

\begin{rem} From Prop.~\ref{finemodabsinn} the conclusions of Lem.~\ref{PDTR} can be stated as $\Cen_{S_n}(G)$ 
is isomorphic to the automorphisms, $\Aut(X/Y, F)$, of $X$ over $Y$, defined over $F$. As in a footnote above, if $T$   
primitive and $G$ not cyclic group 
of prime degree, then $\Aut(X/Y, F)=\{\Id.\}$.
\end{rem} 

\subsubsection{Interpreting Prop.~\ref{finemodabsinn} without fine moduli} \label{withoutfinemod} 
Don't assume any fine moduli conditions. Start from any point $\bz_0\in U_r$, with a cover $\phi_0\in \sH(\bfC,G)^\abs$. Then, the \lq\lq dragging\rq\rq\ process \S\ref{dragbybps}  combined with the fiber product Galois closure construction (proof of Lem.~\ref{homcovers}), allows forming both an inner and absolute (total) space of covers locally over a neighborhood $D_{\bz_0}$ of $\bz_0$ in $U_r$: 
\begin{equation} \begin{array}{c}\Phi^\inn_{D_{\bz_0}}: \sT^\inn_{D_{\bz_0}}\to D_{\bz_0}\times\prP^1_z  \text{ and }
\Phi^\abs_{D_{\bz_0}}: \sT^\abs_{D_{\bz_0}}\to D_{\bz_0}\times\prP^1_z, \text{ and} \\
\text{ map between them giving $\hat\phi^\inn_{\bz} \to \phi^\abs_{\bz}$ on each fiber over $\bz\in D_{\bz_0}$.}\end{array} \end{equation}

The proof of Prop.~\ref{finemodabsinn} shows this suffices to form Hurwitz spaces and the families locally over them. What \cite[p.~57-58]{Fr77} did has two parts.
\begin{edesc} \label{tothomfam} \item \label{tothomfama} Form total families over obvious affine pieces of $U_r$ (noting, without fine moduli, they don't patch together uniquely). 
\item \label{tothomfamb} Use Grothendieck's non-abelian $H^1$ set and his $H^2$ with coefficients in the center sheaf with stalks $\Cen(G)$, applied to \eql{tothomfam}{tothomfama} to form the set of total families. 
\end{edesc} 

\cite{GhT24} notes the dichotomy between three cases for forming such a family: 
\begin{edesc} \label{totfamprop} \item \label{totfama}  $G$ centerless, where such a family is unique over $\sH(G,\bfC)^\inn$; 
\item \label{totfamb}  $G$ is abelian, where such a family exists over $U_r$ though it is not unique; and 
\item \label{totfamc} $G$ has a center, but is not abelian, the \cite{GhT24} construction forms a canonical system of families with natural maps between them over finite covers of $\sH(G,\bfC)^{\Aut(G)}$. \end{edesc}

\cite[p.~3]{GhT24}: \lq\lq For general $G$, one cannot pick out a distinguished choice in a canonical way. This refers to  \cite[p.~57-58]{Fr77} where a cohomological
interpretation of this difficulty is given. They want to form a total family of covers,  doing it in all cases at the loss of getting several copies of the same cover in the family.  \eqref{tothomfam} has each representative cover appear in the total family just once.        

\begin{prob} Find a common framework for \eqref{tothomfam} and \eqref{totfam}. Keep in mind,  Prop.~\ref{finemodabsinn} contains a comparison of the absolute and inner Hurwitz spaces, while \eqref{totfam} does not. \end{prob} 

\subsection{Reduced spaces and a genus formula} \label{finered}  Consider the space from the reduction action of $\PSL_2(\bC)$ (as in \eqref{PSLact}) on the spaces $\sH(G,\bfC)^\dagger$ with $\dagger=\abs$ or $\inn$. Denote the resulting reduced space $\sH(G,\bfC)^{\dagger,\rd}$. Since $\PSL_2(\bC)$ is connected, components of $\sH(G,\bfC)^\dagger$ and $\sH(G,\bfC)^{\dagger,\rd}$ will correspond one-one. \S\ref{redfinmod} gives fine moduli conditions for each corresponding reduced space. 

Our goal is to identify properties separating distinct components. Initially deal with $\sH(G,\bfC)^\dagger$. Then, quotient out by $\PSL_2(\bC)$, reducing the complex dimension of the spaces by three. For $r=4$, normalizing reduced spaces gives a nonsingular cover of the $j$-line ramified only over $\{0,1,\infty\}$. 

Continuing \S\ref{redfinmod}, \S\ref{genformr=4}  uses the induced $H_4$ \eqref{iso} action on {\sl reduced Nielsen classes\/} (Def.~\ref{redclasses}),  in particular showing how to compute components, cusps and genuses of these $j$-line coverings.  \S\ref{examples} use the rubric of Prop.~\ref{cosetbr} to make these computations on examples. 

\subsubsection{Reduced inner and absolute spaces} \label{redfinmod}  Using reduced Nielsen classes, we can make computations on reduced Hurwitz spaces. 

 \begin{defn} \label{redclasses} For $r=4$, the Klein 4-group $K_4=\sQ''\eqdef \lrang{\sh^2,q_1q_3^{-1}}$ is the {\sl reduction group\/} and  $\Cu_4=\lrang{q_2,\sQ''}$ is the {\sl cusp group}. Then, for  $\dagger=\abs$ or $\inn$ equivalence, and $r=4$, \begin{center} the {\sl reduced  $\dagger$ Nielsen class\/} is $\ni(G,\bfC)^{\dagger}/\sQ''\eqdef \ni(G,\bfC)^{\dagger,\red}$. \end{center} \end{defn} 

For $r=4$, define $M_4$  to be the quotient of the braid group $B_4$  (with classical generators denoted $Q_1,Q_2,Q_3$) with these extra relations: 
$$\tau_1=(Q_3Q_2)^3=1,\ \tau_2=Q_1^{-2}Q_3^2=1,\ \tau_3=(Q_2Q_1)^{-3}=1\text{ and } \tau=(Q_3Q_2Q_1)^4=1.$$ \cite[Lem.~2.10]{BFr02}  shows adding these relations to $B_4$  is 
equivalent to adding $q_1^2q_3^{-2}=1$ to $H_4$. This produces new equations: 
\begin{equation} \label{basM4} 
q_1q_2q_1^2q_2q_1=(q_1q_2q_1)^2=(q_1q_2)^3=1.\end{equation} 
 With $\sQ =\lrang{(q_1q_2q_3)^2,q_1q_3^{-1}}$, \cite[Thm.~2.9]{BFr02} says the following. 
 \begin{edesc} \label{sQact} \item \label{sQacta}  $\sQ\norm H_4$ is the quaternion group of order 8; it contains the one nontrivial involution, $z=(q_1q_3^{-1})^2$ in $H_4$, generating its center, and acting trivially on inner Nielsen classes.  
 \item \label{sQactb}   So, $\sQ$ acts on all our Nielsen classes through $\sQ''$. 
\item  \label{sQactd}  $\bar M_4=H_4/\sQ$ acts on {\sl reduced\/} Nielsen classes as $\PSL_2(\bZ)$, making the induced $U_j$ cover a natural upper half-plane quotient.  
\end{edesc} 

 From \eql{sQact}{sQactb}, $H_4$ action on reduced Nielsen classes factors through the relation $q_1q_3^{-1}=1$ (\cite[\S3.7]{BFr02} and 
\cite[Prop.~3.28]{BFr02}).  Our actions will all be on Nielsen classes, modulo inner action. So, rather than explicating forming $H_4/\lrang{z}$ to get to its action on reduced classes, as in \eql{sQact}{sQactd}, we abuse notation slightly and refer to $\bar M_4$ in \eql{sQact}{sQactd} as $H_4/\sQ''$ acting.  

When $r=4$, fine moduli of reduced spaces divides into three conditions \cite[Prop.~4.7]{BFr02} on ${}^\dagger$ 
(${}^\abs$ or ${}^\inn$) equivalence classes. For a braid orbit, $\sO\le \ni(G,\bfC)^\dagger$, define the {\sl reduced  braid orbit\/} to be the $H_4/\sQ''$ orbit on $\sO/\sQ''=\sO^\rd$
\begin{edesc} \label{redfinemod} \item \label{redfinemod0}  Before reduction, the Hurwitz space, $\sH(G,\bfC)^\dagger$, has fine moduli (Prop.~\ref{finemodabsinn}). \item \label{redfinemoda}   {\sl b-fine moduli}\footnote{Stands for {\sl birational\/} fine moduli.}: The Klein 4-group, $K_4$, through which the reduction group, $\sQ''$ maps,  acts faithfully on $\sO$: all orbits have length 4.  
\item \label{redfinemodb} Given \eql{redfinemod}{redfinemoda}, the actions of  $\gamma_0\eqdef q_1q_2\!\! \mod \sQ''$ and $\gamma_1\eqdef q_1q_2q_1\!\!\mod \sQ''$ (the elliptic point branch cycles) on $\sO^\rd$ have no fixed points.   \end{edesc}  When there are several components (braid orbits), the conditions \eql{redfinemod}{redfinemoda} and \eql{redfinemod}{redfinemodb} may vary from component to component.  Fine moduli for  $\sO^\rd$ is equivalent to these two conditions.  

\begin{exmpl}[Not fine reduced moduli] \label{nofine} Consider $D_{\ell^{k\np1}}$ (dihedral group of order $2\cdot \ell^{k\np1}$) with $\ell$ odd, and absolute equivalence the standard degree $\ell^{k\np1}$ representation on the (unique) conjugacy class, $\C$, of involutions. \cite{Fr95} opens with showing that the compactifications of $\sH(D_{\ell^{k\np1}},\bfC_{2^4})^{\dagger,\rd}$, $\dagger =\abs$ and $\inn$ with $\bfC_{2^4}$ four repetitions of the involution class, $k\ge 0$, over $\prP^1_j$  identify with the respective modular curves $X_0(\ell^{k\np1})$ and $X_1(\ell^{k\np1})$. 

\S\ref{weilpairing}, in relating to Serre's \OIT\ program takes a related Nielsen class, $\ni((\bZ/\ell)^2\xs \bZ/2, \bfC_{2^4})$. Here are the respective genuses of covers in the absolute and inner families.
\begin{edesc} \item Points of $\sH((\bZ/\ell^{k\np1})^2\xs \bZ/2,\bfC_{2^4})^\abs$ correspond to covers of genus $\geng_\abs$: $$2((\ell^{k\np 1})^2\np \geng_\abs \nm1)=4\frac{((\ell^{k\np 1})^2-1)}2\text{ or }\geng_\abs=0.$$ 
\item Points of $\sH((\bZ/\ell^{k\np1})^2\xs \bZ/2,\bfC_{2^4})^\inn$ correspond to  covers of genus $\geng_\inn$: $$2((2\ell^{k\np 1})^2\np \geng_\abs \nm1)=4\frac{2(\ell^{k\np 1})^2}2\text{ or }\geng_\inn=1.$$\end{edesc} 

Formula \eqref{RHOrbitEq} gives a non-classical computation of the respective genuses of the reduced spaces as $j$-line covers.  The Hurwitz spaces (absolute and inner) for both families of covers have fine moduli (Prop.~\ref{finemodabsinn}), but the reduced spaces don't. For example, when the group is $D_{\ell^{k\np1}}$, there is one braid orbit containing an \HM\ rep. (Def.~\ref{HM}). Easily compute that $\sQ''$ stabilizes it. This shows the necessity of having the \HM\ rep. not be given by involutions in Lem.~\ref{K4faithful}. 

In both cases, the absolute spaces are spaces of covers, $\phi: X\to \prP^1_z$, with $X$ of genus 0, and the Galois closure a genus one curve above 1. Since the degree of $\phi$ is odd, $\ell^{k\np1}$, over any field of definition of the cover, we can take $X$ a copy of $\prP^1_z$. The map from the inner space covering to the absolute space lies over the identity on $U_j=\prP^1_j\setminus \{\infty\}$ \eql{sQact}{sQactd} making the genus 1 curve a homogenous space for an elliptic curve.  
\end{exmpl} 
 
\begin{rem}[Fine ${}^\abs$ vs fine ${}^\inn$ moduli]   \label{notfineabs} Prop.~\ref{finemodabsinn}, since $\Cen(G)\le \Cen_{S_n}(G)$, says fine absolute moduli implies fine inner moduli. 
The former holds if there is no proper group between $G$ and $G(1)$ (primitivity; implied by double transitivity) and $G$ is not cyclic of prime order. It can, though happen that $\Cen_{S_n}(G)$ is not trivial, but $\Cen(G)$ is. For example, for $T$, the regular representation of a centerless $G$, $\Cen_{S_n}(G)$ is isomorphic to $G$ with the opposed multiplication. 
\end{rem} 

\begin{rem}[Reduced fine moduli for $r\ge 5$] \label{Jr} We don't use reduced fine moduli for $r\ge 5$ in this paper. Still, for completeness, there is no group like that $\sQ''$ as in \eql{redfinemod}{redfinemoda} to worry about. That is, b-fine moduli holds automatically, assuming $\sH(G,\bfC)^\dagger$ has fine moduli. Suppose, however, $\alpha\in \PSL_2(\bC)$ has a fixed point on $U_r$.  The analog of not satisfying  \eql{redfinemod}{redfinemodb}  arises when a  point $J_0\in J_r$ is fixed by $\alpha\in \PSL_2(\bC)$ and the reduced Hurwitz space has a singular point $\bp$ above $J_0$. That corresponds to $\bg\in \ni(G,\bfC)^\dagger$ fixed by $\alpha$, with $\bp=\bp_\bg$ the corresponding cover. \end{rem}

\begin{rem}[What $\bp^\rd\in \sH(G,\bfC)^{\dagger, \rd}$ represents] \label{redtarget} Consider $\bp\in \sH(G,\bfC)^\dagger$ -- with fine moduli -- represented by $\phi: X\to \prP^1_z$. For $\bp^\rd\in  \sH(G,\bfC)^{\dagger, \rd}$, the image of $\bp$, there may be no cover $X\to \prP^1$ over the coordinates for $\bp$ representing it. \cite[Reduced Cocycle Lemma 4.11]{BFr02}   gives the precise cohomological condition for a target isomorphic to $\prP^1$ over the coordinates of $\bp$.\end{rem} 

\subsubsection{A genus formula when $r=4$} \label{genformr=4}  As usual $\sH(G,\bfC)^\dagger$ has $\dagger=\inn$ or $\abs$ Nielsen classes.  Using reduced Hurwitz spaces  compares \cite{FrV91} to \cite{GoH92} and \cite{GhT24}. For $r=4$, reduced spaces are upper half-plane quotients ramified over the expected $j$-line places, but they {\sl aren't\/} modular curves except as variants on the case $G$ is a dihedral group case.  See Ex.~\ref{nofine}. 

\begin{prob} \label{mmtc}[Main \MT\ conj.] Starting over a particular number field $K$, show high tower levels -- $\sH(G_k,\bfC)^{\inn,\rd}$, $k>>0$, have no $K$ points. For $r=4$, the explicit approach has been to use Falting's Thm.~and show the genus of all components goes up with $k$. In proving Prob.~\ref{mmtc} for $\ell= 2$, 3 and 5 when $G=A_5$, \cite{BFr02} followed this procedure. 

 \end{prob} 
 
Then, from \eqref{basM4}, respectively denote the images of $q_1q_2$, $q_1q_2q_1$ and $q_2$ in $\bar M_4$ acting on reduced Nielsen classes as 
$\gamma_0$, $\gamma_1$ and $\gamma_\infty$. These satisfy product-one (as in Def.~\ref{HNC}): 
\begin{equation} \label{prodOneM4} 
\gamma_0\gamma_1\gamma_\infty=1.\end{equation}

The upper half-plane appears as a classical ramified 
Galois cover of 
the $j$-line minus $\infty$.  The elements $\gamma_0$ and $\gamma_1$ in $\bar M_4$ generate the
local monodromy of this  cover around 0 and 1 \cite[\S4.2]{BFr02}.  

Denote $q_1q_2q_3$ as 
$\sh$, the
shift from \eql{iso}{isob}. From the above,  $\sh$ and 
$\gamma_1$ are the same in $\bar M_4$.  Denote $\prP^1_j\setminus\{\infty\}$ by $U_\infty$. Prop.~\ref{j-Line} is \cite[Prop.~4.4]{BFr02}. 
\begin{prop}[$j$-line branch cycles] \label{j-Line}  Therefore $H_4$ acts on reduced Nielsen classes 
(as used in Prop.~\ref{j-Line} given by \eql{sQact}{sQactd}) through $\bar M_4$.  
Then, $\bar M_4$ orbits on 
$\ni(G,\bfC)^{\dagger,\rd}$  correspond one-one to  $H_4$ orbits on $\ni(G,\bfC)^\dagger$.  

For $\sO'$, a reduced orbit corresponding to a Nielsen class orbit $\sO$, orbits of the cusp group, $\Cu_4$ give the cusps. Denote the  
respective  actions of 
$\bar\gamma=(\gamma_0, \gamma_1,\gamma_\infty)$ by $\bar\gamma'=(\gamma_0', \gamma_1',\gamma_\infty')$. \begin{center} Then, $\sO'$ corresponds to a cover of $\beta_{\sO'}: \sH^\rd_{\sO'}\to U_\infty$ \\ with  $\bar\gamma'$ a  branch cycle description of its compactification over $\prP^1_j$. \end{center}   
\end{prop}

\newcommand{\Stab}{\texto{Stab}}

Suppose $\{\bg,(\bg)q_2,(\bg)q_2^2,...\}$ is the orbit of $\bg=(g_1,g_2,g_3,g_4)$ under $q_2$. For $\bg^*$ any element in the orbit, the product of its 2nd and 3rd entries is always $g_2g_3=g$; denote $\ord(g)$ by $o$ (called the {\sl middle product\/} or $\mpr$). Below, denote the orbit length (or width) by $\wid(\bg)$,  of  $q_2$ on $\bg$.\footnote{That is interpreted as the ramification index of the cusp over its image in $j=\infty$.} With actual numbers in Prop.~\ref{gamOrbit} we indicate the pair $(\mpr(\bg),\wid(\bg))$   by $(u,v)$ and refer to this as its {\sl orbit type}. With the center of $\lrang{g_2,g_3}$ denoted $\Cen(g_2,g_3)$, the following is \cite[Prop.~2.17]{BFr02}. 

\begin{prop} \label{gamOrbit}  If $g_2=g_3$, then $u=v=1$. With $g_2\ne g_3$, $g=g_2g_3$ and $g'=g_3g_2$:  
\begin{edesc} \label{midprod}  \item 
$u=\ord(g_2g_3)/|\lrang{g_2g_3}\cap \Cen(g_2,g_3)|$.  Also,  $v=2\cdot u$, unless, 
\item  \label{midprodb} with $x=(g)^{(u-1)/2}$ and $y=(g')^{(u-1)/2}$ (so $g_2y=xg_2$ and  $yg_3=g_3x$),
\begin{center} $u$  is odd, and  $yg_3$ has order 2. Then,  $v=u$.\end{center}  \end{edesc}  \end{prop}

Denote a $q_2$ orbit with type $(u,v)$ by $\cO(u,v)$. For $\bg\in \cO(u,v)$, $$\text{use $\Stab_{\sQ''}(\bg)$ (resp.~$\Stab_{\sQ''}(\cO(u,v))$}$$ for the stabilizer in $\sQ''$ of $\bg$ (resp.~the subgroup of $\sQ''$ mapping $\bg$ into $\cO(u,v)$). Since $\sQ''\norm \Cu_4$, $|\Stab_{\sQ''}(\bg)|$ and $|\Stab_{\sQ''}(\cO(u,v))|$ depend only on $\cO(u,v)$. 

\begin{defn}[Reduced orbit length] \label{cardorbsh} The reduced orbit factor associated to $\cO(u,v)$ is  $$f_{u,v}=|\Stab_{\sQ''}(\cO(u,v))/\Stab_{\sQ''}(\bg)|. \text{ An $f_{u,v}\ne 1$ gives {\sl orbit shortening}.}$$ \end{defn} 

With an actual cusp computation, several $\gamma'_\infty$ orbits may have the same $(u,v)$. Use a peripheral symbol $a$ to distinguish them. Riemann-Hurwitz then 
gives the genus
$g_{\sO'}$ of the reduced Hurwitz space component $\sH_{\sO'}^\rd$ corresponding to the reduced braid orbit $\sO'$ as  
\begin{equation} \label{RHOrbitEq} 2(|\sO'|+g_{\sO'}-1)=\frac {2 (|\sO'|-\tr(\gamma_0'))}3 
+
\frac{|\sO'|-\tr(\gamma_1')} 2 + \sum_{\cO'(u,v;a)\subset \sO'}\frac{v} {f_{u,v}}-1. \end{equation} 

Lem.~\ref{K4faithful}, rephrases \cite[Lem.~7.5]{BFr02}, assuring b-fine moduli on some of our examples.

\begin{lem} \label{K4faithful} Assume $r=4$,  $G$ centerless, and $\sO$ a braid orbit in $\ni(G,\bfC)^\dagger$ containing an \HM\ rep. $\bg=(g_1,g_1^{-1}, g_2,g_2^{-1})$. Then the $K_4$ action is faithful unless $g_1$ and $g_2$ are involutions. \end{lem} 

\subsection{Moduli Definition Fields: Part I}   \label{mdfsec} For a field $F$, a variety $V$ defined over $F$, and $\bp\in V$, $F(\bp)$ denotes the field generated over $F$ by the coordinates of $\bp$.  Suppose $\sH'$ is a component of $\sH(G,\bfC)^\dagger$. Usually assuming $\sH'$ has fine moduli, we seek a field $\bQ_{\sH'}$ with the following property, 
\begin{defn}[Moduli definition field] \label{mdf}  For $\bp\in \sH'(\bar \bQ)$ there will be a representative cover $\phi^\dagger: X^\dagger\to \prP^1_z$ with equations defined over $\bQ_{\sH'}(\bp)$, and any other cover representing $\bp$ will be equivalent to $\phi^\dagger$ over some extension of $\bQ_{\sH'}(\bp)$. \end{defn}   

In lieu of Thm.~\ref{cosetbr}, \S\ref{bclmodel} improves the original Branch Cycle Lemma (\BCL) as a model for Def.~\ref{mdf}. \S\ref{configspnotmdf} (Ex.\ref{davprs}) is an explicit example that came from the solution of Davenport's problem. It shows the moduli definition field is {\sl not\/} always the definition field of the moduli space with its map to its configuration space.  S\ref{fibGalClos} deals with Galois closures of covers. 

\begin{rem} \label{mfdef} Our concentration on points on fine moduli spaces, combined with our use of Grauert-Remmert, allows a fairly uniform approach. There are, however, places where one must pause. \begin{edesc} \label{2moddefs} \item \label{2moddefsa} We sometimes, as in \S\ref{weilpairing}, use spaces that don't have fine moduli (in going from Hurwitz spaces to reduced Hurwitz spaces); and 
\item \label{2moddefsb} in comparing points on a \MT\ with points on a Jacobian variety, as in \S\ref{cmabelvar}, on Shimura-Taniyama \CM\ varieties, the definitions of moduli fields aren't tranparently compatible.\end{edesc} Using remarks in \S\ref{heiscase}, our approach works because we selected limited examples to apply Thm.~\ref{cosetbr}\end{rem} 

\subsubsection{The \BCL\ as a model} \label{bclmodel} 
Denote the least common multiple of elements in $\bfC$ by $N_\bfC$. Recall the elements $\Aut(G,\bfC)$ preserving $\bfC$, and the corresponding subgroup of $N_{S_n}(G,\bfC)$ (\S\ref{statements}). 
 \begin{defn} \label{ratclass} With  $\zeta_{N_{\bfC}}$ a primitive $N_{\bfC}$th root of 1, 
Consider these subgroups of $G(\bQ(\zeta_{N_{\bfC}})/\bQ)$: $$\begin{array}{c} M_{\bfC,\inn}\eqdef \{u\in \bZ/N_{\bfC}\mid (u,N_\bfC)=1 \text{ and }\bfC^u=\bfC\} \\ \text{ and }
M_{\bfC,\abs}\eqdef \{u\in \bZ/N_{\bfC}\mid (u,N_\bfC)=1 \text{ and }\bfC^u=\bfC \mod N_{S_n}(G,\bfC)\}. \end{array} $$ 
We say $\bfC$ is a {\sl rational union\/} if $M_{\bfC,\inn}=(\bZ/N_{\bfC})^*$. \end{defn}

Assuming fine moduli, Cor.~\ref{GQaction} says $\sigma\in G_\bQ$ maps a representative of $\bp\in \sH(G,\bfC)^\dagger(\bar\bQ)$ to a representative of $\bp^\sigma\in \sH(G,\bfC^{n_\sigma})^\dagger$ with $n_\sigma$ the cyclotomic integer associated to $\sigma$. Compatible with Prop.~\ref{bcl}, Def.~\ref{mdfvar} is a variant on Def.~\ref{mdf}. 
\begin{defn} \label{mdfvar} Replace  $\sH(G,\bfC)^\dagger$ (with fine moduli) with  an absolutely irreducible component, $\sH'$. Its moduli definition field, $\bQ_{\sH'}$, give the minimal subfield (of $\bar \bQ$) satisfying \eqref{moddefprop}. \end{defn} 
\begin{edesc} \label{moddefprop} \item \label{moddefpropa} The fiber over $\bp'\times \prP^1_x$ in the unique total representing family $\Psi_{\sH'}: \sT\to \sH'\times \prP^1_x$ gives a cover representing $\bp'$ over $\bQ_{\sH'}(\bp')$.
\item  \label{moddefpropb}  Applying $\sigma\in G_{\bQ}$ to $\phi_{\bp'}: X_{\bp'} \to \prP^1_x$ -- giving $\phi_{\bp'}^\sigma: X_{\bp'}^\sigma \to \prP^1_x$ -- represents a cover in $\sH'$ if and only if $\sigma\in G_{\bQ_{\sH'}}$.\end{edesc} 

Assuming fine moduli and {\sl irreducibility\/} for the Hurwitz space $\sH(G,\bfC)^\dagger$, the {\sl Branch Cycle Lemma\/} (\BCL\ of \cite[\S5.1]{Fr77}) gives 
\begin{equation} \label{oldbcl} \begin{array}{c}\text{the moduli definition field for $\sH(G,\bfC)^\dagger$ is   the fixed field of $M_{\bfC,\dagger}$ in Def.~\ref{ratclass},} \\ 
\text{an explicit cyclotomic field,  depending only on $\bfC$ and the equivalence $\dagger$.} \end{array} \end{equation} 

In lieu of Thm.~\ref{cosetbr}, we don't need $\sH(G,\bfC)^K$ to be irreducible. Replace that by \eqref{pullbackabsirr}.    
\begin{edesc} \label{pullbackabsirr} \item Assume we know $\bQ_{\sH'}$, with $\sH'\le \sH(G,\bfC)^\abs$, classes of covers that are an orbit for lift invariants under $N_{S_n}(G,\bfC)$. 
\item For the inner Hurwitz space: Replace absolute irreducibility of $\sH(G,\bfC)^\inn$ by there is one absolutely irreducible component of $\sH\le \sH(G,\bfC)^\inn$ above $\sH'$.
\end{edesc} 

\begin{prop}[extended \BCL] \label{bcl} Assume \eqref{pullbackabsirr} holds. Then $\bQ_{\sH}$ exists and is $\bQ_{\sH'}$ with the fixed field, $F_{\bfC,\inn}$, of $M_{\bfC,\inn}$ adjoined.\footnote{When I left my tenured position at Stony Brook in 1974 for a full professorship at UC Irvine, I was given the least experienced typist. That typist couldn't produce a copy of \cite{Fr77} that looked good photocopied into print. It could use a redo for typos and updates for placement on the arXiv and the proof of this extension.} 

 \end{prop}  

\cite[Lem.~2]{FrV91} showed how to replace $(G,\bfC)$ -- when these did not have fine moduli -- with an explicit, not canonical -- \lq\lq covering\rq\rq\ $(G^*,\bfC^*)$ with fine moduli, sufficing for some applications. 

\subsubsection{Example moduli definition field} \label{configspnotmdf} 

The \BCL\ arose in solving Davenport's problem \cite{Fr73}. Ex.~\ref{davprs} explicitly displays a moduli definition field  that is {\sl not\/} the definition field of the Hurwitz space component over the configuration space. Davenport's problem was $L=\bQ$ in Prob.~\ref{davprob}. 
\begin{prob} \label{davprob} Describe all pairs of (inequivalent) genus 0 indecomposable covers $\phi_j: X_j\to \prP^1_z$, $j=1,2$ covers, with total ramification over $\infty$ (polynomial covers), defined over a number field $L$, for which a certain arithmetic property holds.\footnote{For almost all primes of $L$, the covers have identical ranges on the residue class fields. This turned out to be equivalent to Schinzel's problem: Among polynomials pairs $f_1,f_2$, with $f_1$ indecomposable, find those for which $f_1(x)-f_2(y)$ is (nontrivially) reducible.} \end{prob} 
\begin{exmpl}[Davenport pairs] \label{davprs}  Let $T_j$ be the representation of $\phi_j$ in Prob.~\ref{davprob}. \cite{Fr73} gives this corollary of \cite{Fr70}:  indecomposable over $\bC$ for the polynomials in Prob.~\ref{davprob} is the same as indecomposable over $L$,  the Galois closures of the two covers are the same \cite[Prop.~3]{Fr73}, and $\tr(T_1(g))=\tr(T_2(g))$ for $g\in G$ ($\tr$ the trace) is equivalent to the arithmetic statement.  

The latter implies $\deg(T_1)=\deg(T_2)\eqdef n$. We now see the representations are doubly transitive \cite[Lem.~2]{Fr73}.  These being inequivalent polynomial covers implies only one class in ${}_j\bfC$ is an $n$-cycle, and covers have at {\sl most\/} 3 finite places  that are ramified \cite[Thm.~1]{Fr73}. Denote one conjugacy class of $n$-cycles by $\C_\infty$, and ${}_j\C_\infty$ the resp.~$n$-cycle classes for $T_j$, $j=1,2$. 

A classical theory -- {\sl difference sets\/} -- suited the branch cycle lemma implying all other $n$-cycle classes have the form $\C_\infty^u$, $(u,n)=1$ and the classes ${}_j\bfC$, $j=1,2$ differed only in their $n$-cycles. Finding those $u$ values was the hard group theory. 

One more general conclusion. \cite[Lem 5]{Fr73}: The cyclotomic field given by the \BCL\ for the moduli definition field of these covers is the fixed field of \eql{multiplier}{multipliera}. \begin{edesc} \label{multiplier} \item \label{multipliera} $\sM_{G,{}_j\bfC}\eqdef \{u\in (\bZ/n)^*\mid {}_j\C_\infty^u={}_j\C_\infty\}.$ Further,  $-1\not \in \sM_{G,{}_j\bfC}$ and ${}_1\C_\infty^{-1}={}_2\C_\infty$. 
\item \label{multiplierb} An $\alpha\in \Aut(G)$, as in Rem.~\ref{autononpres}, maps ${}_1\bfC$ to ${}_2\bfC$; the argument of Cor.~\ref{autocomps} applies. 
\item \label{multiplierc} From \eql{multiplier}{multiplierb},  Hurwitz spaces for $T_1$ and $T_2$ are equivalent covers of $U_4$. 
\end{edesc} 
From \eql{multiplier}{multipliera}, the moduli definition field here is not $\bQ$, giving the result -- no such polynomial pairs -- over $\bQ$ that Davenport expected.  For general number fields $L$, work explicitly with Nielsen classes by noting these gave $G$ closely related to $\PGL_k(\bF_q)$. The two different permutation representations are on points and hyperplanes of projective space.\footnote{There was also an exceptional degree 11 case. This was before the classification of finite simple groups, but eventually, it was shown that these were all cases.}   \cite[\S5]{Fr12} lists possible Nielsen classes and outcomes from these calculations:  

\begin{edesc} \label{concdav}
\item There is just one braid orbit on either Nielsen class, and the cyclotomic definition field from the \BCL\ is given explicitly. 
\item There were only finitely many corresponding Nielsen classes (or degrees), and so only finitely many Davenport polynomial pairs, no matter what is $L$.\footnote{Again, this uses that polynomials give genus 0 covers.} 
\item \label{concdavc} Those with $r=4$ correspond to degrees $n=7, 13$ and 15. 
 \item \label{concdavd}  Reduced $j$-line covers from \eql{concdav}{concdavc} have genus 0 (from \eqref{RHOrbitEq} \cite[\S6]{Fr12}).  \end{edesc} 
The punchline: From \eql{concdav}{concdavd}, the Hurwitz spaces as $j$-line covers are explicitly isomorphic to $\prP^1$ over $\bQ$, so $\bQ$ points are dense in the Hurwitz spaces of \eql{concdav}{concdavd}.  You {\sl must\/} adjoin the moduli definition field ($\not= \bQ$) to get actual polynomial pairs:    \cite[\S 9.2]{Fr99} for a complete exposition. 

Explicit {\bf PARI} generated equations of  \cite[\S5.4]{CaCo99} display the essential parameter; \cite[\S7.2.2]{Fr12} notes their dependence on  \cite{Fr73}.  \cite[Thm.~6.9]{Fr12}. shows all this with reduced spaces of $r=4$ branch point covers as a case of genus formula \eqref{RHOrbitEq}. 
\end{exmpl}  

\subsubsection{Galois closure} \label{fibGalClos}   Consider the extension of constants diagram \eqref{extconstants} from $F(V)/F(W)$, an absolutely irreducible extension defined over $F$, as coming geometrically from 
$\Psi: V\to W$, a finite flat, degree $n$, morphism of normal, absolutely irreducible varieties
over a field $F$. Below, we use the permutation representation $T_\Psi$ attached to $\Psi$. 

Construct the $n$-fold fiber product of $\Psi$:   \begin{equation} \label{fiberprod} \{(\row v n)\in V^n\mid
\Psi(v_i)=\Psi(v_j)\}=V'.\end{equation}  As in the proof of Lem.~\ref{homcovers}, remove the fat diagonal and normalize what remains of $V'$  to form $\bar \Psi: \bar  V\to W$. 
Take a base point $w\in W(\bar\bQ)$ with no singular points of $\bar  V$ over it and $F(w)$ and $\hat F$ disjoint fields over $F$.\footnote{Apply Hilbert's irreducibility to a projection of $W\to U$ defined over $F$,  with $U$ an open subset of some projective space, to get such a point $w$, lying over a point of $U(\bQ)$.}   

A $\pi'\in S_n$ maps $(\row v n)=\bv_1\mapsto \bv_{\pi'}\eqdef (v_{(1)\pi'},\dots, v_{(n)\pi'})$, inducing $\pi': \bar V \to \bar V$ permutating (absolutely) irreducible components and determined by what it does on elements in $\bar \Psi^{-1}(w)$. Therefore, below, we assume $\bv_{\pi'}$ is in the fiber over $w$. We only have to go up to $\hat F$ to get coefficients for the equations of absolutely irreducible components.

\begin{lem}  \label{compcosets} Assume, as above, that $\Psi$ is defined over $F$. Take $\bar V_1$, an absolutely irreducible component of $\bar V$. For $\bv_1\in \bar V_1$, identify $G$ with $G_1\eqdef \{g\in S_n \mid \bv_g\in \bar V_1\}$. For $\pi'\in S_n$, representing a {\sl right\/} coset of $G$ in $S_n$, consider \begin{equation} \label{defVpi'} \bar V_{\pi'}=\{\bv_{g\pi'}=\bv_{\pi'\cdot (\pi')^{-1}\g\pi'}\mid \bv_g\in \bar V_1\}. \end{equation}  
\begin{edesc}  \label{Snact} \item \label{Snacta}  Each $\bar V_{\pi'}$ is a Galois closure of $\Psi$ over $\hat F$ with group $G_{\pi'}=\{(\pi')^{-1}g\pi'\in S_n\mid g \in G\}$.  
\item \label{Snactb}  The $\pi'$ for which  $G_{\pi'}=G$ are those cosets represented by $\pi'\in N_{S_n}(G)$. \eql{Snact}{Snacta} gives the quotients of $G_{\pi'}$ that factor through a particular copy of $V$ in the symmetric product. 
\item  \label{Snactc}   The resulting $G_{\pi'}$, the group of the fiber over $w$ of $\bar V_{\pi'}$, is independent of the choice of $\bv_1\in \bar V_1$; it depends only on the coset of $\pi'$. 
\item \label{Snactd} The set of Galois closure components  over $\bar F$ that lie on $\bar V$ that factor through $\Psi$ is closed under the action of $G_{\bar F}$ acting on these components through the group $N_{S_n}(G)$.  \end{edesc} 
Each $G_F$ orbit in \eql{Snact}{Snactd} is represented by a subgroup of $N_{S_n}(G)/G$.  
 \end{lem}
 
 \begin{proof} Proof of \eql{Snact}{Snacta} From \eqref{defVpi'},  the elements $\psi_{\pi'}: g\mapsto (\pi')^{-1}g\pi'$ map  $\bar V_1$ to the elements in $\bar V_{\pi'}$. Thus $G_{\pi'}$ maps $\bar V_{\pi'}$ into $\bar V_{\pi'}$. This component has $|G_{\pi'}|$ elements and is Galois over $W$ from Galois theory: It has precisely as many automorphisms as the degree of the cover over $W$. Also, $\psi_{\pi'}$  maps the subgroups $G(i)<G$ defining the permutation represention $T=T_1$ to  subgroups $G_{\pi'}(i)$ defining the permutation representation $T_{\pi'}$. 
 
 The first sentence of \eql{Snact}{Snactb} is obvious from the definitions; they define the elements in $S_n$ that normalize $G$. Now, $\bar V_{\pi'}$ maps through $V$ if and only if $\bar V_{\pi'}$ has a quotient $V'$ whose fiber over $w$ is the same as the fiber of $V=V_1$ over $w$. Then  \eql{Snact}{Snacta} shows this happens if only if $\pi'$ is a coset of $N_{S_n}(G)$ in $S_n$. This shows \eql{Snact}{Snactb}.  

For \eql{Snact}{Snactc}, replace $\pi'$ by $g\pi'$ with $g\in G$. Then, $\bv_{\pi'}\mapsto (\bv_g)\pi'$, changing $\bv_1$ to $\bv_g$ in the fiber of $\bar V_1$ over $w$, and $G_{\pi'} \mapsto (\pi')^{-1}(g^{-1}Gg)\pi'=G_{\pi'}$. 

Each absolutely irreducible component of $\bar V$ is determined, as an algebraic set, by its fibers over $w$ indicated by the coset of $G$ in $S_n$ defining it, and $\bar V$ is defined over $F$. Since $\Psi$ is defined over $F$, any conjugate of a $\bar V$ component that factors through $\Psi$ also factors through $\Psi$. This shows  \eql{Snact}{Snactd}, and since the action of $G_F$ will factor through the decomposition group of the collection of components, this also shows the last sentence of the lemma. 
 \end{proof} 

According to \eql{Snact}{Snactd} (and the following sentence), we  can divide the components of $\bar V$ that factor through $\Psi$ into $F$-components. We regard Prop.~\ref{weiluse} as a precise version of \HIT. With no loss, assume there is one $F$-component, denoted $\sA(\bar V_1)$, on which $G_F$ acts through a  transitive permutation representation $T^*$ on $\sA$. Prop.~\ref{weiluse} applies the Weil co-cycle condition; we are not after just the definition field of a variety but the definition field of a Galois cover. The conclusion says coefficients of the components generate the constants in the Galois closure of $\Psi$ in $\sA(\bar V_1)$. 

\begin{prop} \label{weiluse} Assume $G$ is centerless. With the assumption above, we may assume $T^*$ is faithful on the collection of Galois closures of $V\to \prP^1_z$ that factor through $\Psi$. Therefore, if $G_F$ is fixed on the unique equivalence classes of covers,  $G$ is regularly realized as a Galois group over $F$. \end{prop} 

\begin{proof} Consider the (normal) subgroup, $G^*$, of $G(\hat F/F)$ that leaves each element of $\sA(\bar V_1)$ fixed (as an algebraic set). The fixed field, $F^*$,  of $G^*$ in $\hat F$ is Galois over $F$. With the centerless assumption, \cite[Prop.~2]{Fr77} shows there is an algebraic set $V_1^*$ such that $V_1^*\otimes \hat F$ is $\bar V_1$, with $V_1^*$ defined over $F^*$. Now apply $G(F_1^*/F)$ to transport $V_1^*$ to the algebraic sets of the other $\sA(\bar V_1)$ components. It is clear now that $G(F^*/F)$ has faithful action on these transported components. \end{proof} 

Def.~\ref{normcomps} gives the fiber product components analogous for one cover of $\prP^1_z$ of the components of Hurwitz space components of $\sH(G,\bfC)^\inn$ that lie over a given component of $\sH(G,\bfC)^\abs$.\footnote{The point: With a number theory tool like HIT, deal with a Nielsen class rather than one cover at a time. }

\begin{defn}[Normalizer components] \label{normcomps} Denote the union of the components associated with the cosets of $G$ in $N_{S_n}(G)$ as in \eql{Snact}{Snactb} by  $\sA(N_{S_n}(G))$. This contains $\sA(\bar V_1)$ from above Prop.~\ref{weiluse}.  \end{defn}

\subsection{Moduli Definition fields: Part II}  \label{HITNC}  Start from an absolute Nielsen class $\ni(G,\bfC)^\abs$. We run over components of $\sH(G,\bfC)^\inn$ using \S\ref{fibGalClos}.  \S\ref{HSfiberprod} is the Galois closure fiber product construction applied to  Hurwitz spaces. This produces the moduli definition field for an inner component from the moduli definition of an absolute component below it using the division of Thm.~\ref{cosetbr} into homomorphism and automorphism-separated components, and {\sl Weil's cocycle condition\/} applied to (Galois covers) inner moduli. 
\S\ref{fineinn-notfiineabs} assumes only fine inner moduli,  not fine absolute moduli.

\subsubsection{Fiber product applied to Hurwitz spaces} \label{HSfiberprod} 
Suppose $\sH'$ (resp.~$\sH$) is a component on a particular Hurwitz space, $\sH(G,\bfC)^\abs$ (resp.~$\sH(G,\bfC)^\inn$) with $\sH$ lying over $\sH'$,  $\bQ_{\sH'}$ and $\bQ_{\sH}$ the respective moduli definition fields (Def.~\ref{mdf}).\footnote{Compatible with Thm.~\ref{cosetbr} we assume $\bQ_{\sH'}$ has been computed from the \BCL\ or information on homeomorphism-separated components.}   Prop.~\ref{inn-absmdf} does the Galois closure construction in families that allows relating $\bQ_{\sH}$ to $\bQ_{\sH'}$ in Cor.~\ref{corinn-absmdf}. 

This assumes   $\sH(G,\bfC)^\abs$ has fine moduli (the self-normalizing condition for $G(1)$ in $G$ of Prop.~\ref{finemodabsinn}). 
By assumption there is a unique total family, or {\sl fine moduli\/} structure, defined over $\bar \bQ$: \begin{equation} \label{absequat} \Psi_\abs: \sT^{\abs}\to \sH(G,\bfC)^{\abs}\times \prP^1_z\to U_r\times \prP^1_z\end{equation} on $\sH(G,\bfC)^\abs$.  Pullback over $\bp'\times \prP^1_z$ represents the cover  class associated to $\bp'\in \sH(G,\bfC)^{\abs}(\bar \bQ)$. 

\begin{prop} \label{inn-absmdf} A canonical fiber product construction gives the following commutative  diagram  
\begin{equation}\label{innabsdiagram} \begin{array}{ccccc} 
\sT^\inn &\longmapright{\Psi_\inn} {50}  &   \sH(G,\bfC)^\inn\times \prP^1_z  &
 \longmapright{} {50} &  U_r\times\prP^1_z\\
\mapdown{\Psi_{\abs,\inn}}&&  \mapdown{\Phi_{\abs,\inn}\times \Id_z}&& \mapdown{\Id_r\times \Id_z} \\
\sT^\abs & \longmapright {\Psi_\abs} {50}& \sH(G,\bfC)^\abs\times \prP^1_z & \longmapright{} {50}  &U_r\times\prP^1_z.
\end{array} \end{equation}

In  \eqref{inn-absmdf}, $\sH'$   is a component of  $\sH(G,\bfC)^\abs$ and $\sH(G,\bfC)^\inn_{\sH"}=\sH^*_{\sH'} \le \sH(G,\bfC)^\inn$  in the top line is the pullback of $\sH'$ to $\sH(G,\bfC)^\inn$.   
\begin{equation} \label{sH'equat} \rest(\Psi_\abs): \sT^*_{\sH'} \to \sH^*_{\sH'}\times \prP^1_z\to \sH'\times \prP^1_z\to U_r\times \prP^1_z \text{ defined over $\bQ_{\sH'}$.}\end{equation} 

The space $ \sH^*_{\sH'}$ may have several connected components, all conjugate under a  $N_{S_n}(G)/G$ action. Given one of these, $\sH$, the pullback of $\rest(\Psi_\abs)$ over it  gives the following diagram:
\begin{equation}\label{sHdiagram} \begin{array}{ccccc} 
\sT_\sH&\longmapright{\rest(\hat \Psi\otimes \bQ_{\sH'})} {60}  &  \sH\times \prP^1_z  &
 \longmapright{} {50} &  U_r\times\prP^1_z\\
\mapdown{\Psi_{\abs,\inn}}&&  \mapdown{\Phi_{\abs,\inn}\times \Id_z}&& \mapdown{\Id_r\times \Id_z} \\
\sT' & \longmapright {\rest(\Psi_\abs)} {60}& \sH'\times \prP^1_z & \longmapright{} {50}  &U_r\times\prP^1_z. 
\end{array} \end{equation}
The definition field of the \eqref{sHdiagram} upper row is the $\sH$ moduli definition field, $\bQ_{\sH}$. For $\bp\in \sH(\bar \bQ)$ over $\bp'\in \sH'$, the  fiber of $\sT_{\sH}$ over $\bP\times \prP^1_z$ is a Galois closure, over $\bQ_{\sH}(\bp)$, of $X_{\bp'} \to \prP^1_z$.  
\end{prop} 

\begin{proof}  Apply the fiber product Galois closure construction to the diagram of \eqref{absequat}: $V\mapsto \sT^{\abs}$, $W \mapsto \sH^\abs\times \prP^1_z$.  Then, $\sH^\inn$ is the normalization of the integral closure of $\sH^\abs$ in the function field of the resulting $\widehat {\sT^\abs}$. As in \cite[\S3.1.3]{BFr02}, check on the fibers of $\widehat{\sT^\abs}_{\bp'}\to \bp'\times \prP^1_z\subset \sH^\abs\times \prP^1_z$, with (possibly) several components, each a geometric Galois closure of $\sT^{\abs}_{\bp'}\to \bp'\times \prP^1_z$ satisfying:  

\begin{edesc} \label{fiberGalclos} \item \label{fiberGalclosa}  It represents  forming the Galois closure construction on $\sT^\abs_{\bp'}\to \bp'\times \prP^1_z$. \item \label{fiberGalclosb}  Restrict  $S_n$ to a component as in \eql{Snact}{Snactb}; this gives $h: G\to \Aut(\widehat
{\sT^\abs}_{\bp'}/\prP^1_z))$ in the inner Nielsen class $\ni(G,\bfC)^\inn$, an isomorphism between $G$ and the 
group of the cover. \item Mapping between inner and absolute spaces takes $$\bp \text{ to }
\bp'=\Phi_{\inn,\abs}(\bp)\text{ with }
\bz=\Phi_\abs\circ\Phi_{\abs,\inn}(\bp).$$\end{edesc} 

The argument that $\bQ_\sH$ has the moduli definition field property is that if we take the Galois closure construction over $\bQ_{\sH'}(\bp')$ that -- using fine moduli for $\sH(G,\bfC)^\inn$ --   $\hat \sT_{\bp} \to \bp'\times \prP^1_z$ represents $\bp$ over $\bQ_{\sH}(\bp)$. The argument uses Weil's cocycle condition exactly in Prop.~\ref{weiluse}.
\end{proof}

Reminder:  Def.~\ref{braidable} defines {\sl braiding $\alpha\in N_{S_n}(G,\bfC)$}. Cor.~\ref{corinn-absmdf} elaborates on the HIT aspects of Prop.~\ref{inn-absmdf}. Expression \eql{extconstdef+}{extconstdef+c} is the extension of constants  for the  Galois closure over $\bQ_{\sH'}(\bp')$ given by the cover $X_{\bp'}\to \prP^1_z$ for  $\bp'\in \sH'$. 
\begin{cor} \label{corinn-absmdf} Consider a pair  $(\sH',\sH)$ as in \eqref{sHdiagram}. Then,  \begin{edesc} \label{extconstdef+}  \item \label{extconstdef+a} $\rest(\Phi_{\abs,\inn}): \sH\to \sH'\otimes \bQ_{\sH}$ is a geometrically irreducible Galois cover with group $$\{h\in N_{S_n}(G,\bfC)/G\mid h \text{ is  braidable}\}.$$ 

\item \label{extconstdef+b} $\rest(\Phi_{\abs,\inn}): \sH \to \sH'$ is a $\bQ_{\sH'}$ irreducible and Galois cover with group a subgroup of $N_{S_n}(G,\bfC)/G$. 
\item \label{extconstdef+c}  The group of \eql{extconstdef+}{extconstdef+a} is normal in  the group of \eql{extconstdef+}{extconstdef+b} with quotient group $G(\bQ_{\sH}/\bQ_{\sH'})$. 
\item \label{extconstdef+d} For $\bp\in \sH $ over $\bp'\in \sH'(\bar \bQ)$, $G(\bQ_{\sH}(\bp)/\bQ_{\sH'}(\bp'))\le N_{S_n}(G,\bfC)/G$.\end{edesc}  

Restrict $\bp'$ to points with images in $U_r(\bQ)$. Then, the intersection of all the corresponding decomposition fields $\bQ_{\sH'}(\bp')$ (resp.~$\bQ_{\sH}(\bp)$) is $\bQ_{\sH}$. \end{cor} 

\begin{proof} Proof of \eql{extconstdef+}{extconstdef+a}: Select a base point, $\bz^*\in U_r$ and classical generators, $\sP_{\bz^*}$ based at $\bz^*$ (\S\ref{classgens}). Then, each $\bp'\in \sH'$ corresponds to an element $\bg'\in \ni(G,\bfC)^\abs$ lying below some $\bg\in \ni(G,\bfC)^\inn$. Here is the set of $\bg^*$  above $\bg'$: 
$$\{\bg^*=h\bg h^{-1}\mid h\in N'\le N_{S_n}(G,\bfC)/G \text{ with $\bg=(\bg^*)q$ for some $q\in H_r$}\}.$$  That is, $N'$ consists of those $h\in  N_{S_n}(G,\bfC)/G$ for which conjugating by $h$ is braidable.

Proof/explanation of \eql{extconstdef+}{extconstdef+b}  and \eql{extconstdef+}{extconstdef+c}: Suppose $\bar\sigma\in  G(\bQ_{\sH}/\bQ_{\sH'})$ is the image of $\sigma\in G_{\bQ_{\sH'}}$ and $\bp'\in \sH'$ corresponds to a cover in the absolute Nielsen class with $\bp\in \sH$ lying above it. Then, $\sigma$ extends to an action on $\bp$, and on the whole Galois closure construction of \eqref{fiberGalclos}. The result is that $(\bp)^\sigma$ is a cover representing a point in $(\sH)^\sigma$ lying above $\bp'$, inducing the action of $\bar \sigma$ on $\bQ_{\sH}$. This gives the homomorphisms of \eql{extconstdef+}{extconstdef+b}  and \eql{extconstdef+}{extconstdef+c}. The statement of \eql{extconstdef+}{extconstdef+d} therefore interprets as saying a decomposition group is a subgroup of the Galois group of a cover. 

Finally,  consider the last statement of the Cor. From \eql{extconstdef+}{extconstdef+c}, every decomposition field contains $\bQ_{\sH}$. We want to show that for any proper field extension $L/\bQ_{\sH}$, there is a $\bp$ lying over a point of $U_r(\bQ)$ for which $\bQ(\sH)(\bp)$ is disjoint over $\bQ_{\sH}$ from $L$. From the Bertini-Noether reduction \cite[Prop.~10.4.2]{FrJ86}${}_2$ we may reduce to  a dimension 1 version of the situation. Simplify notation and take $K'=\bQ_{\sH'}$ (resp.~$K=\bQ_{\sH}$). 

This gives a sequence of covers \begin{equation} W^*\longmapright {\phi_{W^*}} {20}  W \longmapright{\phi_W} {30} X\longmapright {\phi_X} {15} \prP^1_x\end{equation} with $\phi_X$ an absolutely irreducible cover defined over $K'$, $\phi_{W^*}$ an injection, and the composite $f=\phi_X\circ \phi_W\circ \phi_{W^*}: W^* \to \prP^1_z$ absolutely irreducible, Galois with group $G$, and defined over $K$. To finish, find $x'\in \prP^1_x(\bQ)$ such that for any $w^*\in W^*$ over $x'$, $K(w^*)$ is disjoint from $L/K$. Hilbert's irreducibility theorem says there are infinitely many such $x'\in \bQ$ with $[K(w^*): K]=\deg(f)$. To include the disjointness condition, replace $K$ with $L\cdot K$. Taking the intersection of these $K(w^*)$ fields over $\bQ_{\sH'}$ has the fields $\sQ_{\sH}$  as their  common subfield. 
\end{proof}

\begin{rem}[Applying Thm.~\ref{cosetbr}]  \label{appcosetbr} Assume Schur-Separation property \eqref{schursep} holds.  Apply the generators of $H_r$ ($\sh$ and $q_2$) to $\ni(G,\bfC)^\inn$ to compute the complete braid orbit $\sO_\bg$ of some $\bg\in \ni(G,\bfC)^\inn$ with a particular lift invariant $s_\bg$.   Check those $\alpha\in N_{S_n}(G,\bfC)$ that appear as $(\bg)\alpha\in \sO_\bg$, denoting this $K_\br$. The union over coset representatives, $(K:K_\br)$, of elements in $\sO_\bg$ give the braid orbits on $\ni(G,\bfC)^\inn$ that lie over the image of $\sO_\bg$ in $\ni(G,\bfC)^\abs$.  Now form the corresponding braid orbits running over a list of lift invariant representatives.  The result is a Nielsen class list of all absolute and inner components. \end{rem} 

\subsubsection{Fine inner, but not fine absolute, moduli} \label{fineinn-notfiineabs} Prop.~\ref{withoutfineabs}  states an extension of  Prop.~\ref{inn-absmdf} when $\sH(G,\bfC)^\inn$ has fine moduli ($G$ has no center) but $\sH(G,\bfC)^\abs$ may not.  
Showing the nature of $\bQ_{\sH'}$  (assume $\sH'$ is absolutely irreducible) when it is not given by \BCL\ Prop.~\ref{bcl} is our main goal. We don't give an explicit proof, but note that works similarly except using the stronger application of the Weil cocycle condition that is in \cite[p. 33-35]{Fr77}. 

\begin{prop} \label{withoutfineabs}  Assume $\sH(G,\bfC)^\inn$ has fine moduli, but $\sH(G,\bfC)^\abs$ may not. Also, $\sH'$ is a component of $\sH(G,\bfC)^\abs$ with $\sH\le \sH(G,\bfC)^\inn$ lying above it.  
\begin{edesc} \item Local construction of the fine absolute moduli space gives the construction of the unique fine inner total space, and therefore $\bQ_{\sH}$ for any component $\sH\le \sH(G,\bfC)^\inn$. 
\item Suppose a representing cover of $\bp'\in \sH'$ has definition field $F_{\bp'}$. Then a definition field of a representing cover for $\bp\in \sH$ lying above $\bp'$ is given by $F_\bp\cdot \bQ_{\sH}(\bp)$. 
\end{edesc} \end{prop} 

\section{Towers of Hurwitz spaces} \label{HSTowers} Abelian varieties of dim. $\geng>2$ form a higher-dimensional space than do projective non-singular curves. It is Jacobians of curves in Hurwitz families that we consider in generalizing Serre's \OIT. 

\begin{defn} \label{schottky} Describing the locus of curves on the space of Jacobians by equations about singularities of the $\theta$ divisor of the Jacobian was called {\sl the Schottky problem\/} \cite[\S IV]{Mu76}.\end{defn}  

Modular Towers (\MT s) takes a different approach, using decomposition groups in towers of Hurwitz families to detect special Jacobians and to show what they reveal about the properties of curves. This section constructs these towers and definitions related to their decomposition groups, thereby connecting the unsolved problems of Serre's \OIT\ and related decomposition groups to Hilbert's Irreducibility Theorem (\HIT).

\subsection{$\ell$-Frattini covers} \label{bkliftinv} Refer to a finite group $G$ as $\ell$-perfect if
\begin{equation} \label{lperfect} \text{$\ell | |G|$, but $G$ has no $\bZ/\ell$ quotient.}\end{equation} 
 Lift invariant Def.~\ref{liftinv} suffices with $G$ that is $\ell$-perfect and $\ell'$ conjugacy classes $\bfC$. 

\begin{defn} \label{repcover} A {\sl representation cover\/} of $G$ is a {\sl central\/}, surjective, Frattini cover $\psi_R: R\to G$ for which $\ker(\psi_R)$ is the Schur multiplier, $\SM_G$, of $G$. 
Write the (finite) abelian group $\SM_G$ as a product of its $\ell$-primary parts, $\SM_G=\prod_\ell \SM_{G,\ell}$. For each $\ell$, there is the induced $\psi_{R,\ell}: R_\ell\to G$, an $\ell$-{\sl representation\/} cover. with $\ker(\psi_\ell)$ isomorphic to $H_2(G,\bZ_\ell)=\SM_{G,\ell}$. For $G$ that is $\ell$-perfect, $\psi_{R,\ell}$ is the unique, universal {\sl central\/} $\ell$-extension of $G$.  \end{defn}

\S\ref{l'lift} does the basics on the lift-invariant which comes from $\psi_{R,\ell}$. Then \S\ref{l'lift}  expands to general Frattini covers.  \S\ref{prodMTs} uses these  to  produce Modular Towers, \MT s, that generalize towers of modular curves. It reviews types of \MT s, especially abelianized \MT s,  $\MT_{\ab}$.  
 \begin{edesc} \label{abMT} \item \MT$_{\ab}$s require only one lift invariant check to ensure nonempty \MT\ levels. 
 \item \MT$_{\ab}$s support our investigations of extending \HIT\ and comparing decomposition groups in the tower using Jacobians. 
 \item  \S\ref{weilpairing}, our addition to Serre's case, is the abelianized case. S\ref{heiscase} also uses the abelianized case, though both cases have full \MT s that map to the abelianized \MT s.\end{edesc}
 \S\ref{decgps}, inspired by Serre's case, introduces the two types of decomposition groups on a \MT\ for which we  have some precise understanding: {\sl \HIT\/} where the decomposition group is an open subgroup of maximal (equal to the decomposition group of the  \MT) and \CM\ (or \ST, Shimura-Taniyama) type, akin to most of the conjectures such as Andr\'e-Oort (which \cite{GhT24} called its main motivation). These expand on two famous David Hilbert contributions:  \begin{center} \HIT\ and the theory of {\sl complex multiplication\/} (\CM).  \end{center} The first case of $\CM$ is the explicit description of the abelian extensions of those quadratic extensions of $\bQ$ whose (archimedean) completions are $\bC$.  Asking questions about the decomposition groups of projective systems of points on a \MT\ is a direct analog of Serre's questions about decomposition groups attached to curves in a Hurwitz family based on their Jacobian varieties.   

\subsubsection{The $\ell'$-lift invariant and Frattini covers}   \label{l'lift} We simplify the lift invariant by assuming $\bfC$ consists of $\ell'$ classes. Rem.~\ref{notlperf} shows how to drop the $\ell'$ condition on $\bfC$  and comments on the $\ell$-perfect condition. 
Kernels of Frattini covers are always pronilpotent (product of $\ell$-Sylows) \cite[\S25.6-25.7]{FrJ86}$_{4}$.\footnote{\cite[\S3.2]{Fr20} has an extensive discussion of how to use universal Frattini covers.} So we profitably consider the cases of $\ell$-Frattini covers: Frattini covers with the kernel an $\ell$-group. Then, there is always a universal $\ell$-Frattini cover, $\tilde \psi_\ell: \tilde G_\ell\to G$, that factors through any $\ell$-Frattini cover.  Finally, the abelianization of these covers is given by $\tilde \psi_{\ell,\ab}: \tilde G_{\ell, \ab}= \tilde G_\ell/(\ker(\tilde \psi_\ell,\tilde \psi_\ell)$ (modding out by commutators of the kernel). Possible (nontrivial) lift invariants arise when $G$ has a (nontrivial) central Frattini cover $\psi_R: R\to G$, as in our \S\ref{weilpairing}, \S\ref{absinnAn} and \S\ref{heiscase} examples.  

Def.~\ref{liftinv2} gives the formula (as in Def.~\ref{liftinv})  for the lift invariant in this case.  
From the $\ell'$ condition,  Schur-Zassenhaus allows interpreting $\bfC$ uniquely as classes of same order elements in $R_\ell$. The notation $\hat \bg \in \bfC\cap R_\ell$ as lying over $\bg\in \ni(G,\bfC)$ now makes sense. From the Frattini condition, $\lrang{\hat \bg}=R_\ell$, and from the central condition, $\hat g\in \C\cap H$ lying over $g\in \C$ is unique:  \begin{equation} \label{exliftinv} \hat g\hat g'\hat g^{-1}\text{ and $\widehat {gg'g^{-1}}$ have the same order 
and   lie over $gg'g^{-1}$. So they are equal.} \end{equation}  

\begin{defn}[Lift invariant]  \label{liftinv2}   For $O$ a braid orbit on $\ni(G,\bfC)^\inn$, and $\bg\in O$,  $$\text{the lift invariant is } s_\bg=s_{\psi}(O)\eqdef \prod_{i=1}^r \hat g_i.$$ Apply  braid generators $q$ (the shift or a twist) to $\bg$ and check that $(\bg)q$ has the same lift invariant. Therefore \eqref{exliftinv} is a braid invariant (as in \cite{Fr90}, \cite{Se90}, \cite{FrV91}, \cite{Fr10}). For $O'$ the braid orbit below it on $\ni(G,\bfC)^\abs$, and $\bg'\in \sO'$ below $\bg$, the  braid invariant is the $N_K$ orbit, $S_{\sO'}$ on $s_{\bg}$.
\end{defn} 

Lem.~\ref{liftinvlem} says components with different lift invariants have different moduli properties. As usual, $\dagger$ signifies inner or absolute equivalence. 
\begin{lem} \label{liftinvlem} Suppose $\phi_i: X_i\to \prP^1_z$, $i=1,2$,  are absolute covers for which  $S_{\phi_1}$ contains a lift invariant $\lambda_1$ not in $S_{\phi_2}$. With $R\to G$ the representation cover defining the invariants, take $\C_{-\lambda_1}$ the conjugacy class of $R$ defined by the element $-\lambda_1$. Then, the deformation class of $\phi_1$ is the image of a cover in $\ni(R,\bfC\cup \C_{-\lambda_1})$ while that of $\phi_2$ is not.  \end{lem} 

\begin{proof} Assume $\bg_1\in \ni(G,\bfC)^\inn$ has lift invariant $\lambda_1$, corresponding to $\phi_i\eqdef \phi_{\bg_1}$ lying above $\bg'_1\in \ni(G,\bfC)^\abs$.  Then, the lift, $\tilde \bg_1\in \bfC\cap R$, gives an element $(\tilde \bg_1, -\lambda_1)\in \ni(R,\bfC\cup C_{-\lambda_1})$. The map $\ni(R,\bfC\cup C_{-\lambda_1})\to \ni(G,\bfC)$ induced from $R\to G$ interprets at the Hurwitz space level -- from Riemann's existence Theorem -- as giving a cover $\phi_{(\tilde \bg_1,-\lambda_1)}$ of $\prP^1_z$ that factors through $\phi_{\bg}$. 

Now suppose $\phi_2$, a cover corresponding to $\bg'_2$ is homeomorphic to $\phi_1$. Then, its lift invariant is in the $N_K$ orbit $\lambda_1$, contradicting that the lift invariant is a braid (deformation) invariant. 
\end{proof}

We use $\ell$ (corresponding to $\ell$-adic representations as in \cite{Se68}) instead of $p$ for the main prime that appears in related papers.   From Def.~\ref{liftinv2}, any quotient of $\SM_G$ (Def.~\ref{repcover} --   or as generalized in Rem.~\ref{notlperf} if $(\ell,\SM_G)\not =1$) defines a lift invariant for a braid orbit on a Nielsen class $\ni(G,\bfC)$.  The full separation of Schur components may require the whole central extension, but proper quotients can give important information. Denote by $\psi_{1,0}: G_1\to G=G_0$  the maximal $\ell$-Frattini cover of $G$ with elementary $\ell$ group kernel, $M_1=\ker(\psi_{1,0})$. \cite[\S II.B]{Fr95}.\footnote{For the pure group theory see \cite{EFr80}, or any edition of \cite{FrJ86}, e.g.  \cite[\S 25.6--25.8]{FrJ86}$_4$.}
 
\begin{edesc} \label{explainmuk}  \item \label{explainmuka}  Denote the level $k\np1$ cover by  $\psi_{k,k\np1}: G_{k\np1}=G_1(G_k)\to G_k$.  The projective limit of these covers is $\tG \ell$, the universal $\ell$-Frattini cover.
\item \label{explainmukb} Denote the universal exponent $\ell$ central extension of
$G_k$  by $\mu_{k,\ell}: R_{k,\ell}^*\to G_k$ (Rem.~\ref{lum}).  
\item \label{explainmukc}  Since $\mu_{k,\ell}$ is an $\ell$-Frattini cover, $G_{k\np 1}\to G_k$ factors through $\mu_{k,\ell}$; and  
\item  $\ker(R_{k,\ell}^*\to G_k)$ is the max.~elementary $\ell$-quotient of the Schur multiplier of $G_k$.
\end{edesc} 

In contrast to the mysterious $N_K$ action on lift invariants given in Def.~\ref{liftinv2}, the first paragraph of Cor.~\ref{compmult2} gives a direct action of $N_T/G$ on the Schur multiplier. This can help describe the $N_K$ orbits of Def.~\ref{liftinv2}. Again, $G$ is $\ell$-perfect and $\hat \psi_\ell: \hat R_\ell\to G$ is the $\ell$-representation cover.  

\begin{cor} \label{compmult2}  An $\ell'$ subgroup $H\le N_K$ acts faithfully on the $\tilde G_{\ell}$, thereby producing the universal $\ell$ Frattini cover $\tilde G_\ell\xs H$ of $G\xs H$. This induces an action on $\tilde \psi_{\ell,\ab}$ and on $\ker(\mu_{\ell,k})$ in \eql{explainmuk}{explainmukb} extending its action on $G$ giving the desired $H$ action on the  lift invariant $N_K$ orbits on $\SM_{\ell}$. 

Suppose  $\alpha\in H$ and $s_\sO$ is the lift invariant of a component $\sH\le \sH(G,\bfC)^\inn$ over the component $\sH'\le \sH(G,\bfC)^\abs$. If $(s_\sO)\alpha\not= s_{\sO}$, then $\alpha$ applied to $\sH$ is a component over $\sH'$ distinct from $\sH$. 

Now suppose $\ker(\hat \psi_\ell)=\bZ/\ell^u$, with $\zeta=e^{2\pi i/\ell^u}$. Denote the moduli definition field of $\sH'$ by $K_{\sH'}$ and assume $\alpha^*\in G(K_\sH'(\zeta)/K_{\sH'})$. Then, $\alpha^*$ applied to the equations for $\sH$ gives a component $\sH^{\alpha^*}$ over $\sH'$ with lift invariant $s_{\sO}^{\alpha^*}$. \end{cor}  

\begin{proof} \cite[Prop.~25.13.2]{FrJ86}$_4$ or \cite[p.~134]{Fr95} has the first sentence of the corollary.\footnote{While this can be made primitive recursive, especially in applying it to the abelianized $\ell$-Frattini quotient, it requires ingenuity to compute this. \cite[Part B]{Fr95} can be helpful.}  Since $\ker(\hat \psi_\ell)$ is a finite group, for some $k$, $\tilde \psi_{k,\ab}$ factors through it, Conjugating by $ \alpha$ acts on  $\tilde G_{\ell,\ab}\to G$. This induces  the Frattini quotient $\alpha \hat R_\ell \alpha^{-1} \to G$. For $\hat g$ a lift of $g\in G$ to $\hat R_\ell$,  $\alpha \hat g \alpha^{-1}$ is a lift of $\alpha g\alpha^{-1}=g'\in G$. Therefore, $\alpha \hat \psi_\ell\alpha^{-1}$ is also a representation cover. From uniqueness  of the $\ell$-representation cover, this is $\hat R_\ell \to G$; with $\alpha$ acting on the kernel. 

Now consider the lift invariant $s_{\sO}$ in the second paragraph, moved by $\alpha$. The orbit $\alpha\sO\alpha^{-1}$ will have lift invariant given by the action of $\alpha$ on $s_{\sO}$. As the lift invariant is a braid invariant, these two orbits must be distinct. 

The argument of the proof of Prop.~\ref{bcl} applied to covers in the Nielsen class $\ni(\hat R_\ell, \bfC\cup C_{- s_{\sO}})$ (Lem.~\ref{liftinvlem}) gives the 3rd paragraph statement. As stated in a footnote, the notation of \cite{Fr77} needs enhancement so that it applies to the more advanced notation of this paper. 
\end{proof} 

\begin{rem} \label{ApptoGlk} The 1st and 3rd paragraphs of Cor.~\ref{compmult2} produce components by lift invariants above an absolute component $\sH'$ by different processes. The components in the 3rd paragraph must arise by conjugating by an element in $N_K$, but not necessarily by an $\ell'$ element. There are  unanswered questions here, especially if the Schur multiplier is not cyclic, as in the 3rd paragraph. Then, the \BCL\ now gives its moduli definition field, a cyclotomic extension of $K_{\sH'}$. \end{rem} 

 \subsubsection{Production of \MT s} \label{prodMTs}  
 We produce the inner Hurwitz space components for formulating generalizations of Serre's \OIT. Specifically, abelianized \MT s with  inner, $\PSL_2(\bC)$  reduced,  spaces.
 
 Denote the pro-$\ell$ completion of the fundamental group of the (compact) Riemann surface $X_\bg$ by $\pi_1(X_\bg)^{(\ell)}$. \cite[Prop.~4.15]{BFr02} produces $\sM_{\bar \bg}$, as fitting in this short exact sequence  \begin{equation} \label{psig} \ \pi_1(X_\bg)^{(\ell)}\to \sM_{\bar \bg} \longmapright{\psi_{\bg,\tilde \bg}} {30}G=G(\hat X_\bg/\prP^1_z)\end{equation} with $\bar \bg$ associated to classical generators, as in \S\ref{classgens}, mapping to $\bg$. Then, mod out by the commutators of $\ker(\sM_{\bar \bg}\to G)$ to get   $\sM_{\bar \bg,\ab}$ with $\ker(\sM_{\bar \bg,\ab}\to G)$ the profinite $\bZ_\ell$ homology of $X_\bg$. 
Extending  $\psi_{\bar\bg,\bg}: \sM_{\bar\bg}\to G$ to  ${\tG \ell}_{\ab}$ is
equivalent to extending $\sM_{\bar \bg,\ab}\to G$ to  ${\tG \ell}_{\ab}$.  

 \begin{defn}[\MT] \label{def-MT} A projective system of (nonempty) $H_r$
orbits $\pmb \sO\eqdef \{\sO_k\le \ni(G_k,\bfC)^\inn\}_{k=0}^\infty$   is a   M(odular) T(ower), with its corresponding spaces by ${\pmb \sH}={\pmb \sH}_{\pmb \sO}=\{\sH_k\}_{k \ge 0}$ -- a \MT\ on (starting at) $\ni(G,\bfC)^\inn$. Denote $\ker(G_k\to G_0=G)$ by $\ker_{k,0}$.  

The $k$th level Nielsen class for an {\sl abelianized\/} \MT\  (\MT$_\ab$) replaces $G_k$ with   
$$G_k/(\ker_{k,0},\ker_{k,0})=G_{k,\ab} \text{ \cite[Prop.~4.16]{BFr02}.}$$ Similarly:  ${\pmb \sO}_\ab =\{\sO_{k,\ab}\}_{k\ge 0}$ and ${\pmb \sH}_{{\pmb \sO}_\ab}=\{\sH_{k,\ab}\}_{k\ge 0}$ for the spaces of the corresponding abelianization. \end{defn}

\begin{defn} \label{bHthusHk} For a given value of $k$ in Def.~\ref{def-MT}, we say $\pmb \sH$ {\sl goes through\/} $\sH_k\leftrightarrow$ braid orbit $\sO_k$. Similarly, for the abelianized version. Refer to $\sO_k$ as {\sl obstructed\/} if there is no $\bg_{k\np1}\in \ni(G_{k\np1},\bfC)$ above $\bg_k$. In particular, there is no \MT\ through $\sH_k$. \end{defn}

 The limit group, $\sM_{\bar \bg,\ab}$ is an extension $\sL_{\pmb \bg}\to \sM_{\pmb \bg}\to G$, with kernel  a $\bZ_{\ell}[G]$ lattice with charactistic quotients $\sM_{\bar \bg,\ab}/\ell^{k\np1}\sM_{\bar \bg,\ab} \to \sM_{\bar \bg,\ab}/\ell^{k}\sM_{\bar \bg,\ab}=M_1$, the charactistic $\ell$-Frattini module.     

\begin{defn}[\MT\ quotient] \label{def-MTq}  A quotient, $\sM^* \to G$, of an abelianized \MT\ has an associated $\bZ_\ell[G]$ lattice tail $\sL^* =\ker(\sM^*\to G)$. Then, the $\bZ/\ell[G] $ quotients $\sL^*/\ell^{k\np1}\sL^* \to \sL^*/\ell^{k}\sL^*$ is a $\bZ/\ell[G]$ quotient, $M^*$, (independent of $k$), of $M_1$, the characteristic $\ell$-Frattini module.   \end{defn}  

Our \S\ref{examples} examples use \MT$_\ab$s. We will tend to drop the $\ab$ subscript.  For \MT s and abelianized \MT s, we also have reduced versions with their components covering (respectively) $U_r=\prP^r\setminus \Delta_r$ and $U_r/\SL_2(\bC)$.  \S\ref{weilpairing}  and  \S \ref{heiscase}, as listed in \eqref{mtcases}, use \MT\ quotients. 

\begin{exmpl}[What $M^*\,$s work in Def.~\ref{def-MTq}?]  We won't know for certain, but suppose $\one_G$ is a $\bZ/\ell[G]$ quotient of $M_1$. If this served as an $M^*$, then the corresponding quotient, $\sM_{\bar \bg}^*$, would give an infinite tail on a Schur multiplier quotient for $G$. That is an impossibility.  \end{exmpl}

Princ.~\ref{liftinvprinc} gives the condition for the existence of a \MT, guaranteeing, under a lift invariant condition, that we have nontrivial Nielsen classes. 

Denote the projective limit of all $G_{k,\ab}\,$s by
$\tG \ell/(\ker_0,\ker_0)={\tG \ell}_{\ab}$. Though $G_{1,\ab}=G_{1}$, for $k>1$ the natural map $G_k\to G_{k,\ab}$ has (known) degree  
1 if and only if 
\begin{center} $\dim_{\bZ/p}\ker(G_1\to G)= 1$ $\Leftrightarrow$  $G_0$ is $\ell$ super-solvable \cite[\S5.7]{BFr02}. \end{center}

Prop.~\ref{existProp} addresses, for a component $\sH$ of $\sH(G_k,\bfC)$, when it obstructs a  \MT. Allude to statements on \MT s interchangeably by reference to braid orbits (always assumed nonempty) or spaces. The more elementary parts of Prop.~\ref{existProp} are on subquotients of $M_{k\np1}=\ker(G_{k\np1}\to G_k)$ in which the irreducibles consist only of the trivial module, $\one_{G_k}=\one_G$ (Rem.~\ref{1Gs}). 

\begin{defn}[Loewy Path]  \label{irrpath} A {\sl Loewy path\/} through the indecomposable  module $M_{k\np1}$ consists of a string of irreducible $G_k$ modules $\bar M_u\to \bar M_{u\nm1}\to \cdots \bar M_1$ with $\bar M_i$ in Lowy layer $i$, where $\bar M_{i\np1}\to \bar M_i$ denotes an indecomposable $G_k$ subquotient of $M_{k\np1}$. See Ex.~\ref{minloewy}.  \cite[Lem.~2.4]{FrK97}. \end{defn} 

\begin{exmpl}[Loewy Layer] \label{minloewy}  In Def.~\ref{irrpath} the symbol $\bar M_{i\np1}\to \bar M_i$ for an indecomposable $G_k$ module $M$ means $\bar M_i$ is a quotient, and $\bar M_{i\np1}$ is the kernel of $M\to \bar M_i$. The case where $\bar M_{i\np1}$ and $\bar M_{i}$ are the trivial module is given by the small Heisenberg group \eqref{smheis}. \end{exmpl}

In \eql{steppingup}{steppingupa}, Prop.~\ref{existProp} explains existence of a \MT\ using elements of Nielsen classes.  \eql{steppingup}{steppingupd} gives a general criterion for existence of a \MT\  over a given braid orbit $\sO_k\le \ni(G_k,\bfC)^\inn$ under special circumstances. These include that the orbit contains an \HM\ rep. \eql{steppingup}{steppingupc} gives a pure lift invariant criterion for an abelianized \MT\ over $\sO_k$. The territory between them is spanned by the if and only if lift invariant criterion  \eql{steppingup}{steppingupb} for $\ni(G_{k\np 1},\bfC)$ having an orbit above $\sO_k$. 

\begin{prop} \label{existProp} If $G$ has $\ell'$ center,   then so does  $G_k$, and since $G$ is $\ell$-perfect, so is $G_k$,  $k\ge 1$. 

\begin{edesc} \label{steppingup} \item \label{steppingupa}  There is a \MT\ on a braid orbit $\sO_k\subset \ni(G_k,\bfC)$ if and only if the preimage of $\sO_k$ in $\ni(G_{k\np t}, \bfC)$ is nonempty for all $t\ge 0$. 
\item \label{steppingupb}  A  braid orbit $\sO_k \subset
\ni(G_{k},\bfC)$ is obstructed (Def.~\ref{bHthusHk})  if and only if it is not in the image of $\ni(R_{k,\ell}^*,\bfC)$, with $R_{k,\ell}^*$ the universal central extension of \eql{explainmuk}{explainmukb}. 
\item \label{steppingupc}  There is an abelianized \MT\ on a braid orbit $\sO_{k,\ab}$ of $\ni(G_{k,\ell,\ab},\bfC)$ if and only if $\sO_{k,\ell,\ab}$ has trivial lift invariant computed  from $R_{k,\ell}^*\to G_k$. 
\item \label{steppingupd}  There is a \MT\ on a braid orbit $\sO$ containing $\bg=(\bh_1,\dots, \bh_u)$ with \begin{itemize} \item $\bh_i$ satisfying product-one and  $\lrang{\bh_i}=H_i$ is an $\ell'$ group, $1\le i\le u$. 
\item The \HM\ case has  $H_i$  a cyclic $\ell'$ group.\end{itemize}  \end{edesc} In \eql{steppingup}{steppingupc}  and  \eql{steppingup}{steppingupd} there may be more than one branch (\MT\ braid orbit).
\end{prop}  

\begin{proof} 
 \cite[Prop.~3.21]{BFr02} replaces the phrase \lq\lq has $\ell'$ center\rq\rq\  with \lq\lq is centerless:\rq\rq\  a consequence of interpreting having no center by inspecting the {\sl Loewy displays\/} of the universal $\ell$-Frattini covers of $G$. This version states it for one prime $\ell$. We go through the list \eqref{steppingup} one by one.\footnote {The tests to be passed are independent of what equivalence relation we apply to the Hurwitz space components.} 

{\sl Proof of \eql{steppingup}{steppingupa}}: For $\bg_k\in \sO_k$, finding a \MT\ on $\sO_k$ is equivalent to  producing a sequence $\{\bg_{k\np t}: t\ge 0\}$ with $\bg_{k\np t}\in G_{k\np t}$ and $\bg_{k\np t}\mapsto \bg_{k}$ by the canonical map \eql{explainmuk}{explainmuka}, $t\ge 0$.  Since Nielsen classes are finite sets (therefore compact), and these maps define chains, a \MT\ is a maximal chain. By the Tychonoff Theorem, such exists under the hypothesis   \eql{steppingup}{steppingupa}. 

{\sl Proof of \eql{steppingup}{steppingupb}}:  From \eql{explainmuk}{explainmukc}, $\psi_{k,k\np1}$ factors through $\mu_k$. If $\bg_k\in \ni(G_k,\bfC)$ is the image of $\bg_{k\np1}\in G_{k\np1}^r\cap \bfC$ (as in Def.~\ref{liftinv}), which satisfies product-one, etc., then the image of $\bg_{k\np1}$ in $(R_{k,\ell}^*)^r\cap \bfC)$, $\bg_{k\np1}^*$ also satisfies product-one and generation, etc. 

The converse -- existence of $\bg_{k\np1}^*$ satisfying Nielsen class properties, produces $\bg_{k\np1}$ -- went through two stages.  \eqref{gkobst} rephrases \cite[Obs.~Lem.~3.2]{FrK97}. {\sl No\/} braid orbit, $\sO_{k\np1}\subset \ni(G_{k\np1},\bfC)$ above $\sO_k$ is equivalent to this: 
\begin{edesc} \label{gkobst} \item in {\sl any\/}  Loewy Path (Def.~\ref{irrpath} on $M_{k\np1}$) the trivial $\bZ/\ell[G_k]$ module $\one_{G_k}=\bar M_{i\np1}$ appears as $\ker (G_{**}\to G_*)$   in a sequence $G_{k\np1}\to G_{**} \to G_* \to G_k$ with 
\item  $\bg^{*} \in \ni(G_{*},\bfC)$ over $\bg$, $\bg^{**}\in \bfC\cap G_{**}$ (uniquely defined) over $\bg^*$ and $g_1^{**}\cdots g_r^{**}\not =1$. \end{edesc}  

Now I simplify \cite[Lem.~4.9]{lum}. \cite[Prop.~4.19]{lum}  substitutes all  tests in \eqref{gkobst} by just one: 
\begin{equation} \label{Rk+1} \text{As in Def.~\ref{liftinv}: }s_{R_k^*/G_k}(\tilde \bg)\not =1, \ \tilde \bg\in R_{k,\ell}^*\cap \bfC \text{ over }\bg_k. \end{equation} 
The homological interpretation of this is part of \eql{steppingup}{steppingupc}.\footnote{Here (resp.~in  \eql{steppingup}{steppingupc}) asserting there is a component above $\sO_k$ is nontrivial even when $\ker(R_{k,\ell}^*\to G_k)$ (resp.~$\ker(R_{k,\ell}\to  G_k)$) is 0.}

{\sl Proof of \eql{steppingup}{steppingupc}}:  Following the procedure of \eql{steppingup}{steppingupa}, refer to a maximal projective sequence of  elements $\bg'=\{\bg_{k\np t}\in \ni(G_{k\np t},\bfC)\}_{t=0}^\infty$ as a {\sl branch}. 

\begin{defn}[Component branch] \label{compbr} Denote the corresponding (to $\bg'$) projective sequence of braid orbits $\sB_k\eqdef \sB_{k,\bg'}=\{\sO_{k\np t}\}$ as a {\sl component branch\/}; another way to describe a \MT. \end{defn} Nielsen classes generalize to any (profinite) quotient $G'$ of $\tG \ell\to G$. Consider a braid orbit $O'\le \ni(G',\bfC)$ over
$O_k$. This corresponds  to   
$\psi': \sM_O\to G'$ factoring through $\psi_{O_k}$. Weigel's Th.~\ref{WeigelThm}  says  $\sM_O$ is an {\sl oriented $\ell$-Poincar\'e duality
group\/}. 

Limit braid orbits $\sO^*$ in $\ni(G^*,\bfC)$ define limit groups. Any quotient $G^{\#}$ of $\tG \ell$ has attached component and cusp graphs by running over Nielsen classes corresponding to quotients of $G^{\#}$.

\begin{thm} \label{WeigelThm}  $\sM_\bg$ is a dimension 2 oriented $\ell$-Poincar\'e duality group. \cite[Lem.~4.14]{lum}: $\sO_\bg$ starts a component branch if and only if, running over $\psi_{R'}: R'\to G'$ with $\ker(\psi_{R'})$ a quotient of ${\tsp S} {\tsp M}_{G,\ell}$,  each  $\psi_{G'}: \sM_\bg\to G'$ extending
$\sM_\bg\to G$ extends to $\psi_{R'}: \sM_\bg \to R'$. The obstruction to extending $\psi_{G'}$ to $\psi_{R'}$ is the image in $H^2(\sM_\bg,\ker(\psi_{R'}))$ by
inflation of $\alpha\in H^2(G',\ker(\psi_{R'}))$ defining the extension $\psi_{R'}$.
\end{thm} 

\begin{proof}[Comment] \cite[\S4.3]{lum} discusses this using classical generators $\row x r$ to describe the fundamental group of $\pi_1(X)^{(\ell)}$. \end{proof}

The phrase (dimension 2) $\ell$-Poincar\'e duality \cite{We05} expresses an exact cohomology pairing  \begin{equation} \label{duality} H^k(\sM_\bg,U^*)\times H^{2-k}(\sM_\bg,U)\to
\bQ_\ell/\bZ_\ell\eqdef I_{M_{\bg},\ell}\end{equation} where $U$ is any  abelian $\ell$-power group  that is also a $\Gamma=\sM_\bg$ module, $U^*$ is its dual
with respect to
$I_{\sM_{\bg},\ell}$ and 
$k$ is any integer.
\cite[I.4.5]{Se91} has the same definition, except in place of $\sM_\bg$ has  
 a pro-$\ell$-group. By contrast,   $\sM_{\bg}$ is $\ell$-perfect, being generated by $\ell'$ elements (Lem.~\ref{perfnongen}). 

Capturing the extension problems for forming a \MT\ through quotients of $\sM_{\bg}$ involves Frattini covers of $G$, which are also 
$\ell$-perfect.  The fiber over $O_k$ is empty if and only if there is some central Frattini extension $R\to G_k$ with kernel
isomorphic to $\bZ/\ell$ for which $\psi_\bg$ does not extend to $\sM_\bg\to R\to G$ \cite[Cor.~4.19]{lum}. 

Comment on the proof:  The key is \cite[Prop.~2.7] {Fr95}, which says $H^2(G_k,M_{k\np1}) = \bZ/\ell$ (it is 1-dimensional.) Then, the obstruction to lifting $G_k$ to $G_{k\np1}$ is the inflation of some fixed generator $H^2(G_k, M_{k\np1})$ to $H^2(\sM_g, M_{k\np1})$ \cite[Lem. 4.14]{lum} .
That proof also applies to limit groups. 
\cite[Cor.~4.20]{lum}:  If $G^*$ is a limit group in a Nielsen class and a proper quotient of  $\tG \ell$, then $G^*$ has exactly one nonsplit extension by a $\bZ/\ell[G^*]$ module, and that module must be trivial.

{\sl Proof of \eql{steppingup}{steppingupd}}:  Consider $\bg=(\bh_1,\dots, \bh_u)\in \ni(G_{k},\bfC)$ as in \eql{steppingup}{steppingupd}. Apply Schur-Zassenhaus to lift $H_i$ to $G_{k}$ from $G_k$ giving $\{\bh^*_i\}_{i=1}^t$ satisfying the same conditions in $G_{k\np1}$ as given for  $\{\bh_i\}_{i=1}^t$. Since $G_{k\np1}\to G_k$ is an $\ell$-Frattini cover, it is automatic that $\lrang{\bh^*_i, i=1,\dots, t}=G_{k\np1}$. \end{proof}

\begin{defn} For emphasis on the head of the group $\sM_\bg$, with the $G$ module lattice tail, we sometimes refer to it as a $(G,\ell)$-Poincar\'e Duality group.\end{defn}

\begin{rem}[dropping $(N_{\bfC},\ell)=1$  on the lift invariant] \label{notlperf} To drop the $\ell'$ assumption on $\bfC$, as in \cite[App.]{FrV91}: Replace ${}_\ell R$ by its maximal quotient, for which any  class $\tilde C_i$ of $H^*$ over any $\C_i$ of $\bfC$, has $|\tilde \C_i|=|\C_i|$.  This is equivalent to modding out by the group generated by elements of form $\{g'g(g')^{-1}g^{-1}\mid g'\in \bfC, g\in G\}\cap \SM_G$.\end{rem} 

\begin{rem} \label{lum} \cite[\S 2.1]{lum} exposits on the universal $\ell$-central extension when $G$ is $\ell$-perfect; it was not classical to restrict to one prime at a time, though it is based on \cite{Br82}. Another description of a representation cover is as a maximal quotient of ${}_\ell \tilde G_\ab \to G$ on whose kernel $G$ acts trivially. 

In \eql{explainmuk}{explainmukb} the order of $\ker(\mu_k)$ in $\mu_k: R_{k,\ell}^*\to G_k$ grows with $k$ for fundamentally the same reason the elements of order $\ell$ in the $\ell$-Frattini cover $\mu: \bZ/\ell^2\to \bZ/\ell$ map to 0 by $\mu$.   
\end{rem}
 
 \begin{rem}[Appearances of $\one_G$ in the modules $M_k$] \label{1Gs} This is an addendum to Prop.~\ref{existProp}. For example, if $G_k$ has $\ell'$ center, but $G_{k\np1}$ does not, then $\one_{G_k}$ appears at the left end of the Loewy display of $M_k$. Also, a subquotient with Loewy layers $\one_{G_k}\to \one_{G_k}$ can't appear in $M_k$; that would -- contrary to $G_{k\np1}$ is $\ell$-perfect -- give a homomorphism \cite[(3.17b)]{BFr02} $$G_{k\np1} \to \{\smatrix 1 {b} 0 {1}\mid b\in \bZ/\ell\}\ \to \{b\in \bZ/\ell\}.$$  

Using the universal Frattini cover of $G$ produces a great number of Schur-separated components among the levels of Modular Towers over most $\ell$-perfect finite groups $G$.  For example, from the result of Darren Semmen quoted in \cite[Prop.~5.3]{BFr02}. \end{rem}

\subsubsection{\HIT\ decomposition groups}  \label{decgps}  We state results for covers in a given absolute Nielsen class $\ni(G,\bfC)^\abs$ to remind of  Hilbert's Irreducibility Theorem (\HIT). Start with $\phi: X\to \prP^1_z$ 
\begin{edesc} \item  defined by  $\bp'\in \sH'\le \sH(G,\bfC)^\abs$  over the number field $K=K_{\sH'}(\bp)$; and  
\item a Galois closure, $\hat \phi: \hat X\to \prP^1_z$, of $\phi$, given by $\bp\in \sH\le \sH(G,\bfC)^\inn$ lying over $\bp'$, with $\sH$ a component over $\sH'$.
\end{edesc} 
Assuming fine moduli ($G$ has no center),  then $\hat \phi_\bp$ has definition field $K_{\sH}(\bp)$ given by the moduli definition field of $\sH$. The decomposition group $D_{z'}$ of $z'\in \prP^1_z(\bar K)$ is the Galois group of the field obtained by joining coordinates of all points of $X$ above $z'$, and their conjugates over $K(z')$. This will be a subgroup of the (arithmetic) monodromy group, $\hat G_{\bp/\bp'}\eqdef G(K_{\sH}(\bp)/K_{\sH'}(\bp'))$, an extension of the group of the constants field \eqref{extconstants}. The simplest \HIT\ statement:\footnote{Hilbert's examples didn't need Hurwitz spaces, or the apparatus we use, but we do.} \begin{edesc} \label{HIT1} \item \label{HIT1a} $\Hi_{\phi,K}$: for $z'$ dense in $\prP^1_z(K)$ (even in $\bQ$)  the fiber $X_{z'}$ is irreducible (over $K(z')$); 
\item and  \eql{HIT1}{HIT1a} applied to  $\hat \phi_\bp$, $D_{z'}$ is the monodromy group $\hat G_{\bp/\bp'}$.  \cite{Hi1892}\end{edesc} 

\begin{defn} \label{evenfratt} Call a sequence of finite group covers $$\dots \to H_{k\np1}\to H_k\to \dots \to H_1 \to H_0=G$$ {\sl eventually Frattini\/} (resp.~eventually $\ell$-Frattini) if there is a $k_0$ for which $H_{k_0\np k} \to H_{k_0}$ is a Frattini (resp.~$\ell$-Frattini) cover for $k\ge0$. If the projective limit of the $H_k\,$s is $\tilde H$, we say $\tilde H$ is eventually $\ell$-Frattini. Then, any {\sl open\/} subgroup of $\tilde H$ will also be eventually $\ell$-Frattini. 

Refer to a \MT\, ${\pmb \sH}_{\pmb \sO}$, as Frattini (resp.~eventually $\ell$-Frattini) if its {\sl geometric\/} monodromy group has this property. In the notation of a component branch (Def.~\ref{compbr}), this says the projective system of groups given by the braid group action on the sequence $\pmb B_{\pmb \sO}=\{B_k\}_{k\ge 0}$ has this property.  \end{defn}

Use the notation ${\pmb \sH}_{\pmb \sO}=\{\sH_k\}_{k \ge 0}$ for a \MT\ on ($\ni(G,\bfC)^\inn$, above Def.~\ref{bHthusHk}) with $\pmb\sO=\{\sO_k\}_{k\ge 0}$ the corresponding braid orbits on the \MT\ levels. Denote the group of braid actions on $\sO_k$ by $B_k$, $k\ge 0$, assuming we have identified $\pmb \sO$. 
For $\bz'\in U_r(K)$,  denote by $\bar \bp=\{\bp_k\in \sH_k\}_{k\ge 0}$ a projective system of points on ${\pmb \sH}_{\pmb \sO}$ over $\bz'$ lying on the \MT. Consider these systems of groups.

\begin{edesc} \label{decgp} \item \label{decgpa} The projective system of decomposition (arithmetic monodromy) groups, $D_{\bz',k}$ for the cover $\phi_{\bp_k}: X_{\bp_k}\to \prP^1_z$ and its projective limit:  ${\pmb D}_{\bz'}=\lim_{\leftarrow  k} \{D_{\bz',k}\}_{k\ge 0}$.  
\item \label{decgpb} Then, the projective system, $\pmb D_{\MT, U_r}$ of the arithmetic monodromy of the \MT. 
\end{edesc} 

Compatible with the notation,  ${\pmb D}_{\bz'}$, for the  $\bz'$ fiber group, is independent of $\bar \bp$. 
\begin{prop} \label{HIT2} Assume ${\pmb \sH}_{\pmb \sO}$ is eventually Frattini and $L$ is a number field. Then, for a dense set of $\bz'\in U_r(L)$ (or in $U_r/\PSL_2(\bQ)(L)\eqdef J_r(L))$, 
\begin{equation} \label{HITcond} \text{ the fiber of $\pmb \sH$ over $\bz'$ equals the arithmetic group of the Modular Tower over $L$.} \end{equation} 

\end{prop} 

\begin{proof} Given a Galois cover of normal varieties, $\phi: W\to V$ over a field $K$, the decomposition field over $v'\in V(K)$ automatically contains the extension of constants field of $\phi$. Since the use of Grauert-Remmert to form the Hurwitz space components in a \MT\ uses projective normalization, all covers are of normal varieties, if the decomposition group of a fiber contains the geometric monodromy group of the cover, it automatically contains the arithmetic monodromy group. 

With $k_0$ the starting index for eventually Frattini, a standard generalization of \eqref{HIT1} implies the conclusion to \HIT\ holds for any cover of a variety birational to projective space. So it applies to a cover of $U_r$. This gives a dense set of $\bp_{k_0}\to \bz'$ for which the (from above, arithmetic or geometric) monodromy group attached to the cover $\phi_{k_0}\in \sH_{k_0}$, equals the monodromy of $\sH_{k_0}\to U_r$. From eventually Frattini, we can change $k_0$ to any $k\ge k_0$ for a corresponding dense set of $\bz'$. The image of this dense set modulo $\PSL_2(\bQ)$ will be dense in the image, giving the same conclusion for $J_r$. \end{proof}  

\begin{defn} \label{HITdef} Refer to ${\pmb D}_{\bz'}$, in \eql{decgp}{decgpa} as {\sl \HIT\  \/} (resp.~{\sl full\/} \HIT) on the \MT\  if it is an open subgroup of (resp.~equals) the decomposition group of the \MT. \end{defn}

\begin{exmpl}[\HIT\ results without Nielsen classes] \label{indnc} One attempt for a definitive \HIT\ result is to form an explicit ({\sl primitive recursive}) Hilbert Set, $\Hi_{\phi, K} \le \prP^1_z(K)$:  for $z'\in  \Hi_{\phi, K}$, \eqref{HIT1} holds. For one cover, $\phi: X\to \prP^1_z$, \cite[Thm.~2]{Fr74} gives a nonregular analog of the Chebotarev density theorem, and  \cite[Thm.~3]{Fr74} applies it to construct $\Hi_{\phi, K}$ as an arithmetic progression in $\bZ$. 

 \cite[Thm.~4.9]{Fr85} gives an explicit universal Hilbert subset $\Hi_{K}$ for which \eqref{HIT1} holds for {\sl each\/} $\phi$ with finitely many exceptional  $z'\,$s dependent on $\phi$.\footnote{On the scope of  \HIT\ (a la, the title of \cite{Fr85}): I used Weil's Theory of (arithmetic) distributions (1928 thesis), Sprindzuk used diophantine approximation and Weissaur used nonstandard arithmetic.} Examples of  \cite{D87} -- these are for irreducibility of $\phi_{z'}$ -- are memorable:  $\{ 2^n + n\mid n\ge 0\}$, but it relies on Siegel's Theorem; so is not effective. The examples of \cite[Thm.~4]{DZ98} give a bound $N_{\phi}$ such that $\{n\in \Hi_{\phi,\bQ} \mid n> N_{\phi}\}$.\end{exmpl} 

\subsection{Hurwitz spaces and Jacobians} \label{HSJac}  \S\ref{OITver} explains Serre's \OIT\ as about decomposition groups on the fibers of a \MT\ that is identified with a tower of modular curves. This emphasizes the eventually $\ell$-Frattini property. Serre's \OIT\ has two possibilities for ${\pmb D}_{\bz'}$ for $\ni((\bZ/\ell)^2\xs \bZ/2, \bfC_{2^4})$, for a fixed prime $\ell\ne 2$ with list \eqref{serre} stating this more precisely. \S \ref {jacobians} connects Serre's \OIT\ (and generalizing it) and the main \MT\  conjecture to naming and divining properties of the Jacobians along the curves attached to points on Hurwitz spaces. This connection starts with Hilbert's conjecture on geometrically interpreting abelian extensions of complex quadratic fields, but it's a bigger topic than that (see list \eqref{serre}).   
 Prop.~\ref{HIT2} says for any \MT\ that is eventually $\ell$-Frattini, we can expect the \lq\lq general\rq\rq\  ${\pmb D}_{\bz'}$ to be {\sl \HIT}, our name generalizing Serre's $\GL_2$ type.  
 
 But, in Serre's case, there is another type, in \S\ref{cmabelvar}, \CM,  for which the $\ell$-adic representation presents  the Galois group of ${\pmb D}_{\bz'}$ as an abelian extension of $\bQ(j')$ with $j'$ the $j$-invariant corresponding to an elliptic curve with complex multiplication. In \S \ref{weilpairing} this corresponds to a cover in the  Nielsen class of \eql{mtcases}{mtcasesa}. 
 
 \S\ref{WohlfahrtThm} reminds how, using {\sl Wohlfahrt's Theorem\/} and the Riemann-Hurwitz formula Prop.~\eqref{j-Line} for reduced Hurwitz spaces with $r=4$, to exclude a reduced Hurwitz space cover from being a modular curve.  Andr\'e's Theorem \ref{andres} requires knowing our reduced Hurwitz space is not a modular curve to conclude an example where we don't get the \CM \ analog of Serre. Instead, for any compact subset of $\prP^1_j\setminus \{\infty\}$,  only finitely many fibers of the \MT\ are \ST\ (the general analog of \CM) type. This is the case $\ell=2$ in series of examples in \S\ref{heiscase}. 

\subsubsection{Tying to the \OIT} \label{OITver}  

\begin{edesc} \label{serre}  \item \label{serrea}  \CM\ type: With $j'$ corresponding to an elliptic curve with ring of endomorphisms an order in a complex quadratic extension of $\bQ$, then ${D}_{\bz'}$ an open subgroup of $\bZ_\ell$. 
\item \label{serreb} With  $j'$ not an $\ell$-adic integer;  ${D}_{\bz'}$ is full \HIT\ with ${\pmb D}_{\bz'}$ (resp.~geometric monodromy)  equal to $\GL_2(\bZ_\ell)$ (resp.~$\SL_2(\bZ_\ell)$) \cite[\S 3.2]{Se68}.  
\item \label{serrec} With $j'$ an algebraic integer but not of \CM\ type, in our language, ${\pmb D}_{\bz'}$  is \HIT\ (type). In Serre's case, refer to the decomposition groups as of $\GL_2$ type. 
\end{edesc} A Tate paper that never materialized suggested {\sl all\/} non-\CM\ fibers  (not just those in \eql{serre}{serreb}) would give \HIT\ for ${\pmb D}_{\bz'}$;  \eql{serre}{serrec} requires Faltings Theorem. \S\ref{SMR} reviews Serre's constructions, his characterization of compatible collections of $\ell$-adic representations, and especially his showing that \ST\ is included.

\cite[Prop.~3.20]{Fr20} -- Prop.~\ref{PSLFrat} -- shows the eventually $\ell$-Frattini property applies to Serre's case. Thus, Prop.~\ref{HIT2} says the \lq\lq general\rq\rq\ decomposition group of Serre's case has $\GL_2$ type. 
\begin{prop} \label{PSLFrat} The natural cover $\SL_2(\bZ/\ell^{k\np 1})\to \SL_2(\bZ/\ell)$  is an $\ell$-Frattini cover for all $k$ if $\ell> 3$. For $\ell=3$ (resp.~2),
$\SL_2(\bZ/\ell^{k+1})\to \SL_2(\bZ/\ell^{k_0\np1})$, $k\ge k_0$, where $k_0=1$ (resp.~2),  is the minimal value for which these are Frattini covers. For all $\ell$,  $$\PSL_2(\bZ_\ell)\to \PSL_2(\bZ/\ell)\text{ is eventually $\ell$-Frattini}.$$\end{prop}

\subsubsection{Jacobians of curves on a Hurwitz space} \label{jacobians} 
Start from a \MT\  (of inner spaces) and take the level $k$ component $\sH_k$. For each $\bp\in \sH_0$, consider the cover $\hat \phi_\bp: \hat X_\bp\to \prP^1_z$ and  the level $k$ space $\sJ_k$ of covers of the Jacobians, $\sJ_{k,\bp}$, of $\phi_\bp$ with kernel $(\bZ/\ell^{k\np1})^{2\geng}$. \cite[\S6]{Fr10} discusses the natural map $\sH(G,\bfC)^\red\to  \sJ_{G,\bfC}$ curves in a Hurwitz space to their corresponding Jacobian varieties. We need the curve in its Jacobian. \cite[Lect.~III]{Mu76} is analytic, following Riemann using holomorphic differentials, with no reference to $G_\bQ$. 

This starts from Riemann's  birational equivalence of the Jacobian $J_\bp$ associated to the curve $\hat X_\bp$ (of genus $\geng$)  -- here Galois over $\prP^1_z$ is irrelevant -- with the symmetric product $\Symm_{\geng,\bp}=(X_\bp)^{\geng}/S_{\geng}$. An application of the Riemann-Roch theorem shows that, for {\sl general\/} divisors $D_1$ and $D_2$ on $\Symm_{\geng,\bp}$ there is a unique linear equivalence class $D_3\in \Symm_{\geng,\bp}$ linearly equivalent to $D_1\np D_2$. Therefore, modulo linear equivalence $J_\bp$ is the {\sl group\/} of degree 0 divisor classes on $\hat \phi_\bp$. The algebraic structure on $\sJ_\bp$ comes from the analytic functions on multiples of the linear system from the  $\theta$-divisor $\Theta_\bp$. This identifies with the space of divisor classes of degree $\geng \nm1$ modulo linear equivalence and (again due to Riemann, but made algebraic by Weil).  So, there is a definition field of this structure\footnote{Weil needed this to complete his thesis:  the proof of finite generation of the points defined on an abelian variety over a number field.} giving an embedding of $J_\bp$ in a projective space.  Points on $\hat X_\bp$ (resp.~ $\Theta_\bp$) map to points of degree 0 by translating by a divisor class $D'_\bp$ of degree 1 (resp.~$D''_\bp$ of degree $\geng\nm1$). Lem.~\ref{gact} now uses that $\hat \phi$ is a Galois cover: giving an action of $G$ on $J_\bp$ in \eqref{gact2}. 

\begin{lem} \label{gact} There is a copy of $\hat X_\bp$ in $J_\bp$ on which $G$ acts compatibly with its action on $J_\bp$. \end{lem}

\begin{proof}  Take $x_0\in \hat X_\bp$ and form  $\hat X_\bp\nm x_0$. Since $g\in G$ maps $\hat X_\bp$ to $\hat X_\bp$, it maps $\hat X_\bp\nm x_0$ to $\hat X_\bp \nm x_0^g$ by a map that is uniquely detected by what it does to $\hat X_\bp$. This action gives a 1-cocycle $g\in G \mapsto x_0^g\nm x_0$ of translations of $\hat X_\bp\nm x_0$ inside $J_\bp$ along with unique maps between the translations. Now apply Weil's cocycle condition (as in the proof of Prop.~\ref{inn-absmdf}) to construct $\hat X^*_\bp\subset J_\bp$ with $G$ action. 

Fix a basis $\pmb\omega_\bp\eqdef \row \omega \geng$ of the holomorphic differentials on $\hat X_\bp$. Form $\geng$-tuples of integrals
\begin{equation} \label{defpathints} \begin{array}{c} \pmb \Omega(x_0,x)\eqdef(\row {\int_{x_0}^x \omega} \geng), x\in \hat X_\bp \mod L, \text{periods along closed paths at $x_0$};\\
\text{and } \bar {\pmb  {\Omega}}(x_0,\bx)\eqdef( \int_{x_0}^{x_1} \omega_1,\dots, \int_{x_0}^{x_\geng}\omega_{\geng}), \row x \geng \in \hat X_\bp \mod L\end{array}\end{equation}  Therefore, $\bar {\pmb  \Omega}(x_0,\bx)$)  is independent of the choice of paths from $x_0$ to $x$ (resp.~$x_0$ to $x_i$, $i=1,..., \geng$).  The following is due to Riemann. 
\begin{edesc} \label{Riem}  \item \label{Riema}  The collection of path integrals, $J_{\bp}$,  of the 2nd line of \eqref{defpathints}, describe the linear systems on $\hat X_\bp$ of degree 0. These form a complex torus of dimension $\geng$. 

\item \label{Riemb}  From \eql{Riem}{Riema},  the collection of path integrals of the first line gives an embedding of $\hat X_\bp$ in $J_{\bp}$ (dependent on $x_0$) as degree 0 divisors of form $\hat X_\bp \nm x_0=\{x\nm x_0\}_{x\in \hat X_\bp}$. \end{edesc} 

For $g\in G$, write $\pmb \Omega(x_0,x^g)$ as $U(x_0,x^g)\eqdef \pmb \Omega(x_0,x_0^g) + \pmb \Omega(x_0^g,x^g)$.  As a {\sl collection of points\/}, this is the same as $\hat X_\bp \nm x_0$. The second summand is obtained from  $g$ acting on (endpoints of) the paths of the integrals of the second line of \eqref{defpathints}. Here is the action. 

\begin{equation} \label{gact2} \begin{array}{c} \smatrix g {U(x_0,x_0^g)}  0 1 (U(x_0, x)\ 1)^{\tr}= (U(x_0^g, x^g)\np  U(x_0,x_0^g)\ 1)^\tr \text{ and multiplying matrices} \\
\smatrix {g_2} {U(x_0,x_0^{g_2})}  0 1 \smatrix {g_1} {U(x_0,x_0^{g_1})}  0 1=\smatrix {g_2g_1} {U(x_0,x_0^{g_1})^{g_2} \np  U(x_0,x_0^{g_2})} 0 1,\end{array}\end{equation}
the result of first appying $g_1$ and then $g_2$ is the same as applying $g_2g_1$.  \end{proof}  

 Corresonding to $\bp_0\in \sH(G_0,\bfC)$, denote the Jacobian from  the second line of \eqref{defpathints} by $J_{\bp_0}$. Suppose \MT$=\{\sH_k\le \sH(G_{k,\ab},\bfC\}^\inn\}_{k=0}^\infty$ is a(n abelianized) Modular Tower, and $\{\bp_k\in \sH_k\}_{k=0}^\infty$ is a projective sequence of points over $\bp_0$ on the \MT.  For, $k\ge 1$, each $\bp_k\leftrightarrow \hat \phi_k: \hat X_{\bp_k}\to \prP^1_z$ is a Galois unramified cover of $\hat \phi_{\bp_0}$ obtained by 
pullback to $J^*_k$, a quotient of $\psi_{J,k\np1}: J_{\bp_0}\to  J_{\bp_0}$, $k\ge 0$, from modding out by the lattice $\ell^{k\np1}\sL_{\bp_0}$. \footnote{The same $G$ module, $M_1$ in previous notation, is in the kernel from $\ell^{k\nm1}\sL_{\bp_0}\to \ell^{k}\sL_{\bp_0}$ independent of $k$.}  

The action of $G$ extends to the $\ell$-adic Jaoobian module $\sL_{J,\bp_0}$ and to $\sL_{J,\bp_0}/\ell^{k\np1}\sL_{J,\bp_0}=M_{J,k\np1}$ (Lem.~\ref{gact}). This module already appears as the $\bZ/\ell[G]$ quotient of $\sM_{\bar\bg}$ in   \eqref{psig}.\footnote{Unlike the characteristic $\ell$-Frattini module, $M_{J,1}$ may not be indecomposable (as in the Serre's case \S\ref{weilpairing}).}  \begin{defn}\label{jaccase} Taking $G_{J,k}$ as the composite of $G$ and $M_{J,k\np 1}$, this therefore gives a Nielsen class $\ni(G_{J,k}, \bfC)$ with our usual equivalences compatible, extending $\ni(G_k,\bfC)$, with the extending braid group action. \cite{Fr20} referred to this as the {\sl Jacobian case}. \end{defn}

\begin{lem} \label{HrMTact} There is an explicit procedure for computing an open subgroup of the action of $H_r$ (Hurwitz monodromy) on the projective sequence of Nielsen class braid orbits of a \MT\ equivalent to an action on the lattice tail of the extension ${\tG \ell}_{\ab} \to G$ (below \eqref{psig}). This gives a check of the eventually Frattini property of a \MT. 
\end{lem}  

\begin{proof} Use notation as in \eqref{steppingup} for the braid orbits $O_k\le \ni(G_k,\bfC)^\inn$ defining the \MT. Take $\bg_1\in O_1$ to define the cover $\hat X_1\to \hat X_0$ for the first level unramified $\ell$-Frattini cover. By the universal property of the Jacobian variety, the pullback, $\hat X_{J,1}$ of $\hat X_0$ in the Jacobian cover \begin{equation} \label{jacPullback} \text{$\psi_{J,1}$: $J_{\bg} \longmapright {\rm mult. by \ell} {35} J_{\bg}$ is an unramified cover (it may not be connected).}\end{equation} 

Consider the subgroup $H^*\le H_r$ that is fixed on an element of $O_0$, and take the component of $\hat X_{J,1}$ mapping surjectively to $\hat X_1$ which defines the cover $\hat X_1\to \hat X_0$ fulfilling the first step in a \MT. This works inductively for $k\ge 1$, and the kernels of $G(\hat X_{J,k\np1}/\prP^1_j)\to  G(\hat X_{J,k}/\prP^1_j)$ define an $\ell$-adic lattice on which $H^*$ has an orbit in $\ker (\sM_{\bar\bg}\to G)$. Scheier's construction of generators of $H^*$ (using the two standard generators of $H_r$) gives an explicit action on the lattice. The eventually Frattini property of a \MT\ is equivalent to this action being eventually Frattini.  
\end{proof}

\newcommand{\End}{{\text{\rm End}}}
 \subsubsection{\ST\ points and their abelian varieties}  \label{cmabelvar} \S\ref{jacobians} gave Riemann's production of the Jacobian. Riemann also gave the construction of a complex algebraic (embeddable in projective space) abelian variety from $\bC^{\geng}/\sL$ when $\sL$ is a $2\geng$ dimensional lattice with  the imaginary part of the matrix of generators of $\sL$ is positive definite. Def.~\ref{schottky} gives the most famous problem, Schottky's, for differentiating general such complex torii from the Jacobians of curves.\footnote{I can't see its use here.}

Serre's \OIT\ with its two types of decomposition groups -- both eventually $\ell$-Frattini -- immediately raises these questions. For each, the tacit assumption is that you would also ask for which Nielsen classes (or if possible, which \MT s) you would expect the answer to manifest.  This section has sufficient information about the Shimura-Taniyama abelian varieties (\ST) to demonstrate why they appear as the appropriate generalization of \CM.  What is, perhaps, surprising is how much they seem to be the {\sl only\/} type of abelian varieties that garner special attention, though I (and in his case, Serre) emphasize those that I am calling \HIT, giving a definition of them dependent on generalizing Serre's modular curve towers to \MT s. 
 
Take $L$ a number field. If complex conjugation $\cmb{\  }: L\to L$ acts nontivially on $L$, then the fixed field $K$ is real, of index 2 and $L=K(\alpha)$ with $\alpha$ and $\cmb{\alpha}$ conjugate. My notation is similar to \cite[\S5.5]{Sh71}, starting with his \S A on (what he calls) \CM\ fields. 
 \begin{defn} \label{commbar} Refer to $L$ as a \CM\ field if all embeddings $\psi:L\to \bC$ are complex ($L$ is totally complex). Then, all embeddings of $K$ in $\bC$ are real, and  all such $\psi\,s$ commute with  $\cmb{\  }$ acting on $L$.\end{defn} 
 
 \begin{lem} Given two \CM\ extensions $L_i/\bQ, i=1,2$, their composite is another; therefore the Galois closure of a \CM\ extension is also \CM. \end{lem}
 
 \begin{proof} Embeddings of $L_1\cdot L_2$ into $\bC$ are given by compositing separate embeddings of $L_1$ and $L_2$. Check: $\cmb \ : L_1\cdot L_2 \mapsto L_1\cdot L_2$ therefore commutes with any embedding of $L_1\cdot L_2$. The Galois closure of $L_1/\bQ$ is the composite of all the conjugates of $L_1/\bQ$ with each of form $\psi(L_1/\bQ)$, $\psi$ an embedding in $\bC$.  So it satisfies Def.~\ref{commbar}. \end{proof} 
 
 Shimura constructs abelian varieties $A=\bC^n/\sL$ -- a complex torus -- with $\theta: L \mapsto \End_{\bQ}(A)$ with $2n=\deg(L/\bQ)$. He called them  \CM\ type; we will often use \ST. 
 
 \begin{edesc} \label{ST} \item \label{STa} There is a divisor, $D$, on $A$ for which multiples of $D$ have a linear system that gives an embedding of $A$ in projective space.
 \item \label{STb}  \cite[ (5.5.10)]{Sh71} uses the distinct complex embeddings, $\row \phi n$, of $L$ and their conjugates;  to define $A$ from $L_{\bR}=L\otimes_{\bQ}\bR$ modulo a $\bZ$ lattice in $L$ (e.g. integers of $L$). 
 \item \label{STc} \cite[p. 258-259]{Sh71} reminds of Riemann's Theorem saying a complex torus has the structure of an abelian variety (as in \eql{ST}{STa} if and only if there is a {\sl Riemann form}: \begin{itemize} \item  an alternating and $\bR$-valued bilinear  pairing, $E(\bx,\by)$, $\bx,\by\in \bC^n$; and 
 \item $E$ takes integer values on $\sL\times \sL$ with $E(\bx,\sqrt{-1}\by)$ symmetric and positive definite.\end{itemize}
 \end{edesc}
 
Riemann constructed the $\Theta$-function from \eql{ST}{STc}. After normalizing \cite[(5.5.15]{Sh71} gives a formula for a Riemann form:  \eql{ST}{STb} is linear in $\row \phi n$ and its complex conjugates. 
 
\begin{edesc} \label{STquest} \item  \label{STquesta} How would you detect whether a component $\sH_0\subset \sH(G,\bfC)^{\inn,\red}$ contains a dense set of points whose Jacobians are Shimura-Taniyama? 
\item  More generally: As an example related to Schottky's problem (Quest.~\ref{schottky}) and to the Coleman-Oort conjecture, when is an \ST\ variety the Jacobian of a curve? \end{edesc}  
 
 \subsubsection{Using Wohlfahrt's Theorem, \cite{Woh64}} \label{WohlfahrtThm}  We noted that most reduced Hurwitz spaces with $r=4$ (appearing after projective normalization as covers of $\prP^1_j$) are {\sl not\/} modular curves. Relevant to discussing \ST\ varieties, this section, based on computing cusps, shows how to give an example. 
 
  \begin{thm} \label{andres}  \cite[Prop.~6.12]{lum} is the case $\ell=2$ of \S\ref{heiscase}, $\sH(A_4,\bfC_{\pm 3^2})^{\inn,\rd}$ has two components; called there \HM\ and \DI\ components, but labeled here as $\sH_+$ and $\sH_-$ corresponding to their lift invariant values being $\pm 1$. Each is embedded in $\prP^1_j\times\prP^1_j$, but neither is a modular curve.  Andr\'e's Thm. \cite{An98} says there are no accumulation points in either component, off the cusps, whose Jacobians are \ST. Ex.~\ref{COvsAO} notes only $\sH_+$ supports a \MT, and so is relevant to the Coleman-Oort Conjecture, while $\sH_-$ is still relevant to Andr\'e-Oort. \end{thm} 

The remainder of this section proves Thm.~\ref{andres}. 
For $ \Phi^\rd:  \sH^\rd\to U_\infty$, a reduced Hurwitz space covers, let $\Gamma  \le \SL_2(\bZ)$ define it as an upper half-plane 
quotient $\bH/\Gamma$ (\cite[\S 2.10]{BFr02}). Let  $N_\Gamma$ be the least common
multiple (lcm) of its cusp widths;  the lcm of the ramification orders of points of the compactification 
$\bar \sH^\rd$ over $j=\infty$ (lcm of the orders of $\gamma_\infty$ on reduced Nielsen classes, \S\ref{genformr=4}).  

\begin{thm}[Wohlfahrt] \label{wohlfahrt} $\Gamma$ is congruence if and only if it contains the congruence subgroup, $\Gamma(N_\Gamma)$, defined by
$N_\Gamma$.\end{thm}  Using Thm.~\ref{wohlfahrt} to  show (some) $j$-line covers aren't  modular.  Compute $\gamma_\infty$ orbits on $\ni^\rd$. Then, check their distribution among $\bar
M_4=\lrang{\gamma_\infty,\sh}$ orbits ($\sH^\rd$ components). For each $\sH^\rd$ component
$\sH'$,  check  the lcm of $\gamma_\infty$ orbit lengths to compute $N'$, the modulus as if it were a modular curve. Then, see whether a
permutation representation of $\Gamma(N')$ could produce 
$\Phi': \sH'\to \prP^1_j$, and the type of cusps now computed.   

Use notation of \S\ref{redfinmod}. 
\cite[Prop.~3.5]{Fr10} has the \sh-;incidence diagram on the Nielsen class $\ni(A_4,\bfC_{\pm 3})^{\inn,\rd}$ with the detailed calculations and explanation for it in \cite[\S 3.3.2]{Fr10}. Reduced classes are given by modding out by $\sQ''$ on inner Nielsen classes. The $\gamma_\infty$ orbits appear in two blocks with those in the first block labeled $\cO_{1,1}^4,  \cO^3_{1,3}, \cO^3_{3,1}$ and those in the second block labeled $\cO^4_{1,4}, \cO^1_{3,4}, \cO^1_{3,5}$: with each labeling along the top and left side. The integers in each square matrix indicate a pairing between two such orbits $\cO$ and $\cO'$ given by computing the cardinality of the intersection of $\cO$ and the shift applied to $\cO'$. Because $r=4$, these are square matrices.  

The blocks correspond to lift invariant values of $\pm 1$. The first contains the \HM\ reps. whose orbits are $\cO_{1,1}^4$ and $\cO_{3,1}^3$. The second block contains  the \DI\ element whose cusp is labeled $\cO_{1,4}^4$. The superscripts are the lengths of the orbits, or the cusp widths, and the degree of the cover is given by summing the cusp widths in a block. 
Note: Neither of $\sH_0^{\inn,\rd,\pm}$ have reduced fine moduli. The Nielsen braid orbit for
$\sH_0^{\inn,\rd,-}$ (resp.~$\sH_0^{\inn,\rd,+}$) fails \eql{failfm}{failfma} (resp.~and also \eql{failfm}{failfmb}): 
\begin{edesc} \label{failfm} \item  \label{failfma} $\sQ''$ has length 2 (not 4 as required in \eql{redfinemod}{redfinemoda}) orbits; and  
\item  \label{failfmb} $\gamma_1$ has a fixed point (contrary to
\eql{redfinemod}{redfinemodb}). \end{edesc} 

This gives all the data required for applying the genus formula of \eqref{RHOrbitEq}. 
 
\begin{prop} \label{A4L0}  The two
$\bar M_4$ orbits on $\ni(A_4,\bfC_{\pm 3^2})^{\inn,\rd}$, $\ni^+_0$ and $\ni^-_0$, having respective
degrees 9 and 6 over $U_j$, and their normalized completion both have genus 0.  Both have
natural covers 
$\bar
\mu^{\pm}:
\bar \sH^{\inn,\pm}_0\to \prP^1_j$ by completing the map -- using that both are families of genus 1 curves: \begin{equation} \label{ja} \bp\in \sH^{\inn,\rd,\pm}_0\mapsto
\beta(\bp)\eqdef j(\Pic(X_\bp)^{(0)})\in \prP^1_j.\end{equation} Then, this case's identification of inner and absolute reduced classes gives 
\begin{equation} \label{jb} \bp\in
\sH^{\inn,\rd,\pm}_0\mapsto (j(\bp),j(\Pic(X_\bp)^{(0)})),\end{equation} a birational embedding of
$\bar \sH^{\inn,\rd, \pm}_0$ in
$\prP^1_j\times \prP^1_j$. Neither is a modular curve.\end{prop}

\begin{proof} The only point not proved is that neither is a modular curve. In the case of $\sH_0^{\inn,\rd,+}$ (resp.~$\sH_0^{\inn,\rd,-}$)) Thm,~\ref{wohlfahrt} says $\PSL(\bZ/12)$ (resp.~$\PSL(\bZ/4)$ would have an index 9 (resp.~6) subgroup, and that index would divide the order of the group. \eqref{modct} is an algorithm for computing the order of $\GL_n(\bZ/N)$ from which you see  we don't have $9| |\PSL(\bZ/12)|$, et. cet.
\begin{edesc} \label{modct} \item From linear algebra the determinant, $D_M$, of a matrix $M$ with entries in $\bZ/N$ is invertible if and only if the columns of the matrix generate the $\bZ/N$ module.
\item Chinese remainder theorem: $D_M$ is invertible if and only if it is invertible modulo each prime power dividing $N$, reverting the calculation to the case $N$ is a prime power.
\item  \label{modctc} With $N=p^u$, use that $(\bZ/p^u)^k\to (\bZ/p)^k$ is a Frattini cover. \end{edesc} 
Starting with \eql{modct}{modctc}, use the standard algorithm for counting invertible transformations of basis vectors over a finite field, and go back up the ladder of \eqref{modct}.  \end{proof}
 
\begin{rem}[Conjectures related to \CM\ jacobians from Wikipedia] \label{AOB}  Shimura wrote many papers on variants of the Siegel Upper-half space, say \cite{Sh70}. Variants of these conjectures use Shimura varieties. Appropriate for us are these statements for sufficiently large $\geng$:
\begin{guess}[Coleman-Oort]  Coleman: Only finitely many smooth projective curves of genus $\geng$ have  Jacobians of \ST\ type. Oort generalization:  The {\sl Torelli locus\/} -- of abelian varieties of dimension $\geng$ -- has no special subvariety of dimension $> 0$ that intersects the image of the Torelli mapping in a dense open subset. 

For properties of the relation between the space of Jacobians of curves and the moduli space of curves, see discussions of the Torelli map between them \cite{tor}. \MT s, with its emphasis on describing spaces using the braid action on Nielsen classes, is not asking the same kind of questions. For example, while our conditions for fine moduli are stated group theoretically, the Torelli type conditions are about general loci where fine moduli may not hold. \end{guess} \end{rem}

\section{Hurwitz space Components from  Thm.~\ref{cosetbr}} \label{examples}  Our example series of applying Thm.~\ref{cosetbr} {\sl start\/} with one in the series having $\sH(G,\bfC)^\abs$ a space of genus zero covers, even the first (Ex.~\ref{nofine}) modular curve-related example. 
\S\ref{bkliftinv}  explains what we need from Frattini covers, Schur multipliers and the lift invariant. All cases have the order of the Schur multiplier of $G$, $\SM_G$, prime to $N_{\bfC}$ (the \lcm\ of orders of elements in $\bfC$). They have nontrivial, but cyclic, Schur multipliers, for which we understand the moduli definition fields of Schur-separated components.  There is a prime, $\ell$ (explained in each example) related to a specific system of groups.
We assume $G$ is $\ell$-perfect (\eqref{lperfect}; e.g. not abelian).

\S\ref{weilpairing}  connects the lift invariant to the {\sl Weil Pairing\/} as it arises in the moduli definition field of spaces appearing in Serre's \OIT. \S\ref{absinnAn} puts an umbrella over the literature (from Serre, Liu-Osserman, and the author) on Hurwitz spaces starting with $G$ an alternating group, especially where covers in the Nielsen class $\ni(G,\bfC)^\abs$ have genus 0.  Examples show cusps on general reduced Hurwitz spaces can have more intricate structures than they do on modular curves. 

\S\ref{absinnAn} calculates in $A_n$ by multiplying permutations. In \cite{BFr02}, we operated from the right on letters of a permutation representation.  Here, though, we operate from the left. 
Example: In considering the middle product of $\HM_1$ in Lem.~\ref{outerinn}, the result is $$(1\,2\,\dots\,{\scriptstyle\frac{n\np1}2})\cdot ({\scriptstyle \frac{n\np1}2\,\frac{n\np3}2}\,\dots\,n)=(1\,2\,\dots\,n\nm1\,n).$$  Operating from the left on integers, that is the correct product, but not from the right. 

Recall previous notation. An {\sl absolute Hurwitz space component\/}, $\sH'$, corresponds to \begin{equation} \label{absequiv} \text{ a braid orbit, $O'$,  in $\ni(G,\bfC)^\abs\eqdef \ni(G,\bfC)/N_{S_n}(G,\bfC)$.}\end{equation}  Then, $\Psi_{\abs,\inn}: \sH(G,\bfC)^\inn \to \sH(G,\bfC)^\abs$ sends an inner   component $\sH\subset \sH(G,\bfC)^\inn$ -- corresponding to an inner braid orbit $O$ lying above $O'$ -- by restriction $\sH\to \sH'\subset \sH(G,\bfC)^\abs$. 

 \S \ref{heiscase}, with $G=(\bZ/\ell^{k\np1})^2\xs\bZ/3$, starts with a procedure (\S\ref{nonbraidcomps}) for finding the Schur multiplier (and so a non-trivial lift invariant) when $G$ is an $\ell$-split group.\footnote{The procedure uses two brilliant theorems from modular representation theory, Heller's and Jennings, which I learned to use from interacting with the authors of \cite{Be91}, \cite{S05} -- The author was my student at UCI. He started with a hint based on the main example of \cite{BFr02} -- and \cite{Se88}.}   Indeed, this case and that of \S\ref{weilpairing} appear similar: the kernels of the splittings have the same $\ell$-groups, $H=(\bZ/\ell^{k\np1})^2$.  But the Hurwitz space components that arise are different, and the seemingly trivial $H$ is deceiving. 
 
In this case, we start to compare decomposition groups in a \MT\ with the Coleman-Oort Conjecture, Rem.~\ref{AOB}, and Serre's \OIT\  by asking about the two types, \HIT\ and \ST\ decomposition groups in the \MT\ fibers  that arise. 

\eqref{psig} defines $\sM_{\bg}$ as the extension of $G$ given by branch cycles $\bg$ with a tail, the $\ell$-adic cohomology of $\hat X_\bg$. Now figure the relation with the quotient of the Universal abelianized $\ell$-Frattini cover which has a lattice used to define a \MT\ whose level 0 inner space contains $\hat X_\bg\to \prP^1_z$. The goal is to find braid orbits of all homomorphisms of $\sM_{\bg}\to G$ to ${\tG \ell}_{\ab}\to G$, with lattice kernel $L_{\ab, G}$ for abelianized \MT s. The $k$ level modules $\ell^kL/\ell^{k\np1}L$ are the same as $G$ modules and equal to $M_1$.

But there are other Frattini quotients of ${\tG \ell}_{\ab}$, extensions of $G=G_0$,  that you can use in place of  ${\tG \ell}_{\ab}$, on which the braid group acts. These arise by taking any $\bZ/\ell[G]$ quotient, $M_{1,q}$, of $M_1$, forming this at each level $k$, giving the extension $L_{\ab,G,c} \to {\tG \ell}_{\ab,c}\to G$ which inherits all the branch cycle properties of ${\tG \ell}_{\ab}$.  \eqref{mtcases} gives the smallest lattices we can take and Rem.~\ref{otherL} comments on how the role of the lift invariant changes if we use, say, the whole of ${\tG \ell}_{\ab}$ to form the lattice. 

We describe the $G=G_0$ lattice tails, $\sL_{\ab,G,q}$, in our examples, by listing the modules $M_{1,q}$: 
\begin{edesc} \label{mtcases} \item \label{mtcasesa} Serre's Case \S\ref{weilpairing} : $G_0=(\bZ/\ell^{k\np1})^2\xs \bZ/2$, $M_{1,q}=(\bZ/\ell)^2$.
\item \label{mtcasesb} Prop.~\ref{2'-2cusp}: $G_0=A_n$, $\ell=2$, $\bfC_{4\cdot \frac{n\np1}2}$. For $n=5$, $M_{1,q}$ has Loewy display $V_2\oplus V_2\mapsto \one_{A_5}$.\footnote{We know this for a few other values of $n\ge 4$, only. See \S~\ref{WohlfahrtThm} for $n=4$.}
\item \label{mtcasesc} \S\ref{heiscase}: $G_0=(\bZ/\ell^{k\np1})^2\xs \bZ/3$, $M_{1,q}=(\bZ/\ell)^2$.
\end{edesc} 
 The reduced Hurwitz space components, all of dimension 1, are modular curves only in the case \eql{mtcases}{mtcasesa}.  For all cases, the modules $M_1$  from the lattice  kernel of ${\tG \ell}_{\ab} \to G_0$ are indecomposable $G_0$ modules. As $G_0$ {\sl quotients\/} of $M_1$,  they are decomposable  for case \eql{mtcases}{mtcasesa} and for \eql{mtcases}{mtcasesc} when $\ell\equiv 1 \mod 4$. \cite[\S II.B]{Fr95} applies Heller's construction,  using projective indecomposable modules corresponding to the irreducible modules for $\bZ/\ell[G_0]$. Almost a formula for $M_1$: \cite[Proj.~Indecomposable Lem. 2.3 and \S II.C]{Fr95}, except it is difficult to compute projective indecomposables. 
 
 \cite[\S II.C]{Fr95} on the case $p=2$ lists the four simple $\bar \bF_2[A_5]$ modules: $\one_G$, reduction $\mod 2$ of the degree 4 summand of the standard representation and the two conjugate -- over $\bF_4$ -- adjoint representations using that $A_5=\SL_2(\bF_4)$. This gives the second Loewy layer of $M_1$, and shows $G_1\to G_0$ has kernel a 5-dimensional $\bZ/2[G]$ module with the Schur multiplier, $\bZ/2$, at its head. 
 
The remarks Rem .~\ref {Anissues} and \S \ref {Serre-book} show how our main examples extend Serre's \OIT, respectively, in considering the Nielsen classes related to alternating groups and the groups $(\bZ/\ell^{k\np1})^2\xs \bZ/3$.  

\subsection{Lift invariants and \OIT\ Nielsen class} \label{weilpairing} This section's  Nielsen class has a group with semidirect product with $\bZ/2$, a variant on that with the semidirect product with $\bZ/3$  in \cite[\S 5]{Fr20} and \S\ref{heiscase}. This is a different approach to a conclusion in Serre's \OIT\ traditionally from the Weil Pairing. Prop.~\ref{dihbr} uses  two systems of Hurwitz spaces given by their respective Nielsen classes: $$\begin{array}{c} \ni(G_{\ell,k,1_2}=\bZ/\ell^{k\np1}\xs \bZ/2,\bfC_{2^4})\eqdef\ni_{\ell,k,1_2}\text{ and }\\\ni(G_{\ell,k,2_2}=(\bZ/\ell^{k\np1})^2\xs \bZ/2,\bfC_{2^4})=\ni_{\ell,k,2_2}.\end{array}$$ The first is for the usual modular curves and for each $k$ it has one component, while $\ni_{\ell,k,2_2}$ is the main one of Serre's \OIT;  for each $k$ it has several components.   
 
 The rubric of Rem.~\ref{appcosetbr} is applied to find components of our main object of study; \eql{nil2}{nil2a} computes the class $\C_2$.  \S\ref{lpreliminary} is the preliminary setup -- since Serre's \OIT\ wasn't regarded as related to the lift invariant -- intended to also help with the superficially similar example of \S\ref{heiscase} which gets more deeply into the relation of the \OIT\ and hyperelliptic jacobians.\footnote{Serre, in private writings, considered generic extensions of his \OIT\ using hyperelliptic jacobians.} 

\S\ref{LISe} computes the lift invariants of these components to directly find their moduli definition fields, a result attributed by a different approach to Weil's $\ell$-adic pairing. Our computations will be done in semi-direct products of a group $M$ with a group $N$ acting on it, written $M\xs N$. We compute using the notation of $2\times 2$ matrices, with $m_2^{n_1}$ the action of $n_1$ on $m_2$, \begin{equation} \label{gpmult} \text{the product is }\smatrix {n_1} 0 {m_1}  1 \smatrix {n_2} 0 {m_2} 1 = \smatrix {n_1n_2}  0 {m_1^{n_2}\cdot m_2}  1\text{   and \lq\lq$\cdot$\rq\rq\  is group muliplication.}\end{equation}

\subsubsection{The \OIT\ Nielsen class} \label{lpreliminary} 
The lift invariant attached to these Nielsen classes comes from the {\sl small Heisenberg group}:
\begin{equation} \label{smheis} \bH_{\ell,k}\eqdef  \Biggl\{M(a,a',w)\eqdef \begin{pmatrix} 1 &a &w \\ 0& 1 &{a'} \\ 0& 0& 1\end{pmatrix}, a,a',w\in \bZ/\ell^{k\np1}\Biggr\}.\end{equation}  
We show  $\bZ/2$ in $G_{\ell,k,1_2}$ extends to $\bH_{\ell,k}$ with trivial action on the kernel of \footnote{In \S\ref{nonbraidcomps}, we use the notation $\bH_{\ell,k,2}$ to differentiate this extension from another representation cover, $\bH_{\ell,k,3}$, of $V_{\ell,k}$ on which there is a $\bZ/3$ action.}  $$M(a,a',w)\to (a,a')\in (\bZ/\ell^{k\np1})^2.\text{ With the $-1$  action given by ${}^\beta M(a,a',w)=M(-a,-a',w)$.}$$
\begin{equation} \label{cenext} \begin{array}{c}  \text{Check that  $\beta$ applied to } M(a_1,a_1',w_1)M(a_2,a_2',w_2)={}^\beta M(a_1,a_1',w_1){}^\beta M(a_2,a_2',w_2) \\ \text{ or } 
\begin{pmatrix}1 &-a_1\nm a_2 &w_1 \np w_2 \np a_1a_2'\\ 0& 1 &{-a'_1\nm a_2'} \\ 0& 0& 1 \end{pmatrix}   = \begin{pmatrix}  1 &-a_1 &w_1 \\ 0& 1 &-{a'_1} \\ 0& 0& 1 \end{pmatrix}  \begin{pmatrix} 1 &-a_2 &w_2 \\ 0& 1 &-{a'}_2 \\ 0& 0& 1\end{pmatrix}.\end{array} \end{equation} 
The following three statements -- shown in \S\ref{liftinvZellZ2} -- give the significance of this, starting with the distinction between absolute and inner classes, $\ni(\ell,k,2_2)^\dagger$, $\dagger=\abs$ or $\inn$.   
\begin{edesc}  \label{wp} 
\item  \label{wpa}  There are $\phi(\ell^{k\np1})$ braid orbits on the inner classes, $\ni(\ell,k,2_2)^\inn$, whose corresponding components are  conjugate by the action of $G(\bQ(\zeta_{\ell^{k\np1}})/\bQ)$.
\item \label{wpb}  The geometric (resp.~arithmetic) monodromy of the absolute, reduced, spaces as a cover of $\prP^1_j$ is $\SL_2(\bZ/\ell^{k\np1})$ (resp.~ $\GL_2(\bZ/\ell^{k\np1})$). 
\item \label{wpc}  The roots of 1 in \eql{wp}{wpa} arise from the lift invariant to the central extension in \eqref{cenext}. 
\end{edesc} 

Display elements of $\ni_{\ell,k,1_2}$ subject to product-one and generation as in Def.~\ref{HNC}. 
\begin{equation} \label{nil1}  \begin{array}{c} \text{With }A_{\ell^{k\np1}}\eqdef \{\ba=(\row a 4) \in (\bZ/\ell^{k\np1})^4\mid a_1\nm a_2\np a_3\np a_4\equiv 0\mod \ell^{k\np1}, \\ \lrang{a_i\nm a_j,1\le i<j\le 4} =\bZ/\ell^{k\np1}\}, \text{consider }\bg_\ba\in \ni_{\ell, k,1_2} \text{ given by}\\ 
\bigl(\smatrix {-1} 0 {a_1} 1,\smatrix {-1} 0 {a_2}  1, \smatrix {-1} 0 {a_3}  1, \smatrix {-1} 0 {a_4}  1\bigr).
\end{array}\end{equation} 
 
By substituting $(a_i,a_i')\text{ for }a_i, i=1,\dots,4$, in the above with $\wedge$ designating  the wedge product, define 
$\ni_{\ell,k,2_2}\eqdef \{\bg_{\ba,\ba'}\mid \ba\wedge \ba'\ne 0\}$.  The representation $T$ for absolute classes is on the cosets of $\smatrix {-1} 0 0 1$ (resp.~$\smatrix {-1} 0 {\pmb 0}  1$, with $\pmb 0 = (0,0)$) for $\ni_{\ell,k,1_2}$ (resp.~$\ni_{\ell,k,2_2}$). 

\begin{prop} \label{dihbr} The Nielsen class $\ni_{\ell,k,1_2}^\dagger$ has one braid orbit.  The action of $H_4$ on $\ni_{\ell,k,2_2}^\dagger$ extends its action on $\ni_{\ell,k,1_2}^\dagger$. \end{prop} 

\begin{proof} The first sentence is noted geometrically in \cite[Lem.~5]{Fr74}; with more arithmetic detail in \cite[Thm.~2.1]{Fr78} as a special case of a general problem. The second sentence is immediate from the definition of braid action. \end{proof} 

Using Prop.~\ref{dihbr}, compute braid orbits on $\ni_{\ell,k,2_2}^\inn$ by choosing any one allowable $\ba$. Then,  check possibilities for $\ba'$ that go with it. Start with $\ba\leftrightarrow$ shift of an \HM\ rep:
\begin{equation} \label{bash}  \ba_\sh =(0,a,a,0)\text{ with }a\in (\bZ/\ell^{k\np1})^* \text{ and }\ba\leftrightarrow \eqref{nil1} \text{ with }a_1=a_4=0, a_2=a_3=a. \end{equation} 

\begin{lem} \label{JacCond} Represent a class in $\ni_{\ell,k,2_2}^{\inn}$ by $\bg_{\ba,\ba'}$ modulo these conditions: 
\begin{edesc} \label{nil2} \item \label{nil2a}  $(a_1,a_1')=\pmb 0$ and $\sum_{i=2}^4 (-1)^i(a_i,a_i') \equiv {\pmb 0} \mod \ell^{k\np1}$; and 
\item \label{nil2b}  $\{(a_i,a_i')\mod \ell \mid 2\le  i\le 4\}$ aren't all on a line through $\pmb  0$. 
\end{edesc}  Starting with $\ba=\ba_\sh$, allowable $\ba'$, up to inner equivalence, have the form $$\{(0,a_2',a_3',a_3'-a_2' )\} \text{ with }a_3'\nm a_2'\not \equiv 0 \mod \ell.$$ 
\end{lem} 

\begin{proof}  For the 1st item of \eql{nil2}{nil2a}, replace the original element by the inner equivalent representative by conjugating by $\smatrix 1  0 {(a_1/2,a_1'/2)}  1$.  Since $$\smatrix 1 0 {(a_1/2,a_1'/2)} 1 \smatrix {-1} 0 {(a,a')}  1 \smatrix 1 0 {-(a_1/2,a_1'/2)}  1= \smatrix {-1} 0 {(a\nm a_1,a'\nm a_1')}  1,$$  
we may assume $(a_1,a_1')=\pmb 0$. Complete \eql{nil2}{nil2a} from product-one. 

Recognize \eql{nil2}{nil2b} as equivalent to this:  entries of $\bg_{\ba,\ba'}$ generate $(\bZ/\ell^{k\np1})^2\xs\bZ/2$. Given that the first entry is now $\smatrix {-1} 0 {\pmb 0}  1$, this says $\lrang{(a_i,a_i'),i=2,3,4}=V_{\ell,k}$.

Since $V_{\ell,k}$ is a Frattini cover of $V_{\ell,1}$, this is equivalent to showing the image of $\lrang{(a_i,a_i'),i=2,3,4}$ is all of $V_{\ell,1}$. For this, it suffices that in the 2-dimensional space $V_{\ell,1}$, the hoped-for generators aren't all on one line (through the origin). 

Now consider allowable $\ba'$ that go with $\ba_\sh$. Having the 4th entry nonzero $\!\mod \ell$ is necessary and sufficient for the second line condition of \eqref{nil2}; the first line is automatic from its form. 
\end{proof} 

\subsubsection{Values of the lift invariant} \label{liftinvZellZ2} \label{LISe} We show values of the lift invariant to the small Heisenberg group separate braid orbits on $\ni_{\ell,k,2_2}^\inn$. Indirectly, this accounts for the constants that come from $\bQ(e^{2\pi i/\ell^{k\np1}})$ traditionally arising from the {\sl Weil pairing\/}. These now interpret as values of a Nielsen class lift invariant, as given in Def.~\ref{liftinv}. 

Lem.~\ref{Normact} identifies the action of the normalizing group $N_{S_n}(G)$ on Nielsen classes; the effect of conjugating by  elements of $S_n$ that normalize  
\begin{equation} \label{cosets} G_{\ell,k,2_2} =\bigl\{ \smatrix {\pm1} 0 {(a,a')}  1 \mid (a,a')\in (\bZ/\ell^{k\np1})^2\bigr\}.\end{equation} 
 \eqref{cosets} lists the left cosets of $\bZ/2$ running over $(a,a')$, the matrices $M_{a,a'}\eqdef \smatrix 1 0 {(a,a')} 1$ multiplied on the left of the copy of $\bZ/2$, represented by $\{\smatrix {\pm 1} 0 {\pmb 0} 1\}$. 
\begin{lem} \label{Normact} The actions of the normalizer of $G_{\ell,k,2_2}$ in $S_{\ell^2}$, $N_{S_{\ell^2}}(G_{\ell,k,2_2})$, identifies with conjugations by $\GL_2(\bZ/\ell^{k\np1})$. The cosets of  $\SL_2(\bZ/\ell^{k\np1})$ in $\GL_2(\bZ/\ell^{k\np1})$ are represented by the matrices $\smatrix b 0 0 1$, $b\not \equiv 0 \mod \ell$,  with the action of conjugation given by $\smatrix {-1} 0 {(a,a')}  1\mapsto \smatrix {-1} 0 {b(a,a')}  1$.   \end{lem} 

\begin{proof}  If conjugation by  $\gamma$ normalizes $G_{\ell,k,2_2}$,  then it normalizes the characteristic subgroup $(\bZ/\ell^{k\np1})^2$. So it gives an element of $\GL_2(\bZ/\ell^{k\np1})$.  Multiplying  $\smatrix {b^{-1}} 0 0 1 \smatrix {-1} 0 {(a,a')}  1 \smatrix b 0 0 1$ gives the result $\smatrix {-1} 0 {b(a,a')}  1$, concluding the proof. \end{proof} 

Prop.~\ref{GL2} first computes the lift invariant;  \eqref{N2props} shows how the braid orbits on $\ni(G_{\ell,k,2_2},\bfC_{2^4})^\dagger$ fulfill the situation in Thm.~\ref{cosetbr}. Use the notation $M(a,a',w)$, $w\in \bZ/\ell^{k\np1}$ compatible with \eqref{smheis} for an element in $\bH_{k,\ell}\xs \bZ/2$ above $M(a,a')$. 

\begin{prop} \label{GL2} Order 2 elements $\smatrix {-1} 0 {M(a,a',w)}  1\in \bH_{k,\ell}\xs \bZ/2$ have $w=\frac{aa'}2$.

\begin{edesc} \label{aa'} \item \label{aa'a} Since every braid orbit contains an element  $\bg_{\ba_\sh,\ba'}$, to compute all lift invariant values it suffices to compute $s_{\bg_{\ba_\sh,\ba'}}$ with $\ba'=(0,a_2',a_3',a_2'\nm a_3')$ and $a_2'\ne a_3'\mod \ell$.    
\item \label{aa'b}  The lift value from \eql{aa'}{aa'a} is $a(a_3'\nm a_2')$, running over all values in $(\bZ/\ell^{k\np1})^*$ as $\ba'$ varies. \end{edesc} 

\begin{edesc} \label{N2props} \item \label{N2propsa}  There are two braid orbits on $\sH_{\ell,k,2_2}^{\abs}$; above each the inner components have lift invariants that differ by the square of elements in $(\bZ/\ell^{k\np1})^*$.   
\item \label{N2propsb} the inner Hurwitz space components are conjugate by the action of $G(\bQ(e^{2\pi i/\ell^{k\np1}})/\bQ)$, so  $\bQ(e^{2\pi i/\ell^{k\np1}})$  is their moduli definition field; 
\item \label{N2propsc}  and the geometric (resp.~arithmetic) monodromy group of any $\ni_{\ell,k,2_2}^{\inn,\red}$ components over $\prP^1_j$ is $\SL_2(\bZ/\ell^{k\np1})$ (resp.~$\GL_2(\bZ/\ell^{k\np1})$). \end{edesc} 
\end{prop} 

\begin{proof} 
An order 2 lift, $\smatrix {-1} 0 {M(a,a',w)} 1$, of $\smatrix {-1} 0 {(a,a')}  1$ to $\bH(\bZ/\ell^{k\np1})\xs \bZ/2$ from \eqref{cenext} satisfies   $$\smatrix {-1}0 {M(a,a',w)}  1 \smatrix {-1}0 {M(a,a',w)}  1= \smatrix 1 0 {M(-a,-a',w)M(a,a',w)}  1=\smatrix 1 0 {M(0,0,0)}  1.$$
Calculate: $M(-a,-a',w)M(a,a',w)$ has $2w\nm aa'$ in its upper right corner, or $w=\frac{aa'}2$, as stated. 

Use \eql{aa'}{aa'a}, to show \eql{aa'}{aa'b}.  In  the product of  order 2 lifts of $\bg_{\ba_\sh,\ba'}$ entries to  $\bH(\bZ/\ell^{k\np1})\xs \bZ/2$, with $\ba'=(0,a_2',a_3', a_3'-a_2')$, lift invariants run over the $w$ value in the lower left matrix  
$$\smatrix {-1} 0 {M(0,0,0)} 1 \smatrix {-1} 0 {M(a,a_2',\frac{aa_2'}2)} 1 \smatrix {-1} 0 {M(a,a_3',\frac{aa_3'}2)} 1 \smatrix {-1} 0 {M(0,a_3'\nm a_2',0)} 1.$$ 

Multiply the first two matrices, then the last two matrices. This gives  
$$\smatrix 1 0 {M(a,a_2',\frac{aa_2'}2)} 1 \smatrix 1 0 {M(-a,-a_2',\frac{aa_3'}2)}  1  .$$ 

Conclude the lift invariant value is $aa_3'/2 \np aa_3'/2 \nm aa_2'=a(a_3'\nm a_2')$,  an element in $(\bZ/\ell^{k\np1})^*$ according to the conditions of Lem.~\ref{JacCond}.   Lem.~\ref{Normact}  gives the normalizer of $G_{\ell,k,2_2}$ as $\GL_2(\bZ/\ell^{k\np1})$. Its action on Nielsen class elements fixing $\bg_\sh$ allows us to take $\ba'$ anything off of $\ba_\sh$. From the formula for the lift invariant, it clearly takes on all values in $(\bZ/\ell^{k\np1})^*$, giving the full action, as required by Def.~\ref{liftinv2}, of the normalizer. That concludes the proof of \eql{aa'}{aa'b}. 

From Lem.~\ref{Normact}, the action of the normalizer represented by  elements $\smatrix b 0 0 1$ has orbits on the lift invariants given by multiplying by $b^2$. This gives \eql{N2props}{N2propsa}.

We give the effect of  $H_4$ generators -- \S\ref{heisgens} has more details on this action -- on the 2nd and 3rd entries of $\bg_{\ba_{\sh},\ba'}$, after conjugatiing by 
$\smatrix 1 0 {(0,\frac{a_3'\nm a_2'}2)} 1$ to have $\smatrix {-1} 0 {(0,0)} 1 $ in the first entry: 
\begin{equation} \label{OITbraid} \begin{array}{rl}  \sh: \bg_{\ba_{\sh},\ba'}  \to & (\bullet, \smatrix {-1} 0 {(0,a_3'\nm a_2')}  1, \smatrix {-1} 0 {(-a,a_3'\nm 2a_2')} 1, \bullet )\\
q_2: \bg_{\ba_{\sh},\ba'}  \to & (\bullet, \smatrix {-1} 0  {2(a,a_2')\nm (-a,\nm a_3')}  1, \smatrix {-1} 0 {(a,a_2')} 1,\bullet).  \end{array} \end{equation} 

That is, $\sh$ is represented by $\smatrix {-1}{-2} 1 1$ and $q_2$ is represented by $\smatrix 2 {1} {-1} 0$. The square of $\smatrix {-1}{-2} 1 1$ is $-I_2$. Multiply $q_1q_2q_1=q_2q_1q_2$ by $q_2^{-1}$ to get $q_1q_2$. That acts as $$\smatrix {\nm1}{\nm2} 1 1 \smatrix 2 1 {\nm1} 0{}^{-1}= \smatrix {\nm 2}{\nm3} 1 1.$$ Check this has order 3. Therefore, elements of respective orders 3 and 2, independent of $\ell$, represent the actions of $\gamma_0$ and $\sh$. 

So, as expected, they give generators for $\SL_2(\bZ/\ell^{k\np1})$ and thereby give \eql{N2props}{N2propsa} and geometric monodromy statements of the rest of  \eqref{N2props} The arithmetic monodromy statements of \eql{N2props}{N2propsb} and \eql{N2props}{N2propsc} are a special case of Cor.~\ref{compmult2} applied to this case of a cyclic Schur multiplier. That concludes the proof. 
\end{proof} 

\subsection{Absolute vs Inner spaces when $G=A_n$} \label{absinnAn}  
\S\ref{Anabs} considers the spaces $\sH(G,\bfC)^\abs$ with $G=A_n$, $T$ the standard representation and $\bfC$ consisting of $2'$ conjugacy classes (elements of odd order). The Schur multiplier is well known to be $\bZ/2$, for $n\ge 4$, and its presence is graphically clear from the covering $\SL_2(\bF_4)\to \PSL_2(\bF_4)=A_5$ as below \eqref{mtcases}. The situation in applying Thm.~\ref{cosetbr} simplifies: $N_{S_n}(A_n)=S_n$, so $N_K/A_n=\bZ/2$ and this will have trivial orbits on any lift invariants.   Def.~\ref{schur-sep} is very simple in this case: Two braid orbits are Schur-separated if they respectively have lift invariants 0 and 1. The biggest issues are these:  
\begin{edesc} \label{Anhyp}  \item \label{Anhypa} Are all components Schur-separated \eqref{schursep}?
\item If not \eql{Anhyp}{Anhypa}, are there above a Hurwitz space component $\sH'\le \sH(A_n,\bfC)^\abs$ two components $\sH_j\le \sH(A_n,\bfC)^\inn$, $j=1,2$, so conjugate by $S_n/A_n$.\end{edesc}  With high multiplicity  in $\bfC$ (Def.~\ref{highmult}), then  \eqref{liftinvusecom} says there are precisely two (inner or absolute) components. Yet, when absolute covers have genus 0 (so they don't have high multiplicity), we never achieve both lift invariants. 

Thm.~\ref{LO08Se90Fr10} lists results  that start with precursors from Fried, Liu-Osserman and Serre. 
\S\ref{Aninn} analyzes what happens with the inner spaces corresponding to the Nielsen class hypotheses of the results above, where the absolute spaces have one component. In the case of two components, determining the moduli definition field extension of $\bQ$ of these components can be described using discriminants of specific covers in the corresponding absolute classes (Rem.~\ref{Anissues}).

Prop.~\ref{2'-2cusp} uses special Liu-Osserman Nielsen classes to give examples of nontrivial Modular Towers generalizing the main example of \cite{BFr02}. This relates the main theme of this paper to identifying this special case of \eqref{IGPstatement}: 
\begin{equation} \label{Anregreal} \begin{array}{c} \text{where would you find {\sl any\/} $\bQ$ regular realizations} \\  \text{of the characteristic 2-Frattini covers of $A_n$.} \end{array} \end{equation}  

\begin{rem}[Being explicit about \eqref{Anregreal}]  Suppose for some $n\equiv 1 \mod 4$, we could realize all the regular realizations of \eqref{Anregreal} with a uniform bound, $B_n$ on the number of their branch points. A special case of \cite[Thm.~4.4]{FrK97} says there is a \MT\ with each of those regular realizations corresponding to a $\bQ$ point on that tower. 

The Main \MT\ conjecture \cite[Main Conjecture 1.4]{FrK97}, though, says this is not possible, a conjecture generalizing Mazur's Theorem on $\bQ$ points on modular curves, a consequence compatible with generalizing Falting's Theorem. No one has regularly realized even $A_5$ and the exponent 2, 2-Frattini cover ${}_2\psi_5: {}_2 \bar A_5\to A_5$ (with kernel $(\bZ/2)^5$) described in \cite[Prop.~2.4]{Fr95}.\footnote{As special cases of general results,  ${}_2 \bar A_5$ is centerless and $\ker({}_2\psi_5)$ is indecomposable \cite[Lem.~2.4]{FrK97}.} 
\end{rem} 

\subsubsection{$G=A_n$ and absolute spaces} \label{Anabs}  For a conjugacy class $\C$, indicate its cycle type by $(\row u t)$. 

\begin{exmpl} \label{Anexmp}   \eqref{3cycle} summarizes the main example of \cite{Fr10} with $(G,\bfC)=(A_n,\bfC_{3^r})$, $r$ repetitions of 3-cycle classes, $n\ge 5$.\footnote{It also does the boundary examples $r=n\nm1$ and $n=4$.} 
\begin{edesc}  \label{3cycle}  \item \label{3cycle0} Applying \eql{modbyinner}{modbyinnerb}, for any $h\in S_n$ there is a braid from $\bg$ to $h\bg h^{-1}$ if and only if there is such a braid for one case of  $h\in S_n\setminus A_n$. 
\item  \label{3cyclea} If \eql{3cycle}{3cycle0} holds, $\ni(A_n,\bfC)^\abs$ and $\ni(A_n,\bfC)^\inn$ both have only one braid orbit.
\item \label{3cycleb} \eql{3cycle}{3cyclea} holds for $r=n\nm1$ on $\ni(A_n,\bfC_{3^r})^\dagger$, $\dagger=\abs$ or $\inn$.\footnote{For $\bg\in \bfC_{3^r}$ it is not necessary to include that $\lrang{\bg}=A_n$, just that $\lrang{\bg}$ is transitive.} 
\item  \label{3cycled}  With $r\ge n$, \eql{3cycle}{3cycleb} holds by replacing \lq\lq one braid orbit\rq\rq\ with exactly \lq\lq two Schur-separated (braid) orbits\rq\rq\ on $\ni(A_n,\bfC_{3^r})^\dagger$, $n\ge 5$. \end{edesc}  

\cite{Se90} or \cite[Cor. 2.3]{Fr10} gives the circumstance of the initial collaboration between the author and Serre; giving the lift invariance formula of Thm.~\ref{invariance}.  \end{exmpl}

\renewcommand{\pu}{\text{\bf pu}}
 In our usual notation $\ni(G,\bfC,T)$, refer to a conjugacy class in $G\le S_n$ as {\sl pure-cycle\/} if its elements have only one cycle of length exceeding one under the representation $T$.
\begin{defn} \label{purecycles}  \cite{Wm73}: If a non-cyclic $G$ is primitive and contains a pure-cycle,  then $G$ is $A_n$ or $S_n$. An element $g=(\row u t)\in G$ defines the collection of pure cycles $\row \C t$ in the group, $G_\pu$,  generated by all disjoint cycles in elements of $\bfC$. Refer to  $\ni(G,\bfC)^\abs$ as {\sl pure-cycle\/} if all conjugacy classes are pure-cycle and $\bfC$ as odd-cycle, if all $g\in \bfC$ have odd order. 
\end{defn} 

\begin{lem} \label{canpure} Given $\ni(A_n,\bfC)^\abs$, there is a canonical pure-cycle Nielsen class, $\ni(G_\pu,\bfC_\pu)^\abs$ attached to it in the group $G_\pu$ generated by the pure cycles of elements $g\in \bfC$.  

Then, $G_\pu=A_n$ if and only $\bfC$ is odd-cycle. For $G_\pu=A_n$, covers in $\ni(A_n,\bfC_\pu)^\abs$ and in $\ni(A_n,\bfC)^\abs$ have the same genus. Lift invariants of $\ni(A_n,\bfC_\pu)^\abs$  contain those of $\ni(A_n,\bfC)^\abs$. 
\end{lem} 

\begin{proof} Given disjoint cycle notation for $g=(u_1)\dots (u_t))\in G$, define the pure cycle classes associated to $g$ as $\row \C t$. Applying Riemann-Hurwitz, the genus for the two Nielsen classes is the same; the non-zero contributions to the genus, in both cases,  run over disjoint cycles, and those are the same for elements in the respective Nielsen classes. 

If $\bfC$ is not odd-cycle, then there is $g\in \bfC$ containing an even pure-cycle, and that would give an element in $\bfC_\pu$ that is not in $A_n$. Thus,  $G_\pu$ must be $S_n$. This leaves considering $\bg\in \ni(A_n,\bfC)^\abs$, with lift invariant $s_\bg$, whether we can realize that lift invariant in $\ni(G_\pu,\bfC_\pu)^\abs$. 

For $g$ as above, consider the commuting elements -- by abuse denoted as above $(u_i)$ -- and their respective lifts $\widetilde {(u_i)}$.  Since the classes are $2'$, there is a unique lift to $\tilde A_n$, as there is for any of the products $\widetilde {(u_i)}\widetilde {(u_{i\np1})}$. Therefore this is $\widetilde{(u_1)(u_2)}$. So, inductively replacing the lift of an entry $g$ in $\bg$ by the product of the lifts of its disjoint cycles doesn't change the lift invariant.\footnote{This isn't the correct calculation if $G_\pu=S_n$.}  \end{proof} 

For odd order  $g\in A_n$, $n\ge 3$, denote  the  count of length $u$  disjoint cycles in $g$ with $$\frac{(u-1)^2}{8}\equiv 1 \text{ or }2 \mod 4\text{ by }w(g).$$   
For $g \in A_n $ of odd order, let $\omega(g)$ by the sum $\frac{u^2\nm1}8\mod 2$ running of the lengths $u$ of the disjoint cycles of $g$.  The two proofs of Thm.~\ref{invariance} tie together the referenced articles. 

\begin{thm} \label{invariance} Assume $\bfC$ is odd-cycle. For $n\ge 3$, and any  $\bg\in \ni(A_n,\bfC)^\abs$ of genus 0, $$s_\bg=\sum_{i=1}^r  (-1)^{\omega(g)};   s_\bg\text{ is constant on the Nielsen class.}$$ 

Example: For $\phi: X\to\prP^1$  in  $\ni(A_n,\Ct {n\nm1})^\abs$, then   $X$ has genus 0, and $s_\phi=n\nm1 \!\!\mod 2$. 
 \end{thm} 

\begin{proof}  References at the end of Ex.~\ref{Anexmp} give one proof of the lift invariant result.  \cite[Cor. 2.3]{Fr10} gives a short proof of  \cite{Se90}, reverting it to the  example (original) above case, $\ni(A_n,\bfC_{n\nm1})^\abs$. 

Here is a 2nd proof. Assuming genus 0  for pure-cycle $\bfC_\pu$, \cite{LO08} says $\ni(A_n,\bfC_\pu)^\abs$ has one component. Thus, running over $\bg\in \ni(A_n,\bfC_\pu)^\abs$  the lift invariant has only one value. From Lem.~\ref{canpure}, the lift invariant has only one value running over $\bg\in \ni(G,\bfC)$. \end{proof}

The failure of Schur-separation of all components \eql{Anhyp}{Anhypa} as reverting to the pure-cycle case in Thm.~\ref{LO08Se90Fr10} generalizes. Cor.~\ref{LO08Se90Fr10} follows almost immediately from  Lem.~\ref{canpure} and the second proof of Thm.~\ref{invariance}. Rem.~\ref{purecycfail} adds comments for where to look -- in these Nielsen classes -- for its failure. 

\begin{cor} \label{LO08Se90Fr10} Assume  genus 0 for Nielsen class absolute covers and $\bfC$ is odd-cycle. Then, $\sH(A_n,\bfC)^\abs$ has exactly one component if and only if Schur-separation holds. 
 
Now consider the same hypotheses without the genus 0 assumption. Denote by $\ni(G,\bfC)^\abs_k$ the elements $\bg$ with $s_\bg=u$, $u\in \bZ/2$. There is one braid orbit on $\ni(G,\bfC)^\abs_k$, if and only if no other orbit has lift invariant $u$.  

If the Schur-Separation fails above, then it fails for the pure-cycle $\bfC_\pu$ associated to $\bfC$. From \eql{3cycle}{3cycled}, it does not fail for any $\bfC$ for which $\bfC_\pu$ is $\bfC_{3^r}$ for some $r$. 
\end{cor} 

Ex.~\ref{553} gives a Nielsen class of covers of genus $>0$ having just one value of the lift invariant for $G=A_n$ and $\bfC$ odd-cycle.
\begin{exmpl} \label{553} The lift invariant given below comes from \cite[Princ.~5.15]{BFr02}. There are two 5-cycle conjugacy classes in $A_5$, which we denote $\C_{+5}$ and $\C_{-5}$. The notation $\bfC_{+5-5\cdot 3}$ adds the class of a 3-cycle to this. Covers in the Nielsen class $\ni(A_5,\bfC_{5_+5_-3})^\abs$  have genus $\geng_{553}$   given by $$2(5\np \geng_{553}\nm1)=4+4+2=10,\text{  or }\geng_{553}=1.$$

For each ordering of the conjugacy classes 
$\bfC_{5_+5_-3}$, the nielsen class $\ni(A_5,\bfC_{5_+5_-3})^\inn$ 
has exactly one element, for a total of six elements. All representatives $\bg$ 
have $s_\bg=1$.  Note that by including both $\C_{5-}$ and $\C_{5+}$ this makes $\bfC$ a rational union of classes (Def.~\ref{ratclass}).  \end{exmpl}
 
\begin{rem}[Pure-cycle failure?]  \label{purecycfail}  In the first paragraph of Cor.~\ref{LO08Se90Fr10}, the Nielsen class assumes only one value. So if Schur-Separation holds, then there is only one braid component, etc. Using \cite{LO08} and Lem.~\ref{canpure}, in this case, Schur-Separation must hold. The argument of the lemma, though, didn't use genus 0. In the second paragraph, the only lift values are in $\bZ/2$, and we can therefore separate the braid orbits according to those with a given lift value. 

The second proof of Thm.~\ref{invariance} applies, but the strong conclusion does not, since \cite{LO08} did not prove a result that used the value of the lift invariant in place of the genus 0 condition. In private conversation, Brian Osserman didn't realize that formulating Schur-separation didn't require Galois covers. I told him a Schur-separation version of Lem.~\ref{homcovers}. 
\end{rem} 

\begin{rem}[Liu-Osserman on $S_n$] Liu-Osserman considered all pure-cycle Nielsen classes, including $G=S_n$. That works as above, except it doesn't have the possibility of a non-trivial lift value, nor the outer automorphism. I left it out, as a less interesting case of Thm.~\ref{cosetbr}.\end{rem} 

\subsubsection{$A_n$ and inner spaces}  \label{Aninn}  This subsection is dedicated to $G=A_n$, $T$ the standard degree $n$ rep. and odd-cycle covers, in search of automorphism-separated components on $\ni(A_n,\bfC)^\inn$. That is, we take up \lq\lq the top\rq\rq\ of Thm.~\ref{cosetbr} where we already know the components of $\sH(A_n,\bfC)^\abs$ and the question reverts to whether we can braid, $\alpha$, an outer automorphism from $S_n$ on $\ni(A_n,\bfC)^\inn$. According to \eql{3cycle}{3cycle0}, for this question we can take $\alpha$ any 2-cycle in $S_n$. Lem.~\ref{outerinn} says finding these reverts to the case of pure-cycle Nielsen classes. 

\begin{lem} \label{outerinn} As above, if you can braid the outer automorphism on $\ni(A_n,\bfC)^\inn$, then you can braid it on $\ni(A_n,\bfC_\pu)^\inn$. For example, if covers in $\ni(A_n,\bfC)^\abs$ have genus 0, then either $\ni(A_n,\bfC)^\inn$ has one component or two automorphism-separated components. Indeed,  when absolute covers have genus 0, and $r$ is even, it suffices to consider whether we can braid between the two \HM\ reps. For example, with $g_1=({\scriptstyle\frac{n\np1}2} \dots 2\,1)$ \and $g_3=({\scriptstyle \frac{n\np1}2\,\frac{n\np3}2}\,\dots\,n)$, can we braid between $\HM_1=(g_1,g_1^{-1}, g_3,g_3^{-1})$ and $\HM_2=(g_1', (g_1')^{-1}, g_3',(g_3')^{-1})$ with $g_i'=(1\,n)g_i(1\,n)$. \end{lem} 

\begin{proof} Use the canonical association of $\bg\in \ni(A_n,\bfC)^\inn$ and assume for $\alpha\in S_n\setminus A_n$, $(\bg)\alpha=(\bg)q$ for $q\in H_r$. Use the fact that the actions of $q$ and conjugating by $\alpha$ commute, and also with the ${}_\pu$ substitution. By leaving the disjoint cycles in place after the replacement $\bg\mapsto \bg_\pu$ as in the proof of Lem.~\ref{canpure}, apply Lem.~\ref{canpure} and \cite{LO08} to find $q_\pu$ for which $(\bg_\pu)\alpha=(\bg_\pu)q_\pu$. \end{proof} 

Prop.~\ref{An-Snarith} details what happens with examples of this section. It is elementary to check when $\bfC$ is a rational union (see Ex.~\ref{553}). We already noted the absolute space (in this case, and so the inner \eql{fmod}{fmodb}) has fine moduli. 
 
\begin{prop} \label{An-Snarith} Assume Schur-Separation holds for odd-cycle $\ni(A,\bfC)^\abs$. With $\bfC$ a rational union, consider the absolute-inner Hurwitz space cover $$\Phi_{\abs,\inn}:  \sH(A_n,\bfC)^\inn \to \sH(A_n,\bfC)^\abs.$$  From Cor.~\ref{LO08Se90Fr10},  $\sH(A_n,\bfC)^\abs$ has one (resp.~2) absolutely irreducible components according to the lift invariant is 0 (resp.~1) assumed on the corresponding braid orbit $\ni_k$, $k=0,1$, on $\ni(A_n,\bfC)^\abs$.  In either case, components with their configuration maps have moduli definition field $\bQ$ and  \begin{center} $\bp^\abs\in \sH(A_n,\bfC)^\abs$ represents a cover $\phi_{\bp^\abs}: X_{\bp^\abs}\to \prP^1_z$, defined over $\bQ(\bp^\abs)$. \end{center} 

Suppose, vis-a-vis $\Phi_{\abs,\inn}$, $\sH''$ is the pullback of a component, $\sH'\subset  \sH(A_n,\bfC)^\abs$. Since $N_{S_n}(A_n)/\Inn(G)=\bZ/2$ generated by any element of $S_n\setminus A_n$, \eqref{homeoorbitsu=1} gives this. Either:   
\begin{edesc} \label{AnSn} \item \label{AnSna} $\sH''$ is absolutely irreducible and restriction of $\Phi_{\abs,\inn}$ is Galois with group $\bZ/2$; or 
\item \label{AnSnb} $\sH''$ consists of two absolutely irreducible components, $\sH''_1$ and $\sH''_2$, both with moduli definition field $K/\bQ$, $[K:\bQ]\le 2$.\end{edesc} In case \eql{AnSn}{AnSna}, there is a Zariski dense subset of $\bp'\in \sH'(\bar \bQ)$ for which the cover $X_{\bp'} \to\prP^1_x$ has arithmetic Galois closure $S_n$ over $\bQ(\bp')$.  

For \eql{AnSn}{AnSnb}, whatever is $K$, the discriminant of a cover $\bp'\in \sH(A_n,\bfC)'(\bar \bQ)$ is a square in $K(\bp')$. So, if $K=\bQ$ and $\bp'$ has coordinates in $\bQ$, then the discriminant of $\bp'$ is a square in $\bQ$. 
\end{prop}  

\begin{proof} The statement on representation of the cover $\phi_{\bp^\abs}$ over $\bQ(\bp^\abs)$ is from \eqref{totfam}. From Cor.~\ref{corinn-absmdf},  either $\sH(A_n,\bfC)^\inn$ has two components, or $\Phi_{\abs,\inn}$ is Galois with group $\bZ/2$. Since $n\ge 4$, the normalizer of  $A_n(1)$ in $A_n$ is just $A_n(1)$ and both $\sH(A_n,\bfC)^\dagger$, $\dagger=\inn$ or $\abs$ have fine moduli as in Prop.~\ref{finemodabsinn}. From  Thm.~ \BCL\  \ref{bcl}, in case \eql{AnSn}{AnSna}, since $\sH''$ is absolutely irreducible, its moduli definition field is $\bQ$; in case \eql{AnSn}{AnSnb} the components are either permuted among each other, or they both have moduli definition field $\bQ$.  

Suppose \eql{AnSn}{AnSna} holds. Then, Hilbert's irreducibility theorem says the density is 1 (for essentially any density) of $\bp$ for which the cover over $\bp$ has arithmetic monodromy $S_n$ by the definition of $\sH(A_n,\bfC)^\inn$ in this case.  The statement on the discriminant is from algebraic number theory. When an extension is geometrically $A_n$, the discriminant tells you whether it is arithmetically $S_n$ by whether its square root extends the definition field. 
\end{proof}

\begin{exmpl}[Ex.~\ref{Anexmp} continued] \label{Anexmp2} \cite[\S2.10.1, Table 2]{BFr02} uses the {\sl \sh-incidence matrix\/} for $\ni(A_5,\bfC_{3^4})^{\dagger,\rd}$ with $\dagger=\abs$ and $\inn$. From this, we read off the cusps and genus of a cover. 
\cite[Prop.~2.19]{Fr20} does the same for $\ni(A_4, \bfC_{\pm 3^2})$, which is, for $\ell=2$, our main example, as in Prop.~\ref{A4L0}: two components, Schur separated, and both components at level 0 have genus 0.  \S\ref{sephmanddi} reminds of the \sh-incidence matrix and applies it for the main example of this paper. \end{exmpl} 

Applications required a precise (and somewhat long) version of the construction of  Nielsen classes representatives in Prop.~\ref{2'-2cusp}. So we left it to \cite{Fr27}, but indicate in the proof below examples of where an easy construction gives many \MT s. 
\begin{prop} \label{2'-2cusp} Let $\bd=\row d r$, $r\ge 3$,  with  $\ni_{\bd}^\abs$ a Nielsen class of odd pure-cycle genus 0 covers. Then,  $G= A_n$, $n\ge
4$. For $\ell=2$, there is a (nonempty) abelianized 
\MT\ above any component of 
$\sH(A_n,\bfC_{\bd})^\inn$ if and only if \begin{equation} \label{MTover} \sum_{i=1}^r \frac{o(g_i)^2-1}8 \equiv 0 \mod 2. \end{equation} For
$\ell\ne 2$, there is always an abelianized \MT\ above any component of $\sH(A_n,\bfC_{\bd})^\inn$. 

If the $d_i\,$s are equal in pairs, there is always (irrespective of $\ell$) a (full\wsp not abelianized) \MT\  over any component of $\sH(A_n,\bfC_{\bd})^\inn$. \end{prop} 

\begin{proof} Appearances of alternating groups come from \cite{Wm73}, whose hypotheses \cite[Thm.~5.3]{LO08}) imply a noncyclic, transitive subgroup
$G$ of $A_n$,   generated by  odd pure-cycles must be $A_n$, $n\ge 4$. If we exclude that $G$ is cyclic, then $G=A_n$, $n\ge 4$, in any
such Nielsen class. If, however, $G=\lrang{h}$, then transitivity implies $h$ is an $n$-cycle. Apply the pure-cycle and genus 0 conditions.  Conclude: all 
$g_i\,$s are invertible powers of $h$. By \RH: $2(n-1)=r(n-1)$, $r=2$, contrary to hypothesis. 

Why $\ni(A_n,\bfC_\bd)$ is nonempty: For $r=3$ and  $\g_{e_1\cdot e_2\cdot e_3}=0$, there is a unique  $$\bg\in \ni(G,\bfC_{e_1\cdot e_2\cdot
e_3})^\abs \text{ with }\ord(g_i)=e_i, i=1,2,3.$$
\cite[Princ.~4.9]{Fr27} constructs Nielsen class reps., for all $\bd$ satisfying the conditions above for $r=4$, it also notes their easy construction when the $d_i\,$s are equal in pairs through \HM\ reps. Then, 
and outside that case, it constructs special representatives having {\sl  split-cycle cusps}.   

The Schur multiplier for $A_n$ is $\bZ/2$. From Thm.~\ref{invariance}, the left side of \eqref{MTover} is the value of the lift invariant for $\ell=2$, and the lift invariant is trivial for $\ell\ne 2$. 
So, Prop.~\ref{existProp} says \eqref{MTover} gives an abelianized \MT\ for $\ell=2$ over a trivial llift invariant braid orbit of $\ni(A_n,\bfC_{\bd})$, and  such a \MT\  always exists when $\ell\ne 2$. 
\end{proof} 

\begin{rem}[$A_n$ component issues] \label{Anissues} \cite[\S3.4]{Fr27} does this in the case of $A_n$ with $n\equiv 1 \mod 4$ and $\bfC$ consists of four $\frac{n\np1}2$ cycles.  With notation from \eqref{AnSn}, the hardest  \cite{LO08} case, toward finding their absolute Hurwitz spaces had one component, was $\sH(A_n,\bfC_{(\frac{n\np 1}2)^4})^\abs$, $n\equiv 1 \mod 4$. 

\cite{Fr27} extends their paper to the inner case:  is the moduli definition field   $\bQ$ or a quadratic extension of $\bQ$? This  reverts to a property of an explicitly constructed function $f_n:\prP^1_x\to \prP^1_z$ in the absolute class, mapping $\{0,\infty,\pm 1\}\to \{0,\infty,\pm 1\}$: Is the discriminant of $f_n$ a square in $\bQ$? Note: We can compositionally iterate the $f_n\,$s. \end{rem}  

\subsection{A Nielsen class for $(\bZ/\ell^{k\np1})^2\xs \bZ/3=G_{\ell,k,3}$, $k\ge 0$} \label{heiscase}  
As for $G_{\ell,k,2_2}$ in \S\ref{weilpairing}, $G_{\ell,k,3}$ is solvable. Here, $\bfC=\bfC_{\pm 3^2}$, two repetitions each of the order 3 classes in $\bZ/3$; $\bZ/3$ acts by taking $A^*=\smatrix 0 {-1} {1} {-1}=\zeta_3=e^{2\pi i/3}$ acting on $\bZ^2=\sO_K$ -- left action as in linear algebra classes -- the algebraic integers of $\bZ[\zeta_3]$ on the right. Reducing $\mod \ell^{k\np1}$, $A^*$ on $V=\lrang{\bv_1=\zeta_3,\bv_2=\zeta_3^2}\otimes \bZ$. In matrix multiplication notation: $\tr=$  transpose turns a one-row vector to a one-column matrix:\footnote{This notation matches how elements in the Nielsen class multiply. The $\bZ/3$ action descends from an action on the free group on two generators,} \begin{equation} \label{A*act} \begin{array}{c}A^*\bv_1=A^*(1\  0)^\tr =(0\ 1)^\tr=\bv_2 \text{ and }\\A^*  \bv_2^\tr = -\bv_1^\tr-\bv_2^\tr =(-1,-1)^\tr \text{ from }\zeta_3^2\cdot \zeta_3=1=-\zeta_3-\zeta_3^2. \end{array}\end{equation}  

The representation $T$ is on the cosets of $\bZ/3=\{((0,0),\bZ/3)\}$ in $G_{\ell,k,3}$.  \S\ref{nonbraidcomps} shows the Schur multiplier of $G_{\ell,k,3}$ is nontrivial: giving an $\ell$-Frattini extension of the group with $\bZ/\ell^{k\np1}$ kernel in the center of the extension. It is, therefore, superficially similar to the OIT example of \S\ref{weilpairing}, but the $\ell$-Sylow of the restriction of its {\sl representation cover\/} is not $\bH_{\ell,k}$.  

While the lift invariant is our main separator of components, for some Hurwitz spaces, there can be more obvious geometric reasons why its components are dealt with in separate collections. \S\ref{HM-DIPrinc} collects components in a subspace, $\sH_{\HM-\DI}$, where the components (all reduced Hurwitz spaces of dimension 1) have compactifications over $\prP^1_j$ with a cusp -- over $j=\infty$ -- of width 1 (Lem.~\ref{HM-DIcomps}). These are of two such cusp types (\HM\ and \DI\ as in \eqref{hm-diHyp}).\footnote{We decided not to deal in this paper with whether there are other components, since these components provide all the lessons on lift invariants we could handle.}  
Ex.~\ref{HM-separated} explains the related main example of \cite{Fr95}, which led to the name \HM\ (Harbater-Mumford). 

Then, \S\ref{liftinvell03} (Lem.~\ref{l03liftinv}) computes the lift invariants of the components in $\sH_{\HM-\DI}$ achieving all possible values in $\bZ/\ell^{k\np1}$.  
Following Thm.~\ref{cosetbr} (rubric Rem.~\ref{appcosetbr}), we list absolute components, first from lift invariant values and then including the separation between \HM\ and \DI\ components, followed by listing the automorphism-separated components above each absolute component. 
 
Lem.~\ref{l03liftinv} (\S\ref{liftinvell03}) computes the lift invariant for Nielsen classes corresponding to components in $\sH_{\HM-\DI}$ ($\ell\ne 3$). The formula is explicit. At level $k$, those in  $(\bZ/\ell^{k\np1})^*$ are \DI\ orbits; \HM\ orbits have lift invariant 0, but so, too, do some \DI\ orbits. \eqref{GL3properties} gives the genuses of the covers in two of the relevant families. 
\begin{edesc}  \label{GL3properties} \item $\sH(\bZ/3, \bfC_{\pm 3^2})^\inn$ has covers of genus $\geng_{\bZ/3, \inn}=2$: $2(3 \np \geng_{\bZ/3,\inn}\nm 1)=4\cdot 2$.
\item $\sH(G_{\ell,k},\bfC_{\pm 3^2})^\abs$ has covers of genus $\geng_{G_{\ell,k},\abs}$: $$2((\ell^{k\np 1})^2\np \geng_{G_{\ell,k},\abs} \nm1)=4\cdot 2\frac{((\ell^{k\np 1})^2\nm 1)}3\text{ or }\geng_{G_{\ell,k},\abs}=\frac{\ell^{2(k\np1)}\nm1}3.$$ 
\end{edesc}

\begin{rem} \label{quadrecrem} Action of $\bF_\ell[\bZ/3]$  on $V_{\ell,0}$ has two 1-dimensional subspaces  if and only if $x^2\np x\np1$ --  irreducible for $\ell=2$ -- 
is reducible. For $\ell\ne 2$,  this is equivalent to $x^2+3$ is reducible:  equivalent to -3 is a square $\!\! \mod \ell$.  \eqref{quadrec}  applies quadratic reciprocity: $\left(\frac 3 \ell\right) \left(\frac \ell 3\right) = (-1)^{(\frac {(3 \nm 1)}2 \frac{(\ell\nm1} 2)}$.

\begin{edesc} \label{quadrec}  \item either -1 and 3 are both squares $\!\! \mod \ell \Leftrightarrow \ell\equiv 1 \!\!\mod 4 \text{ and } 1 \!\!\mod 3$, or 
\item neither -1 nor 3 are squares $\!\!\mod \ell \Leftrightarrow \ell\equiv 3 \!\!\mod 4$  and $1\!\!\mod 3$. \end{edesc} These conclusions from quadratic reciprocity imply $-3$ is a square $\mod \ell$,  $\Leftrightarrow \ell\equiv 1 \!\!\mod 3$.  \end{rem} 

\subsubsection{The Schur multiplier of $G_{\ell,k,3}$} \label{nonbraidcomps}  We use the matrix multiplication indicated in \eqref{gpmult}. An element $\bv\in V_{\ell,k}=(\bZ/\ell^{k\np1})^2$ is represented by $\smatrix 1 0 {\bv} 1$, $\alpha$ by $\smatrix {\alpha} 0 {\pmb 0} 1$ with $\pmb 0=(0,0) \in V_{\ell,k}$ with the conjugacy classes of $\alpha$ in $V_{\ell,k}\xs \bZ/3$ the set $\C_+=\{\smatrix {\alpha}0  {\bv^{\alpha}\nm \bv}  1 \eqdef {}_\bv\alpha\mid \bv\in V_{\ell,k}\}$ compatible with matrix multiplication and the notation for the \OIT\ group in \S\ref{liftinvZellZ2}. 
\begin{defn} \label{alpha-gen} Refer to   $\bv\in V_{\ell,k}$ as an $\alpha$-generator  if $\lrang{\alpha,\bv}=V_{\ell,k}\xs \bZ/3$.\end{defn} Denote $\bH_{\ell,k,2}$ for $\bH_{\ell,k}$ in \eqref{smheis} to indicate the representation cover with a $\bZ/2$ action on it.

\begin{lem} There is no extension of the action of $\alpha$  to $\bH_{\ell,k,2}$ to produce a central extension of $G_{\ell,k,3}$. Still, there is a central extension, $\bH_{\ell,k,3}\to V_{\ell,k}$, with kernel $\bZ/\ell^{k\np1}$ on which $\bZ/3$ acts,  producing the universal central extension $\bH_{\ell,k,3}\xs\bZ/3\to G_{\ell,k,3}$.  \end{lem} 

\begin{proof}  Try extending $\alpha$ to the small Heisenberg group acting trivially on the center:  substitute  $\smatrix {\alpha} {M(a,a',w)} 0 1$ for $\smatrix {-1} {M(a,a',w)} 0 1$ in the expression for $\beta$ in \eqref{cenext} to check if  
\begin{equation} \label{cenextalpha} \begin{array}{c}  \text{$\alpha$ applied to } M(a_1,a_1',w_1)M(a_2,a_2',w_2)={}^\alpha M(a_1,a_1',w_1){}^\alpha M(a_2,a_2',w_2)  \text{ or is } \\ 
\begin{pmatrix}1 &-a_1\nm a_2\nm a_1'\nm a_2' &w_1 \np w_2 \np a_1a_2'\\ 0& 1 &{a_1\np a_2} \\ 0& 0& 1 \end{pmatrix}  = \begin{pmatrix}  1 &-a_1\nm a_1' &w_1 \\ 0& 1 &{a_1} \\ 0& 0& 1 \end{pmatrix}  \begin{pmatrix} 1 &-a_2\nm a_2' &w_2 \\ 0& 1 &{a}_2 \\ 0& 0& 1\end{pmatrix}.\end{array} \end{equation} 
Result: the upper right-hand positions on the two sides are not generally equal.  

Lem.~\ref{carmichael} gives the extension in the last statement of the Lemma, with the centralizing $\bZ/3$ action stated in \eql{ell2k}{ell2kb}, given in detail in \eqref{extalpha}. \end{proof} 

We use these basic facts applied to an $\ell$-group $K$, with $|K|> \ell$.
\begin{edesc} \label{basicl} \item \label{basicla} A subgroup of  $K$ of index $\ell$ is automatically normal, and  
\item \label{basiclb} $K$ contains a subgroup of index $\ell$ \cite[p.~122]{Ca56}. 
\item \label{basiclc}  $K$ contains an element $w\not = 1$ in its center \cite[p.~68]{Ca56}. 
\end{edesc} 

Lem.~\ref{carmichael2} gives cases from Lem.~\ref{carmichael} satisfying the additional assumptions \eql{framework}{frameworkb} and \eql{framework}{frameworkc}. 
\begin{edesc} \label{framework} \item \label{frameworka}  The nontrivial center $C$ of $K$ \eql{basicl}{basiclc} has order $\ell$; 
\item \label{frameworkb} the homomorphism $K\to K/C$ is a Frattini cover\footnote{From \eql{framework}{frameworka} this is automatic.} with a split, faithful action of an $\ell'$-group $H$ on $K/C$;  and 
\item \label{frameworkc} $H$ extends to $K\xs H$ acting trivially on $C$. \end{edesc}

\newcommand{\da}{\dot{a}} \newcommand{\db}{\dot{b}} \newcommand{\dx}{\dot{x}} \newcommand{\dy}{\dot{y}}
Assume \eqref{framework} holds for $K$. Each group in  Lem.~\ref{carmichael}  is a $\bZ/\ell$ extension of $V=\lrang{a,b}=(\bZ/\ell)^2$, distinguished by the orders of generators $\da,\db$ of $K$. 
\begin{lem} \label{carmichael} With $\ell$ odd, there are three nonisomorphic nonabelian groups of order $\ell^3$. Each has generators $\da,\db$ with $\lrang{w=\da\db\da^{-1}\db^{-1}}=C$ with these respective properties: 
\begin{edesc} \label{typel} \item \label{typela} for $K_{\ell,\ell}$, $\da$ and $\db$ have order  $\ell$;   
\item \label{typelb} for $K_{\ell^2,\ell^2}$, $\da$ and $\db$ have order $\ell^2$, $\da^\ell=\db^\ell=w$; and 
\item \label{typelc} for $K_{\ell^2,\ell}$,  $\da$ (resp.~$\db$) has order $\ell^2$ (resp.~$\ell$).
\end{edesc} 

There is an $\ell^{k\np1}$ version,  $K_{\ell^2,\ell^2, k}$, of  $K_{\ell^2,\ell^2}$ whose properties we list in \eqref{ell2k}.  
\end{lem} 

\begin{proof} From \eql{framework}{frameworkb}, $K$ contains a normal subgroup, $V^*$, of order $\ell^2$. The same argument shows $V^*$ is abelian. Also, since $K$ is nonabelian, it has only one subgroup, $\lrang{w}$, of order $\ell$ in its center.  

If $V^*=\lrang{\da}$ is cyclic, its automorphism group is $(\bZ/\ell^2)^*$, invertible integers mod $\ell^2$,  acting by putting $\da$ to $\ell'$ powers. Conjugate $V^*$ by $\db\in K\setminus V^*$ (of $\ell$-power order).  Replace $\db$ by an appropriate $\ell'$ power to have it act as \begin{equation} \label{bcong} \begin{array}{c} \da\mapsto  \db^{-1}\da\db=\da^{1\np\ell}\text{ giving }K_{**}=\lrang{\da,\db\mid \da^{\ell}=w}; \text{ and }\\ \text{ from }\db^{-1}\da\db\da^{-1}=w, \da\db\da^{-1}=\db w.\end{array}\end{equation} There are two cases with $V_{\db}\eqdef \lrang{\db}$.
\begin{edesc} \label{cyclicV}
\item \label{cyclicVa} $K_{**}=K_{\ell^2,\ell^2}$: For $\ord(\db)=\ell^2$, $V_{\db}\norm K_{**}$ and $V\cap V_{\db}=\lrang{w}$.   
\item  \label{cyclicVb} $K_{**}=K_{\ell^2,\ell}$: $\ord(\db)=\ell$, and $V_{\db}$ is not normal.  \end{edesc} 
We do the case $K_{**}=K_{\ell^2,\ell^2}$, leaving  $K_{\ell^2,\ell}$ to the reader. 
Each element in the group has the form $\da^m\db^n w^u$ using that  $w$ centralizes; then reduce  $m,n,u$ mod $\ell$. 
\begin{edesc} \label{immult} \item \label{immulta}  For example, we can always write $\da^m\db^n$ as $\da^{m'}\db^{n'}w^u$ with $m\equiv m', n\equiv n' \mod \ell$. 
\item \label{immultc} Replace $\db^n\da^m$ by $\da^m\db^nw^{-m\cdot n}$ using $\db\da=\da\db w^{-1}$, applying \eql{immult}{immulta} when necessary.  \end{edesc} 
This multiplication is associative since it doesn't depend on where you might put (\ )s, but only on the cardinality of $\da\,$s to the right of $\db\,$s. 

We already have $K_{\ell,\ell}$ as the small Heisenberg group of \S\ref{smheis}. Here is a list of the $K_{\ell^2,\ell^2}$  generalization, to level $k$: 
\begin{edesc} \label{ell2k} \item It fits in the short exact sequence $$\lrang{w_k}\eqdef \bZ/\ell^{k\np1} \to K_{\ell^2,\ell^2, k}\eqdef\lrang{\dot a_k,\dot b_k}\longmapright{\psi_k}{20} (\bZ/\ell^{k\np1})^2\text{, as an $\ell$-Frattini cover of $(\bZ/\ell^{k\np1})^2$};$$ 
\item \label{ell2kb} with a $\bZ/3$ action that centralizes $\ker(\psi_k)$, extending the $\bZ/3$ action for $k\nm1$, etc. 
\item In the  exponent condition in \eql{immult}{immulta}  use that the commutator of $ \dot a_k$ and $\dot b_k$ is $w_k$. \end{edesc} 
Here is the  $\bZ/3$ action of \eql{ell2k}{ell2kb}; we use $\dot a, \dot b, u$, but it works for these generators with the $k$ subscripts as well.
As in \eqref{A*act}, $\lrang{\alpha}=\bZ/3$; $\alpha$ acts on $V_{\ell,k}=\lrang{a,b}$: $a\mapsto b$ and $b\mapsto -a\nm b$. With $\da$ and $\db$  respective generators of $K_{\ell^2,\ell^2,k}$ lying over $a$ and $b$, use multiplicative notation. 
\begin{edesc} \label{extalpha} \item Extend $\alpha$ (resp.~$\alpha^{-1}$) by $(\da,\db)\mapsto  (\db,\db^{-1}\da^{-1}=(\da\db)^{-1})$ (resp.  $((\da\db)^{-1},\da)$.
\item Then $(\da,\db)\mapsto (\db,\db^{-1}\da^{-1})$ has order  3: 
$$\begin{array}{c}  (\da,\db)\longmapright{(\alpha)^2} {25} (\db^{-1}\da^{-1}, \da)\longmapright {\alpha}{20}  (\da,\db); w=aba^{-1}b^{-1}\longmapright {\alpha} {20} \ b(b^{-1}a^{-1})b^{-1}(ab) \\ =(ba)^{-1}(ab); 
\text{ conjugate by $ab$ and we are back to $w$.}\end{array}$$
\end{edesc}
What we showed was that the $\ell$-Sylows of $G_{\ell,k,2_2}$ and $G_{\ell,k,3}$ give distinct representation covers of $(\bZ/\ell^{k\np1})^2$, they they have central extension kernels with the same Schur multiplier; showing $G_{\ell,k,3}$ is the universal central extension of $\bH_{\ell,k,3}$. 
 \end{proof}

\subsubsection{The \HM-\DI\ principle}  \label{HM-DIPrinc}   The following \HM-\DI\ principle will simplify computations. Instead of the whole Hurwitz space, consider the union of reduced components containing cusps defined by the following Nielsen class representatives:  
\begin{edesc} \label{hm-diHyp} \item \label{hm-diHypa}  An \HM\ rep. of form $(g_1,g_1^{-1}, g_3, g_3^{-1})$, $\lrang{g_1,g_3}=G$; or 
\item \label{hm-diHypb}  A {\sl double identity}, \DI,  element of form $(g_1,g_2,g_1,g_4)$ satisfying product-one with $$\lrang{g_1,g_2,g_4}=G, \text{ and }g_2,g_4\in \C_{-}.$$ 
\item \label{hm-diHypc}  Apply \eql{midprod}{midprodb} to  $\bg=(g_1,g_2,g_3,g_4)\in \bfC_{\pm 3^2}$ to conclude the cusp width of $\bg$ is 1 if $g_2=g_3^{\pm 1}$ and exceeds 1, otherwise. \end{edesc} We speak of \HM\ and \DI\ orbits or components. Rem.~\ref{hybrid} explains there are hybrids of these for level $k>0$. 
\begin{lem} \label{HM-DIcomps}  Cusps associated to \sh\ applied to \eql{hm-diHyp}{hm-diHypa} and $q_1$ applied to \eql{hm-diHyp}{hm-diHypb} have width 1, and a Hurwitz space component of $\sH(G_{\ell,0,3},\bfC_{\pm 3^2})$ has a cusp of width 1 if and only if its braid orbit has one of these cusps. Denote the union of such  components by $\sH_{\HM-\DI}$. The total space so defined has moduli definition field $\bQ$. \end{lem}

\begin{proof} The only piece requiring proof is the last line, and this follows because the Hurwitz space itself has moduli definition field $\bQ$ and $G_\bQ$ acting on the components preserves the collection of the cusp widths of each component. (as the technique used on Ex.~\ref{HM-separated} shows here). \end{proof} 

Denote $\bg=(g_1,g_2,g_3,g_4)$ in the Nielsen class as in $\ni_{\pm \pm}$ if its elements are, in order, in the classes $\C_+,\C_-,\C_+,\C_-$. The steps for analyzing components of $\sH_{\HM-\DI}$  for applying Thm.~\ref{cosetbr}: 
\begin{edesc} \label{divliftinv0} \item \label{divliftinv0a} Lem.~\ref{l03liftinv} computes lift invariants of \DI\ elements in $\ni_{\pm \pm}$, finding we achieve, at level 0, all possible lift invariant elements. Again, \HM\ elements have trivial lift invariant.  
\item \label{divliftinv0b} As in Ex.~\ref{0LiftInv}, some \DI\ elements have lift invariant 0. We need to know if its component is distinct from  \HM\  components. \end{edesc}

\begin{exmpl} \label{HM-separated}  The proof of \cite[Thm.~3.21]{Fr95} uses projective normalization of the Hurwitz space in its function field, indicating how the absolute Galois group detects properties of Hurwitz spaces on their boundaries. The main application distinguishes the union of \HM\ components of a Hurwitz space by a total degeneration of curves in the family on the boundary. Then, it gives a criterion -- \HM-gcomplete -- for a braid orbit to contain all \HM\ reps in a Nielsen class, and that this implies the corresponding component has moduli definition field $\bQ$.  
This used a special device, \cite[(3.21)]{Fr95}, a (normalization) specialization sequence, designed explicitly for Hurwitz space compactification. Nevertheless, \cite{DEm06} carried out a Deligne-Mumford-type compactification that included the same result. 

Second: This has been used to show many \MT s that have moduli definition field $\bQ$ at all their levels. Thus, the Main \MT\  Conjecture can't be proven by showing that high \MT\ levels have high degree moduli definition field over $\bQ$. This criterion does not apply, though, to $r=4$. The Main Conjecture for $r> 4$ may require generalizing  Faltings'  Theorem to higher dimensions. 
\end{exmpl} 

\subsubsection{Lift invariants of the \DI\ components in $\sH_{\HM-\DI}$} \label{liftinvell03}  

Lem.~\ref{carmichael} gives  the Schur (representation) cover of $G_{\ell,k,3}$, after adding the $\bZ/3$ action:  $K_{\ell^2,\ell^2,k}\xs \bZ/3$. We have two purposes.  
\begin{edesc} \label{levk} \item \label{levka} To compute the lift invariant, as a separator of components, especially at level 0.
\item \label{levkb} Data for explicitly forming a \MT\ over the \HM\ and \DI\ components. \end{edesc} This section gives what we need.  \cite{Fr26} finishes \eql{levk}{levkb} for the \DI\ components using Lem.~\ref{l03liftinv}. 

Use the notation of ${}_{x^my^n}\alpha=(x^my^n)\alpha y^{-n}x^{-m}$ for a general element of $\C_3=\C_+$ since $w$ is in the center. Similarly, for $\C_{-3}=\C_-$ replace $\alpha$ by $\alpha^{-1}$. Each Nielsen class element braids to  one in 
$$\ni_{\pm\pm}=\{ (\alpha, {}_{x^m_2y^n_2}\alpha^{-1}, {}_{x^m_3y^n_3}\alpha, {}_{x^m_4y^n_4}\alpha^{-1}) \text{ satisfying product-one and generation.}\}$$ 
Lem.~\ref{carmichael2} gives the steps for computing $\ni_{\pm\pm}$ lift invariants. For compatibility  use $\dx$ and $\dy$ in place of $\da$ and $\db$ from Lem.~\ref{carmichael}. Then, \eqref{conjclasses} gives presentations in $\hat G_{\ell,0,3} $ of the order 3 lift -- with $\alpha^{-1}$ on either the right or left -- of an element in $\C_{\nm} $. For products of powers of $\dx$ and $\dy$, take {\sl standard form\/} to be $\dx^m\dy^nw^u$. 

\begin{lem} \label{carmichael2} Useful formulas for writing a conjugate in standard form (all exponents $\!\!\mod \ell$):
\begin{equation} \label{2forms}   \begin{array}{c} \text{ \rm a. } (\dy\dx)^n=\dx^n\dy^nw^{\frac{n(n\np1)}2}, \text{ \rm b. }\dy^m\dx^n=\dx^n\dy^mw^{m\cdot n} \\ \text{ \rm c. }{}_{\dx^m\dy^n}\alpha^{\pm 1}=  {}_{\dy^n\dx^m}\alpha^{\pm 1}, \text{ \rm d. } (\dx\dy)^n =  \dx^n\dy^n w^{\frac{(n\nm 1)n}2}.\end{array}\end{equation} 

The order 3 lifts of elements in $\C_{\nm}$  have either of these two forms running over $m,n$: 
\begin{equation} \label{conjclasses} \begin{array}{c}  {}_{\dx^m\dy^n}\alpha^{-1}=\dx^m\dy^n(\alpha^{-1}\dy^{-n}\dx^{-m}\alpha)\alpha^{-1} \text{ or } \alpha^{-1}(\alpha \dx^m\dy^n \alpha^{-1})\dy^{-n}\dx^{-m}   \\ \text{which are respectively } \begin{cases} \dx^m\dy^n(\dx\dy)^n\dy^{-m}\alpha^{\nm1}=\dx^{m\np n}\dy^{2n\nm m}\alpha^{-1}w^{\frac{3n^2\nm n}2}, 
\\ \alpha^{\nm1} (\dx\dy)^{-m}\dx^n \dy^{-n}\dx^{-m} =\alpha^{\nm 1} \dx^{n\nm 2m}\dy^{\nm m\nm n}w^{\frac{3m^2\np m\nm 4n^2}2}.
\end{cases} \end{array} \end{equation}
\end{lem} 

\begin{proof}  Since results only depend on exponents $\!\!\mod \ell$, we can assume all exponents are $\ge 0$. For  \eqref{2forms} a., to get to standard form in $(\dy\dx)^n$, running over $1\le i\le n$, move the $i$th $\dy$ past all $\dx\,$s ($n\nm i\np1$ of them) to its right. Use \eqref{bcong} to replace each $\dy\dx$ by $\dx\dy w$. The cumulative $w\,$s are  $w^{\sum,_{i=1}^n  n\nm i\np1}=w^{\frac{n\cdot (n\np1)}2}$ to the right of $\dx^n\dy^n$,  \eqref{2forms} b. is even easier. For \eqref{2forms} c., consider 
$${}_{\dx^m\dy^n w^u}\C_{\pm}= \dx^m\dy^n w^u\C_{\pm}w^{-u} \dy^{-n}\dx^{-m} .$$ The result follows since  $w$ is in the center and $w^uw^{-u}=1$.  Finally, for  \eqref{2forms} d.  $$(\dx\dy)^n=\dx(\dy\dx)^{n\nm 1}\dy = \dx^{n}\dy^n w^{\frac{(n\nm 1)n}2}\text{ from \eqref{2forms} a}.$$

Details of \eqref{conjclasses}: The 1st line arranges for $\alpha^{-1}$ to be on, respectively, the right and left using the \eqref{extalpha} action. Apply $\alpha$ in the 1st case and aim for standard form with $\alpha^{-1}$ and a power of $w$ on the right. To finish that calculation write  $\dy^n(\dx\dy)^n$ as $\dx^n\dy^{2n}w^{u(n))}$. First move each $\dy$ in $\dy^n$ past $n$ copies of $\dx$.  For each such move, add one $w$ to the right side. That leaves $(\dx\dy)^n(\dy^n)w^{n^2}$. The exponent for $w$ is $n^2\np \frac{(n\nm 1)n}2= \frac{3n^2\nm n}2$ from   \eqref{2forms} d. 

For the 2nd cases line, put $(\dx\dy)^{-m}\dx^n \dy^{-n}\dx^{-m}$ in standard form.  First apply \eqref{conjclasses} d. and b.: 
$$\begin{array}{c} \mapsto \dx^{\nm m}\dy^{\nm m} \dy^{\nm n} \dx^{ n} \dx^{\nm m} w^{u}=\dx^{\nm m}\dy^{\nm m \nm n} \dx^{n\nm m}w^u \text{ with } u=\frac{m^2\np m \nm 2n^2}2,\\ \text{then apply \eqref{conjclasses} b. }\mapsto \dx^{n\nm 2m}\dy^{\nm m\nm n} w^{u\np (m\nm n)(m\np n)}.  \end{array} $$
which we calculate to conclude the expression for the second case.
\end{proof}

There is little difference between the proof of Lem.~\ref{l03liftinv} for $k=0$ and for general $k$, except for taking exponents $\mod \ell^{k\np1}$. To simplify notation, we take $k=0$. \begin{lem} \label{l03liftinv} The lift invariant of a \DI\ element in $\ni_{\pm \pm}$ is the product of the entries of some \begin{equation} \dot{\bg}_{m_2,n_2,m_3,n_3}\eqdef (\alpha^{-1}, {}_{\dx^{m_2}\dy^{n_2}}\alpha^{-1}, {}_{\dx^{m_3}\dy^{n_3}}\alpha^{-1})\in \bfC_{-3}.\end{equation}  The following hold: 
\begin{edesc} \label{results} \item \label{resultsa} Generation for the image element, $\bg_{m_2,n_2,m_3,n_3}\in \ni(G_{\ell,0,3},\bfC_{\nm3})$, fails if and only if $\lrang{m_2x^{},n_2y^{}}$ is an eigenspace for $\alpha$ (in particular, then $\ell\equiv 1\mod 3$, Rem.~\ref{quadrec}).
\item \label{resultsb} Assuming generation in \eql{results}{resultsa} the lift invariant of  $\dot{\bg}_{m_2,n_2,m_3,n_3}$ is $w^{m_2^2-n_2^2-m_2n_2}$. 
\item \label{resultsc} For $k=0$ and $\ell> 5$, there are distinct \DI\ orbits running over  $u\in (\bZ/\ell)$. For $\ell=5$, the lift invariants run over the squares in $(\bZ/\ell)^*$. 
\end{edesc} 
\end{lem} 

\begin{proof} Apply $q_2^{-1}$ to braid $(g_1,g_2,g_1,g_4)$ to $(g_1,g_1,g_1^{-1}g_2g_1, g_4)$. Now check, with $g_1^{-1}g_2g_1=g_2'$,   that this has the same lift invariant as $(g_1^{-1}=g_1^2,g_2',g_4)\in \ni_{0,-3^3}$ which we take to be $\dot{\bg}_{m_2,n_2,m_3,n_3}$, subject to the product-one condition using \eqref{conjclasses}:  
\begin{edesc} \label{product-oneres} \item \label{product-oneresa} $n_2 \nm 2m_2 \np m_3\np n_3=0$ and $\nm m_2\nm n_2 \np 2n_3 \nm m_3=0$; 
\item   \label{product-oneresb} add the terms of \eql{product-oneres}{product-oneresa} to get  $m_2 = n_3=m_3\np n_2$, or $m_3=m_2\nm n_2$.  \end{edesc} 
That shows \eql{results}{resultsa}.  
Braid $\dot{\bg}_{m_2,n_2,m_3,n_3}$ to $({}_{\dx^{m_2}\dy^{n_2}}\alpha^{\nm1},\dx^{m_3\np n_3}\dy^{2n_3\nm m_3}w^{\frac{3n_3^2\nm n_3}2}\alpha^{-1},\alpha^{-1})$.   
Apply the (left) shift and the second case of \eqref{conjclasses} to get the lift value by eliminating the product  $\alpha^{-1}\alpha^{-1}\alpha^{-1}=1$. Use product-one \eql{product-oneres}{product-oneresa} and  \eqref{conjclasses} b. (in the middle terms) of  
$$\dx^{m_3\np n_3}(\dy^{2n_3\nm m_3}w^{\frac{3n_3^2\nm n_3}2} \dx^{n_2\nm 2m_2})\dy^{\nm m_2\nm n_2}w^{\frac{3m^2_2\np m_2\nm 4n_2^2}2}.$$
Using \eql{product-oneres}{product-oneresb}, the lift invariant is $w^{\frac{3m_2^2\nm m_2}2} w^{\frac{3m_2^2\np m_2\nm 4n^2_2}2} w^{(m_2\np n_2)(n_2\nm 2m_2)}=w^{m_2^2-n_2^2-m_2n_2}$, thus concluding  \eql{results}{resultsb}.

We achieve all lift invariant values $\mod \ell$ follows if the 2-form $m_2^2-n_2^2-m_2n_2$ maps onto $\bZ/\ell$. A solution $(m_2',n_2')\in (\bZ/\ell)^2$ then gives solutions $(um_2',un_2')$ for any $u\in \bZ/\ell$. So, achieving all lift values is equivalent to $x^2-x-1=(x\nm 1/2)^2 \nm 5/4$  -- or $x^2-5$ -- has both square and nonsquare values for $x\in \bZ/ \ell$. It has only square values $\!\!\mod 5$. 

For $\ell\ne 5$,  the nonsingular projective curve $C_a$ in $\prP^2$ defined by $x^2-5y^2-az^2=0$  has rational points over $\bF_\ell$ from the triviality of Brauer-Severi varieties over finite fields. The value $a=1$ (resp.~a primitive root $\!\!\mod \ell$)  is a square (resp.~nonsquare), concluding the proof of \eql{results}{resultsc}. 
\end{proof} 

\begin{exmpl} \label{0LiftInv}[\DI\ orbits with 0 lift invariant]   Since \HM\ orbits have lift invariant 0, we have the question if these \DI\  braid orbits are homeomorphism-separated from all \HM\ braid orbits. They are not: Lem.~\ref{bothhmanddi-lem} and Ex.~\ref{bothhmanddi}. 

Apply \eql{results}{resultsc} -- quadratic reciprocity (\eqref{quadrec} with 3 replaced by 5.  The relevant formula is $$\left(\frac 5 \ell\right) \left(\frac \ell 5\right) = (-1)^{(\frac {(5 \nm 1)}2 \frac{(\ell\nm1)} 2}=1 \implies \text{ for }\ell \equiv 1, 4 \mod 5\ \exists\ x,  x^2 -5\equiv 0  \!\!\mod \ell.$$
Subexample: $\ell=11$, $x=4$. Translate this to a \DI\ element: $m_2=3, n_2=-1 \mod 11$,  and from \eqref{product-oneres},  $m_3=4, n_3=3$. Also, since $11 \equiv 2\mod 3$, generation holds (Ex.~\ref{quadrec}). 
\end{exmpl} 

\begin{rem} \label{hybrid} For $\hat \bg\in \ni_k\eqdef \ni(G_{\ell,k,3},\bfC_{\pm 3^2})^\inn$, \eqref{ell2k} gives the source of its lift invariant, $s_{\hat \bg}\eqdef s_{\hat \bg,k}$,  at level $k$. Reduction of $\hat \bg \mod \ell^{k'}, 0\le k' \le k$ gives rise to a series of elements $\hat \bg(k')\in \ni(G_{\ell,k',3},\bfC_{\pm 3^2})^\inn$. If two elements of $\ni_k$ have different series of lift invariants, then they certainly belong in separate components. We don't at this time know the relationship between \DI\ components and having this series consist of $\ell'$ lift invariants. 
\end{rem}

\subsection{Serre's goal and Coleman-Oort}  \label{Serre-book}  \S\ref{sephmanddi} applies the \sh-incidence matrix to analyze the steps  in \eqref{divliftinv0}, for applying Thm.~\ref{cosetbr} restricted to the $\sH_{\HM-\DI}$ components. \S\ref{SMR} gives the context for Serre's $\ell$-adic representations, outlining his Main Theorem, which he applied to the \ST\ representations. 
\S\ref{HIT-ST}  summarizes expectations for types of fibers on a given \MT. 

We conclude with a statement on the broader context of our approach, driven by properties of finite groups that fit into series, and its relations to many unsolved problems in Galois theory (such as the Regular Inverse Galois Problem). The opening paper, \cite{Fr95}, on \MT s stated this. The goals of \cite{Fr26} bridge the topics, and the gap of many years of two Serre books (\cite{Se68}, \cite{Se92}, see \cite{Fr94}, not to mention that gap-bridger, Galois cohomology, a topic between Serre and me over many years). Here's what makes my approach look so different. As the parameter (usually $\ell$) changes, my moduli space (of curve covers) seems to change in a style distinct from  that given by, for example, the moduli of abelian varieties of dimension $\geng$. True in a way, but expanding the applicable problems requires seeing that isn't always an essential difference. For example, even for elliptic curves, there are different spaces, $X(\ell^{k\np1})$, as $\ell$ varies, and also --- should you so desire Ð- my series of examples often fit within one rubric, with one finite group, $H$, acting on a lattice for which you vary the $\ell$-adic completion. The results, however, for two different $H\,$s can be extraordinarily different, as examples \S \ref{weilpairing} and \S \ref{heiscase} show, even if the lattices seem the same.  The very short \S\ref{heisgens} uses Serre's case to outline the next step in \cite{Fr26} -- constructing in these examples Heisenberg generators to make explicit checking for \HIT\ points on their \MT s.

\subsubsection{Applying  the \sh-incidence matrix} \label{sephmanddi}  Start with \eql{divliftinv0}{divliftinv0b}: Are the \DI\ components of $\sH_{\HM-\DI}$ of lift invariant 0 in Ex.~\ref{0LiftInv} homeomorphism-separated  from the \HM\ components. 

With $\bv=-(m_2,n_2)$, here is an example \DI\ element: $$\bg_{\DI}=(\alpha_0^{-1},{}_\bv\alpha_0, {}_\bv\alpha_0,{}_\bw\alpha^{-1}_0) \text{ (with   } \sh(\bg_{\DI})=(\alpha_0,\alpha_0, {}_{(m_2,n_2)}\alpha_0^{-1}, {}_{(m_3,n_3)}\alpha_0^{-1})).$$ Its cusp has just one element since its middle product commutes with its 2nd and 3rd terms. Determine $\bw$ from the product-one condition, $2\bw\np \bw^\alpha=2\bv^\alpha\np \bv$. 

We want to see if $\sh(\bg_{\DI})$ is in the braid orbit of an \HM\ rep. 

\begin{lem}  \label{CuspsDIandHM} \label{bothhmanddi-lem} Start with, when does the cusp of $\bg=({}_{\bv_1}\alpha_0,{}_{\bv_2}\alpha_0^{-1},{}_{\bv_3}\alpha_0,{}_{\bv_4}\alpha_0^{-1})$ contain an \HM? Conjugate $\bg$ by $\smatrix 1 0 {\nm \bv_1} 1$ to assume $\bv_1=\pmb 0$.  Denote  $\bv_2\nm \bv_2^{\alpha^{-1}} \np \bv_3^{\alpha^{-1}}\nm \bv_3$ by $\bw_{2,3}$. Multiply $\smatrix \alpha 0 {\bv_2^\alpha\nm \bv_2} 1$ and $\smatrix {\alpha^{-1}} 0 {\bv_2^{\alpha^{-1}}} 1$ to see the middle product of $\bg$ is $\smatrix 1 0 {\bw_{2,3}} 1$. Then,   the cusp contains a \HM\ if either there is $k_2\in \bZ/\ell$ for which,  $\bv_2+k_2\bw_{2,3} = \pmb 0$ or $\bv_3$, or there is a $k_3\in \bZ/\ell$ for which  $\bv_3+k_3\bw_{2,3} = \bv_2$ or $\bv_4$. 

The related question for the cusp containing a \DI: if there is a $k_2'\in \bZ/\ell$ (resp.~$k_3'$) for which $\bv_2+k_2'\bw_{2,3}=\bv_4$ (resp.~$\bv_3+k_3'\bw_{2,3}=\pmb 0$). 
\end{lem}

\begin{proof} From Prop.~\ref{j-Line}, for this Nielsen class, the full cusp containing $\bg$ consists of the conjugations of $\bg$ by powers of $\smatrix 1 0 {\bw_{2,3}} 1$. The listings in the statement of the Lem.~are just the conditions that we get \HM\ or \DI\ elements in this one $q_2$ (cusp) orbit. 
\end{proof} 

To include all levels of \MT s, where lift invariants of braid orbits fall in $\bZ/\ell^{k\np1}$, requires considering the jumps of lift invariant values in going from $(\bZ/\ell^{k\np1})^*$ to lift invariants in $\bZ/\ell^k$. We expect 
 \sh-incidence matrices used in \cite[\S2.3.3]{Fr20} to simplify this, but Ex.~\ref{bothhmanddi} gives a major issue. 

\begin{exmpl}[Cusps containing \HM\ and \DI\ reps] \label{bothhmanddi} We know that there are several \HM\ orbits in these  Nielsen classes, but do the \DI\ orbits with lift invariant 0 belong in braid orbits separate from \HM\ orbits?  The simplest possibility they are not, would be if $\sh(\bg_{\DI})$ (notation of Lem.~\ref{CuspsDIandHM}) is in the cusp of an \HM\ rep. There are several cases. For example, the condition, there exists $k_2\in \bZ\ell$ for which $\bv_2+k_2\bw_{2,3} = \pmb 0$ is for the cusp of $\bg$ to contain an \HM\ rep. Similarly, $\bv_3+k_3'\bw_{2,3}=\pmb 0$ for some $k_3$ is that it contains a \DI\ rep.  That is, is $|\sh(\bg_{\DI})\cap \cO_{\bg}|=1$? 

The generation condition for Nielsen classes demands that $\bv_2$ $\alpha$-generates (Def.~\ref{alpha-gen}). Since $\bw_{2,3}=(\bv_2\nm \bv_3)^{(1\nm\alpha^{-1})}$, by subtracting the equations see that $\bv_2\nm \bv_3$ is an $\alpha$ eigenvector. By adding them, also  $\lrang{\bv_2,\bv_3}=\lrang{\bv_2\nm \bv_3}$. That is, $\bv_2$ does not $\alpha$-generate, so this is impossible. \end{exmpl}

\begin{rem} \label{notlperfect} The group theory differs between \S\ref{weilpairing} and this section because $(\bZ/\ell)^2$ is not $\ell$-perfect. That allows it to have two non-isomorphic representation covers, one given by the small Heisenberg group, the other not, partly explaining why these examples are so very different. 

Once we get the $\bZ/3$ action involved, then $G_{\ell,0,3}$ {\sl is\/} $\ell$-perfect ($\ell\ne 3$), and it has a unique representation cover. The same for adding the $\bZ/2$ action to get $G_{\ell,0,2}$ ($\ell\ne 2$).   The lift invariant computations of Lem.~\ref{liftinvell03} and Lem.~\ref{l03liftinv} allow graphically presenting the components of all the \MT s coming from this section from knowing all components with lift invariants from the values of the 2-form $m_2^2 -n_2^2 -m_2n_2$, with one complication, those \DI\ components with lift invariant 0. 
 \end{rem}

\subsubsection{Compatible $\ell$-adic representations of $G^\ab_K$} \label{SMR} 
Serre starts with the short (adele/ idele) exact sequence from {\sl Class field Theory: CFT}:  $K$ a number field, $\bm=(m_{\nu_1}. \dots, m_{\nu_s})$, an $s$-tuple of integers attached to finite valuations of $K$ indicating multiplicities. 

Start with $G^\dagger_K$ a quotient of the absolute Galois group of $K$. Serre's interest is in the maximal abelian quotient, $G^\ab_K$. Then a system of representations referencing $\ell$, at the minimum, means $\rho_\ell: G^\dagger_K \to \Aut(V_\ell)$, running over almost all $\ell$, with $V$ given by a $\bZ[G^\dagger_K]$ module tensored with $\bQ_\ell$. {\sl Compatibility\/} means the  characteristic polynomials $\det(1 - \Fr_{\rho_\ell, \bp}T)$  associated with Frobenius elements (conjugacy classes attached to a prime $\bp$ of $K$) attached to different (almost all)  $\ell$ are the same (and defined over $\bQ$). Spanning from Weil (with abelian varieties) and Grothendieck with \'etale cohomology of non-singular projective varieties,  Serre goes outside of these algebraic-geometric places affording {\sl all\/} $\ell$-adic representations, but only for the abelian case, $G_K^\dagger=G^\ab_K$.\footnote{One reason for that, is that is the only case where we know how to get a handle on $G_K^\dagger$. But, the point of the \OIT\ theorem, and the conjectures like Coleman-Oort, is this case stands out even when considering what is the image of $G_K$ in acting on a Tate module of an abelian variety.}

Applied to $\bQ$ algebras $A$, here are the steps starting with the class field theory (CFT) sequence for computing the profinite group of abelian extensions, $G^\ab_K$, of $K$:  $C_\bm$ is the group generated by ideals modulo principle ideals $(u)$, $u$ in the ring of integers for which $u-1$ in the  $m_\nu$ power of the ideal for $\nu$, for all indexes $\nu$.

\newcommand{\mult}{\rm mult}

\begin{equation} \label{cft}  \begin{array}{c} \text{II-7}: 1 \to K^*/E_\bm \to I_\bm \to C_\bm \to 1,  \text{ with $I$ the ideles,} \\
\text{ $E_\bm$ given by \eql{restscalars}{restscalarsc} and $G_K^\ab$ the  projective limit of $C_\bm$ over $\bm$.}\end{array} \end{equation}  He forms a $K$ torser (multiplicative, algebraic group), $S_\bm$, over $K$ whose $\bQ_\ell$ values produce compatible $\ell$-adic representations of $G_K^\ab$. Here's the sequence, with $d=[K:\bQ]$: \begin{edesc} \label{restscalars} \item $\sG_{\mult}(A)\eqdef \{(x\in A\mid \exists y \in A, \text{ with }xy=1\}$  assigns invertible elements $A^*$. \item  Apply Weil's  {\sl restriction of scalars\/}  $T=R_{K/\bQ}(\sG_{\mult}/K)$,  a dimension $d$ torus over $\bQ$; its $A$ points are   $(K\otimes_\bQ A)^*$, so $T(\bQ)=K^*$. \item \label{restscalarsc} For subgroup $E_\bm \le K^*$,   indexed as above by $\bm$, take $\bar E_\bm$ its Zariski closure in $T$, and  $T_{E_\bm}=T/\bar E_\bm$ gives $K^*/E_\bm = A$.
\label{restscalarsd} This gives $K^*/E_\bm\to T_\bm=T/\bar E_\bm$ (also a torus), a {\sl  pushout}, $I_\bm \to S_\bm$ given by the 2-cocycle of the sequence \eqref{cft}. \end{edesc} 

This produces a diagram, \cite[p. II-9]{Se68}, with the upper line from \eqref{cft} and the lower line  $1 \to T_\bm(\bQ) \to  S_\bm(\bQ) \to C_\bm \to 1$. Applying the class field theory identification  with  of $G^\ab_K$, \cite[\S II.3]{Se68} then uses the homomorphism $\pi_\ell: T(\bQ_\ell) \to S_\bm(\bQ_\ell)$ to define $\epsilon: G^\ab_\ell \to S_\bm(\bQ_\ell)$, a system of compatible $\ell$-adic representations with values in $S_\bm$. Using that $S_\bm$ is a torus, \cite[pgs.~II-10--II-23]{Se68} shows this gives $\phi_\ell: G^\ab_K\to \Aut(V_\ell)$, an abelian $\ell$-adic, semi-simple (completely reducible), representation of $G^\ab_K$ on $V_\ell$ fulfilling the title of the book.

Those, however, don't give all such representations. By limiting to {\sl abelian\/} $\ell$-adic representations and this characterizing rubric -- using the definition of {\sl locally algebraic\/} -- these tori $S_\bm$ go beyond the paradigm that started with abelian varieties, and \'etale cohomology of nonsingular projective varieties. There was the surprise of \cite{De72a}: K3 surfaces have \'etale cohomology in the category of abelian varieties. Then the obvious question answered by \cite{De72b} shows that even complete intersections could have \'etale cohomology outside that generated by abelian varieties. 

Despite the properties shown in \cite{De74},  how difficult is it to divine structure on the $\ell$-adic representations of the absolute Galois group of $\bQ$. \S\ref{HIT-ST} questions what we know of separating even Serre's case from abelian varieties. For example,  \cite[p.~III-11]{Se68} notes that if you don't consider abelian representations, it isn't true that any compatible set of $\ell$-adic representations of $K$ is unramified (trivial on the ramification subgroups of $\bp$) outside a finite set of places.

\subsubsection{Locating the \HIT\ and \ST\ fibers on a \MT} \label{HIT-ST}  \cite[II-\S 2.8]{Se68}  repeats the  Shimura-Taniyama  (ST) definition of a \CM\  abelian variety $A$ of dimension $d$ defined over $K$ with its "CM field" $K_A$ of degree $2d$ embedded  $i: K_A\to \End_K(A)\otimes \bQ= \End_K(A)_0$. The difference, as seen from \S\ref{cmabelvar}: \ST\ gives an actual abelian variety; but Serre shows the action on an \ST\ abelian variety, giving, for $K$ a totally complex extension of $\bQ$  an image of (not necessarily the whole) $G_K^\ab$ from its action on the corresponding $\bQ_\ell$ Tate module of the \ST\ abelian variety. Therefore, this is an example abelian $\ell$-adic representation coming from his $S_\bm$ construction.   

This gives $V_\ell$ the Tate module $\otimes \bQ_\ell$, a free $K_{A,\ell}$ rank 1 module, giving $\rho_\ell: G(\bar K/K) \to \Aut(V_\ell)$ commuting with $K_{A,\ell}$, identifying $\rho_\ell$ with a homorphism $G(\bar K/K)\to K^*_{A,\ell}= T_{K_A}(\bQ_\ell)$. Then, \cite[p. II-27 to II-29]{Se68}, Theorems 1 and 2, give the $\ell$-adic properties of $G_K^\ab$ with values in $T_{K_A}$ corresponding to a modulus $\bm$ and a morphism $\phi: S_\bm \to T_{K_A}$. The restriction of $\phi \to T_\bm$ can be read off from a homomorphism $\mu: K \to \End_{K_A}(\bT)$ with $\bT$ the tangent space of $A$ at the origin. Serre interprets this using his modules $S_\bm$ \cite[11-28]{Se68}, essentially from \cite[p.~148]{ShT61}.  

The Coleman-Oort Conjecture is about when a Jacobian of a $\geng$ curve is \ST, and says that we expect on compact sets in the moduli of genus $g$ curves, assuming the genus is large, that there are only finitely many \ST\ fibers. In our situation, we have a \MT\ over an absolute component $\sH'$ with moduli definition field $K$ based on a lattice $\sL$ appearing in each fiber. 

The lattice is a {\sl quotient\/} of the Tate module of the Jacobian of the curve attached to $\bp\in \sH'$. We are asking when the $G_K$ action gives the decomposition group either \HIT\ or abelian. 

\begin{defn} \label{MST} Call a abelian Decomposition case a {\sl Motivic\/} \ST\ fiber if it comes from a quotient of the Tate Module of an \ST\ abelian variety.  \end{defn} 

\begin{edesc} \label{coconj} \item \label{coconja} Is the decomposition group abelian only as in Def.~\ref{MST}  (see Rem.~\ref{commMST}). 
\item \label{coconjb} Excluding \eql{coconj}{coconja}, are the  fibers \HIT\ off of the fibers described in \eql{coconj}{coconja}. 
\item \label{coconjc} What other, than \HIT\ and the fibers of \eql{coconj}{coconja}  could there be? \end{edesc}

 If it is strictly analogous to Serre's \OIT, then the response to \eql{coconj}{coconjc} would be that the occurrence of other fiber types would be rare, only finitely many times on compact subsets or none at all. The most astonishing aspect of Serre's \OIT\ is that, in his case, there were only these two fiber types. He wrote several papers testing the expectations for  $\GL_2$ fibers.

\begin{exmpl} \label{COvsAO} In Thm.~\ref{andres} we have two components of $\sH(A_4,\bfC_{\pm 3^2})^{\inn,\rd}$. Only $\sH_+$ supports a \MT\ (the other is obstructed by the lift invariant). In the examples that continue in \cite{Fr26b}, we use the Schreier construction for forming generators of a subgroup of a free group, to express the generators of the fundamental group of $\hat \phi_\bg: \hat X_\bg\to \prP^1_z$ in terms of the entries of $\bg$. This will give the braid action akin to \eqref{OITbraid} as in Serre's case, and gives data for whether the geometric monodromy action is eventually Frattini. We don't know how often this happens at this time.

Also,  we don't know of any paradigm between \HIT\  and abelian, but the \MT\ constructions provide explicit examples of towers of moduli spaces for which we can ask if such exist.  \end{exmpl}

\begin{rem} \label{otherL}  The list of lattices in \eqref{mtcases} are the smallest feasible lattices for the groups in question for which we would ask about \MT s. In Serre's case and \S\ref{heiscase} we use a proper quotient of the characteristic $\ell$-Frattini module (and the lattice from ${\tG \ell}_{\ab}$). Therefore, while lift invariants help separate the components, we don't get obstructed components to forming higher levels of a \MT, but rather information on the moduli definition field. A nontrivial component lift invariant always obstructs components for the lattice from ${\tG \ell}_{\ab}$.
\end{rem}

\begin{rem}[Comments on Def.~\ref{MST}] \label{commMST} The problem is to characterize this using the Nielsen class description instead of a Riemann form,  as for a {\sl full\/} \ST\ fiber to connect to Serre's characterization of abelian $\ell$-adic representations. The moduli definition field of a \MT\ is in the decomposition field of every fiber of a \MT. Therefore, say from above,  so long as the geometric monodromy of a  \MT\ is not abelian and is eventually Frattini,  if there are Motivic \ST\ fibers, then their arithmetic monodromy differs from that of an \HIT\ fiber.\end{rem}  

\subsection{Generalizing the Heisenberg generators appearing in Serre's \OIT} \label{heisgens}  To actually compute the $H_r$ action on the lattice attached to a \MT, and thereby decide if the \HIT\ property holds, we start by abstracting the situation in our form of Serre's case. For a given prime $\ell$, the situation in Serre's case, looks like this (notation of Prop.~\ref{dihbr}): 
\begin{edesc} \label{scase} \item \label{sclasea}  $g_1=\smatrix {-1} 0 01$ with the $i$th entry $\smatrix{-1} 0 {\ba_i}  1$ with $\ba_i\in(\ell^{k\np1})^2$, $i=2,3,4$, with product-one and generation conditions gives a Nielsen class element $(\row g 4)$.  
\item Get the $H_4$ action from that of  the shift and 2-twist \eql{iso}{isob}.\footnote{This applies to $H_r$ in the generalization, as well, since the shift and 2-twist are also its generators.} 
\item \label{scasec} For example, for a given $\ell$,  the $u$-th power of the 2-twist on a line in the lattice represented by a projective sequence $\{\bg_k\leftrightarrow (\ba_{2,k},\ba_{3,k}\nm \ba_{2,k}\}_{k=1}^\infty$ maps it to $$\{(\ba_{2,k}+u(\ba_{2,k}-\ba_{3,k}),\ba_{3,k}\nm \ba_{2,k}+u(\ba_{2,k}-\ba_{3,k})\}_{k=1}^\infty, \text{ easily seen to give an action.}$$ 
\end{edesc} 
\eql{scase}{scasec} expresses the lattice elements in terms of branch cycles giving the Hurwitz monodromy action on the lattice. \cite{Fr26} generalizes this process to the \S\ref{absinnAn} and \S\ref{heiscase} examples. 
\begin{appendix}  
\begin{center} \large \bf Appendices \end{center} 
\renewcommand{\labelenumi}{{{\rm (\teql \alph{enumi})}}} 
\S\ref{classgens} gives us the classical topological generators from which the \lq\lq dragging a cover\rq\rq \ process (\S\ref{dragbybps}) works. \S\ref{fibGalClos} gives the Galois closure process that is at the heart of relating the Hurwitz space pairs $\sH(G,\bfC)^\abs$ and $\sH(G,\bfC)^\inn$ on which we base Thm.~\ref{cosetbr}.  

\section{Classical generators of $\pi(U_\bz, z_0)$} \label{classgens} 
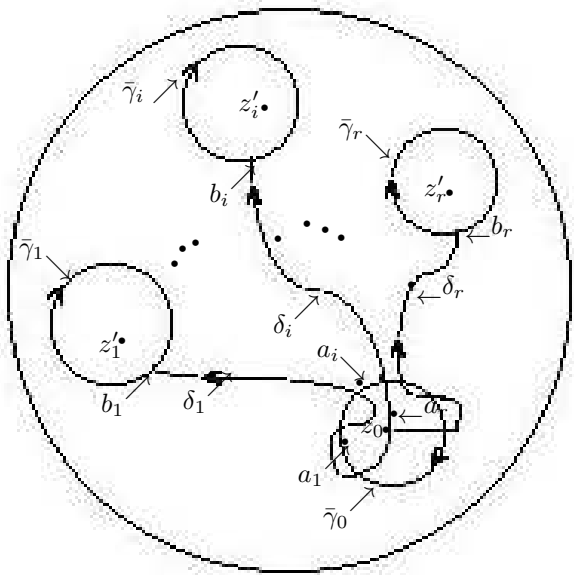
\begin{figure}[h]
\caption{Example classical generators based
at $z_0$}\label{Fig1-221b}  
\vskip3.2in
\begin{picture}(100,0)
\put(-150,-20){\vbox{ \includegraphics[width=.9\linewidth]{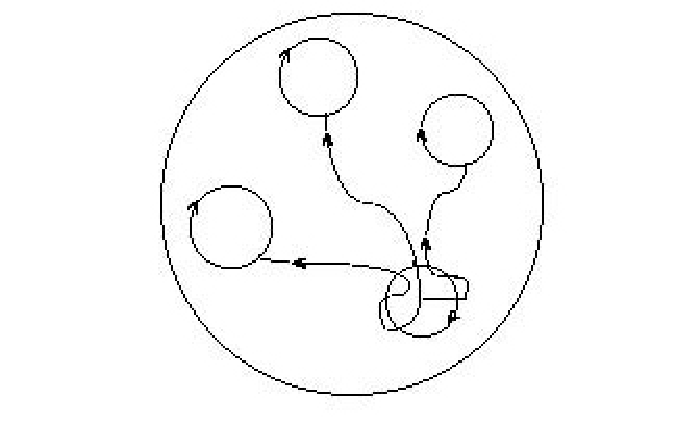}}}

\put(30, 119){$\scriptscriptstyle\bullet$}
\put(22, 111){$\scriptscriptstyle\bullet$}
\put(25, 116.5){$\scriptscriptstyle\bullet$}
\put(85, 121){$\scriptscriptstyle\bullet$}
\put(79, 124) {$\scriptscriptstyle\bullet$}
\put(72, 126){$\scriptscriptstyle\bullet$}

\put(94,48){$z_0$}
\put(-5, 79){$z_1'$}
\put(48, 170){$z_i'$}
\put(118, 138){$z_r'$}
\put(102,48){$\scriptscriptstyle\bullet$}
\put(2, 82){$\scriptscriptstyle\bullet$}
\put(56, 170){$\scriptscriptstyle\bullet$}
\put(126, 138){$\scriptscriptstyle\bullet$}

\put(4, 175){$\bar\gamma_i\nearrow$}
\setbox0=\hbox{$\searrow$} \put(86, 154){\raise \ht0\hbox
{$\bar\gamma_r$}$\searrow$ }

\setbox0=\hbox{$\nearrow$} \put(80, 20.5){\lower \ht0\hbox
{$\bar\gamma_0$}$\nearrow$ }

\setbox0=\hbox{$\searrow$} \put(-35, 108.5){\raise \ht0\hbox
{$\bar\gamma_1$}$\searrow$ }

\setbox0=\hbox{$\nearrow$} \put(27, 64){\lower \ht0\hbox
{$\delta_1$}$\nearrow$ }
\setbox0=\hbox{$\nearrow$} \put(61, 93){\lower \ht0\hbox
{$\delta_i$}$\nearrow$}
\setbox0=\hbox{$\leftarrow$} \put(115, 97){$\leftarrow$\raise \ht0\hbox
{$\delta_r$} }

\setbox0=\hbox{$\nearrow$} \put(-4, 63){\lower \ht0\hbox
{$b_1$}$\nearrow$ }
\setbox0=\hbox{$\nearrow$} \put(37, 142.5){\lower \ht0\hbox
{$b_i$}$\nearrow$ }
\setbox0=\hbox{$\leftarrow$} \put(133.5, 120){$\leftarrow$\raise \ht0\hbox
{$b_r$} }
\put(36, 69){$\scriptscriptstyle\bullet$}
\put(61, 120.5){$\scriptscriptstyle\bullet$}
\put(111.5, 103){$\scriptscriptstyle\bullet$}

\setbox0=\hbox{$\nearrow$} \put(70.5, 36.5){\lower \ht0\hbox
{$a_1$}$\nearrow$ }
\setbox0=\hbox{$\searrow$} \put(77.5, 70){\raise \ht0\hbox
{$a_i$}$\searrow$ }
\setbox0=\hbox{$\leftarrow$} \put(108, 52){$\leftarrow$\raise \ht0\hbox
{$a_r$} }
\put(86.5, 43.5){$\scriptscriptstyle\bullet$}
\put(92, 66){$\scriptscriptstyle\bullet$}
\put(105, 54){$\scriptscriptstyle\bullet$}
\end{picture}
\end{figure}

Let $z_0$ be a point
on $U_\bz$. Let $D_i$ be a disc with center $z_i$,  $i = 1,\dots ,r$.  Assume
these discs are disjoint and each excludes $z_0$.  Let
$b_i$ be a point on the boundary of $D_i$.  Regard this boundary,
oriented clockwise, as a path ${\bar \gamma}_i$ with initial and end
point $b_i$.  Finally, let $\delta_i$ be a simple {\sl simplicial} path with initial point $z_0$ and end point $b_i$. Assume, also,
that
$\delta_i$ meets none of  ${\bar \gamma}_1,\dots ,{\bar
\gamma}_{i-1},{\bar \gamma}_{i+1},\dots ,{\bar \gamma}_r$, and it meets
${\bar \gamma}_i$ only at its endpoint.
 
With $D_0$ a disc with center $z_0$ and disjoint from each of
the discs $D_1,\dots ,D_r$, consider the first point of
intersection of $\delta_i$ and the boundary ${\bar \gamma}_0$ of
$D_0$.  Call this point $a_i$.  Suppose $\delta_1,\dots
,\delta_r$ satisfy two further conditions:  
\begin{edesc} \item they are pairwise
nonintersecting, excluding their initial point $z_0$; and
\item $a_1,\dots, a_r$ appear in order clockwise around ${\bar
\gamma}_0$.  \end{edesc} Since the paths are simplicial, this last condition is
independent of the choice of $D_0$, at least for $D_0$ sufficiently
small. 

With these conditions, the ordered collection of closed
paths $\delta_i^\sph{\bar \gamma}_i^\sph\delta_i^{-1} = \gamma_i$,
$i = 1,\dots ,r$, in Fig.~\ref{Fig1-221b} are  {\sl classical generators\/}
(for $\bz$)  based at $z_0$.  We say $\gamma_i$ is a {\sl
classical loop around\/} $z_i$. 

\end{appendix}

\providecommand{\bysame}{\leavevmode\hbox to3em{\hrulefill}\thinspace}
\providecommand{\MR}{\relax\ifhmode\unskip\space\fi MR}
\providecommand{\MRhref}[2]{%
\href{http://www.ams.org/mathscinet-getitem?mr=#1}{#2}}
\providecommand{\href}[2]{#2}


\begin{thebibliography}{DeJaL93}


\bibitem[An98]{An98} Y.~Andr\'e, \emph{Finitude des couples d'invariants modulaires singuliers sur une courbe alg\'ebrique plane non modulaire}, Journal f\"ur die Reine und Angewandte Mathematik, \textbf{505} (1998), 203--208, MR 1662256 (2000a:11090). 

\bibitem[Be91]{Be91} D.J. Benson, \emph{I: Basic representation theory of finite groups and associative algebras},
Cambridge Studies in advanced math., vol. \textbf{30}, Cambridge U. Press, Cambridge, 1991.

\bibitem[Be72]{Be72} L. Bers, \emph{Uniformization, Moduli and KIeinian groups}, Bull. London Math. Soc.  \textbf{4} (1972) 250--300.


\bibitem[BFr02]{BFr02} P.~Bailey and M. D. Fried, \emph{Hurwitz monodromy, spin separation and higher levels of a Modular Tower}, Arithmetic fundamental groups and noncommutative algebra, PSPUM vol. \textbf{70} of AMS (2002), 79--220. arXiv:math.NT/0104289 v2 16 Jun 2005.

\bibitem[BiFr82]{BiFr82} R. Biggers and \bysame, \emph{Moduli spaces of covers and the Hurwitz monodromy group,} Crelle's Journal \textbf{335} (1982), 87--121.

\bibitem[Br82]{Br82} K.S.~Brown, \emph{Cohomology of groups}, Springer Graduate Texts \textbf{87}, Springer-Verlag New York, 1982. 

\bibitem[Ca56]{Ca56} R.D.~Carmichael, \emph{Groups of Finite Order}, Dover Publications, 1956. 

\bibitem[CaCo99]{CaCo99} M.~Couveignes and P.~Cassou-Nogues, \emph{Factorisations explicites de $g(y) \nm h(z)$}, Acta. Arith., Zakopane Conference on Number Theory, 1999, Proceedings of the Schinzel Festschrift, Summer 1997.

\bibitem[D87]{D87} P. D\`ebes, \emph{R\'esultats r\'ecents li\'es au th\'eor\`eme d'irr\'eductibili\'e de Hilbert.} S\'em. Th. Nombres Paris \textbf{85/86}, 
27\`eme ann\'ee, 1987, 19--37.
 
\bibitem[DEm06]{DEm06} \bysame and M.~Emsalem, \emph{Harbater-Mumford Components and Towers of Moduli Spaces.} 
 J.~Inst.~Math.~Jussieu, vol. \textbf{5}, no 3, 2006, 351--371.

\bibitem[DFr94]{DFr94} \bysame and M.D.~Fried, \emph{Nonrigid situations in constructive Galois theory}, Pacific Journal \textbf{163} (1994), 81--122.

\bibitem[DZ98]{DZ98} \bysame and U.~Zannier, \emph{Universal Hilbert subsets}. Math.~Proc.~Cambridge Phi. Soc., vol. \textbf{124}, 1998, 127--134. 
\bibitem[EFr80]{EFr80} J.L.~Ershov and M.~Fried, \emph{Frattini covers and projective groups without the extension property}, Mathematische Annalen \textbf{253} (1980), 233-239.

\bibitem[De72a]{De72a} P. Deligne, \emph{La conjecture de Weil pour les surfaces K3}, Inv.~math. \textbf{15} (1972) 206--226.
\bibitem[De72b]{De72b} \bysame, \emph{Les intersections compl\`etes de niveau de Hodge}, Inv.~math. \textbf{15} (1972)  237--250.
\bibitem[De74]{De74} \bysame, \emph{La Conjecture de Weil I}, IHES Pub.~Math. \textbf{43} (1974), 273--307.


\bibitem[Fr70]{Fr70} M.D.~Fried, \emph{On a Conjecture of Schur}, Michigan Math. J. Volume 17, Issue 1 (1970), 41--55
\bibitem[Fr73]{Fr73}, \bysame \emph{The field of definition of function fields and a problem in the reducibility of polynomials in two variables}, Illinois Journal of Math. \textbf{17}, (1973), 128--146.

\bibitem[Fr74]{Fr74} \bysame, \emph{Arithmetical properties of function fields.II}, Acta.~Arith. \textbf{25} (1974), 225--258.
\bibitem[Fr74b]{74b} \bysame, \emph{On Hilbert's irreducibility theorem}, Journal of No. Theory \textbf{6 }(1974), 211--232. 

\bibitem[Fr77]{Fr77}  \bysame, \emph{Fields of Definition of Function Fields and Hurwitz Families and; Groups as
Galois Groups}, Communications in Algebra \textbf{5} (1977), 17Ð82. 

\bibitem[Fr78]{Fr78}  \bysame,  \emph{Galois groups and Complex Multiplication}, Trans.A.M.S. \textbf{235} (1978), 141--162. 

\bibitem[Fr85]{Fr85} \bysame, \emph{On the Sprindzuk-Weissauer approach to universal Hilbert subsets}, Israel Journal of Mathematics \textbf{51} (1985), 341--363.
\bibitem[Fr90]{Fr90} \bysame, \emph{Arithmetic of 3 and 4 branch point covers: a bridge provided by noncongruence subgroups of $\SL_2(\bZ)$}, Progress in Math. Birkhauser 81 (1990), 77--117. Spring 1989, Seminar Delange-Pisot-Poiteu. 

\bibitem[Fr94]{Fr94}  \bysame, \emph{Topics in Galois Theory, review of J.-P. Serre, 1992}, Bartlett and Jones Publishers, BAMS \textbf{30} \#1 (1994), 124--135. ISBN 0-86720-210-6.  Also,   \emph{Enhanced review of J.-P. Serre's Topics in Galois Theory, with examples illustrating braid rigidity}, Recent Developments in the Galois Problem, Cont. Math., proceedings of AMS-NSF Summer Conference, Seattle \textbf{186} (1995), 15Ð-32. 
  
\bibitem[Fr95]{Fr95} \bysame, \emph{Introduction to Modular Towers: Generalizing dihedral groupÐmodular curve connections}, Recent Developments in the Inverse Galois Problem, Cont. Math., proceedings of AMS-NSF Summer Conference 1994, Seattle \textbf{186} (1995), 111--171.

\bibitem[Fr99]{Fr99} \bysame, \emph{Variables Separated Polynomials and Moduli Spaces}, No. Theory in Progress, eds. K. Gyory, H. Iwaniec, J. Urbanowicz, proceedings of the Schinzel Festschrift, Summer 1997 Zakopane, Walter de Gruyter, Berlin-New York (Feb. 1999), 169--228.

\bibitem[Fr05]{Fr05} \bysame, \emph{Relating two genus 0 problems of John Thompson}, Volume for John Thompson's 70th birthday, in Progress in Galois Theory, H. Voelklein and T. Shaska editors 2005 Springer Science, 51--85.

\bibitem[Fr05b]{Fr05b} \bysame, \emph{The place of exceptional covers among all diophantine relations}, Finite Fields and Their Applications 11 (2005) 367--433.

\bibitem[Fr06]{lum} \bysame, \emph{The Main Conjecture of Modular Towers and  its higher rank generalization}, in {\sl Groupes de Galois arithmetiques et differentiels\/}  (Luminy 2004; eds. D.~Bertrand and P.~D\`ebes),  Seminaires et Congres, Vol.~\textbf{13} (2006), 165--233.     

\bibitem[Fr10]{Fr10}  \bysame, \emph{Alternating groups and moduli space lifting Invariants}, Arxiv \textbf{\#0611591v4}. Israel J. Math. 179 (2010) 57--125 (DOI 10.1007/s11856- 010-0073-2).

\bibitem[Fr12]{Fr12} \bysame, \emph{Variables separated equations: Strikingly different roles for the Branch Cycle Lemma and the Finite Simple Group Classification:} arXiv:1012.5297v5 [math.NT] (DOI 10.1007/s11425-011-4324-4). Science China Mathematics, vol. \textbf{55}, January 2012, 1--72.

\bibitem[Fr20]{Fr20} \bysame, \emph{Moduli relations between $\ell$-adic representations and the Regular Inverse Galois Problem}, Issue 1, Vol.~\textbf{5} of Grad.~J. of Math. Issue 1 (2020), 38--75.
\bibitem[Fr26]{Fr26} \bysame, \emph{Heisenberg Quotients of Riemann Extensions: Applications to Coleman-Oort}, in preparation March 16, 2026. 
\bibitem[Fr26b]{Fr26b} \bysame,\emph{Monodromy, $\ell$-adic Representations \& the Regular Inverse Galois Problem}, book in preparation. 
\bibitem[Fr27]{Fr27} Alternating group odd-cycle Modular Towers; comparison with Liu-Osserman results.
\bibitem[FrJ86]{FrJ86} \bysame and M.~Jarden, \emph{Field arithmetic}, Ergebnisse der Mathematik editions I-IV, vol. \textbf{11}, Springer Verlag, Heidelberg, 1986 (455 pgs); 4th Edition 2023 (813 pgs) ISBN 3-540- 22811-x. We quote here the 4th ed., using [FrJ86]${}_4$.

\bibitem[FrK97]{FrK97} \bysame and Y.~Kopeliovic, \emph{Applying Modular Towers to the Inverse Galois Problem}, Geometric Galois Actions II Dessins d'Enfants, Mapping Class Groups and Moduli \textbf{243}, London Mathematical Society Lecture Note series, (1997) 172--197.
\bibitem[FrV91]{FrV91} \bysame and H.~V\"olklein, \emph{The inverse Galois problem and rational points on moduli spaces}, Math.~Ann.~290, (1991) 771--800.  

\bibitem[FrV92]{FrV92} \bysame and \bysame,  \emph{The embedding problem over an Hilbertian-PAC field}, Annals of Math \textbf{135}(1992), 469--481.

\bibitem[Fu69]{Fu69} W.~Fulton, \emph{Hurwitz Schemes and Irreducibility of Moduli of Algebraic Curves},  Annals of Math. (1969), Second Series, Vol. \textbf{90}, No. 3, 542--575. 

\bibitem[GhT24]{GhT24} A.~Ghigi and C.~Tamborini, \emph{A Topological Construction of Families of Galois Covers of the Line}, Algebraic and Geometric Topology, Vol. \textbf{24:3} , 1569--1600. 

\bibitem[GoH92]{GoH92} G.~Gonz\'alez-D\'iez and W.J.~Harvey, \emph{Moduli of Riemann surfaces with symmetry}, In Discrete
groups and geometry (Birmingham, 1991), volume \textbf{173} of London Math.~Soc.~Lecture Note
Ser,  75--93. Cambridge Univ. Press, Cambridge, 1992.

\bibitem[Gr-Re57]{Gr-Re57}  H.~Grauert and R.~Remmert, \emph{3 papers in Comptes Rendus}, Paris Band 245 (1957), 819-822, 822-825, 918-921. 

\bibitem[Gr61]{Gr61} A.~Grothendieck, \emph{S\'eminaire de g\'eometrie algebrique,} \textbf{1960-61}, IHES.  

\bibitem[GSM03]{GSM03} R. Guralnick, P. M\"uller, J. Saxl, \emph{The rational function analoque of a question of Schur and exceptionality of permutations representations}, Mem. Amer. Math. Soc. \textbf{162} (2003) 773 ISBN 0065-9266.

\bibitem[Ha77]{Ha77} R.~Hartshorne, \emph{Algebraic Geometry}, Graduate Texts in Math. \textbf{52}, Springer-Verlag, 1977.

\bibitem[Hi1892]{Hi1892} D. Hilbert, \emph{\"Uber die Irreduzibilit\"at ganzer rationaler Funktionenen mit ganzzahligen Koeffizienten}, J.~Reine Angew. Math \textbf{110} (1892) 194--229.

\bibitem[La71]{La71} S.~Lang, \emph{Algebra}, 2nd edition, Addison-Wesley, Reading 1971. 

\bibitem[LO08]{LO08} F.~Liu and B.~Osserman. \emph{The irreducibility of certain pure-cycle Hurwitz spaces}, Amer.~J.~Math., 2008 \textbf{130 (6)}, 1687--1708. 



\bibitem[Mu66]{Mu66}D. Mumford, \emph{Introduction to Algebraic Geometry}, Cambridge, Mass.: 
Harvard University Notes (1966).

\bibitem[Mu76]{Mu76} D.~Mumford, \emph{Curves and their Jacobians}, 2nd printing (1976), Ann Arbor Press.


\bibitem[S05]{S05} D.~Semmen, \emph{JenningsÕ theorem for $p$-split groups}, Journal of Algebra \textbf{285} (2005) 730--742. 

 \bibitem[Se68]{Se68} J.-P.~Serre, \emph{Abelian $\ell$-adic representations and elliptic curves}, New York, Benj. Publ., 1968. 

\bibitem[Se81]{Se-Cheb} J.-P. Serre, \emph{Quelques Applications du Th\'eor\`eme de
Densit\'e de Chebotarev}, Publ.~Math.~IHES
\textbf{54} (1981), 323--401.

\bibitem[Se88]{Se88} \bysame, \emph{Letter on Ext}, private correspondence. 

\bibitem[Se90]{Se90} \bysame, \emph{Rel\^evements dans $\tilde A_n$}, C. R. Acad.~Sci.~Paris 311 (1990), 477--482.

\bibitem[Se91]{Se91} \bysame, \emph{Galois Cohomology}, translation of the original Springer lN5, 1964 edition (notes by M.~Raynaud) from the French by Patrick Ion, Springer-Verlag (1991). 

\bibitem[Se92]{Se92} \bysame, \emph{Topics in Galois theory}, no. ISBN \#0-86720-210-6, Bartlett and Jones Publishers, notes taken by H. Darmon, 1992. 
\bibitem[Se97]{Se97} \bysame, \emph{Lectures on the Mordell-Weil Theorem}, Vieweg, 3rd edition, 1997.

\bibitem[Sh70]{Sh70} G.~Shimura, \emph{On Canonical models of arithmetic quotients of bounded symmetric domains}, Ann.~of Math., \textbf{91} (1970), 144-222.

\bibitem[Sh71]{Sh71} \bysame, \emph{Introduction to the Arithmetic Theory of Automorphic Functions}, Pub.~Math.~Soc. of Japan \textbf{11}, Princeton U. Press, 1971. 
\bibitem[ShT61]{ShT61} \bysame and Y.~Taniyama, \emph{Complex Multiplication of Abelian Varieties and its applications to number theory}, Publ.~Math.~Soc. Japan, \textbf{6}, 1961. 

\bibitem[tor]{tor} The Torelli map, \emph{https://mathoverflow.net/questions/8938/is-the-torelli-map-an-immersion}. 
\bibitem[V96]{V96} \bysame, \emph{Groups as Galois Groups}, Cambridge Studies in Advanced Mathematics \textbf{53} (1996).

\bibitem[We05]{We05} T.~Weigel, \emph{Maximal $p$-frattini quotients of $p$-poincare duality groups of dimension 2}, volume for O.H. Kegel on his 70th birthday, Arkiv der Mathematik--Basel (2005).

\bibitem[Wm73]{Wm73} A.~Williamson, \emph{On primitive permutation groups containing a cycle}, Math.~Zeit. \textbf{130} 
(1973), 159--162.
\bibitem[Woh64]{Woh64} K.~Wohlfahrt, \emph{An extension of F. Klein's level concept}, Illinois J. Math. 8 (1964),
529--535. MR MR0167533 (29 \#4805)
\end{thebibliography}
\end{document}